\documentclass[final,oneside,notitlepage,11pt]{article}         % Dokumentklasse

 \usepackage[T1]{fontenc}              % wichtig
 \usepackage[utf8]{inputenc}                          
 \usepackage{amsmath, amssymb}          % Mathesymbole
\usepackage{hyperref}  
\usepackage[left=3cm,right=3cm,top=3cm]{geometry}
\usepackage{amsthm} 
 \usepackage{icomma}                   
 \usepackage{enumerate}                
 \usepackage[dvipsnames]{xcolor}                    
 \usepackage{setspace}                  
 \usepackage{needspace}                
 \usepackage{url}                      
 \usepackage{mathrsfs}                 
 \usepackage{graphicx}
\usepackage{cancel}   
\usepackage[all, cmtip]{xy}
\usepackage[capitalise]{cleveref}
\usepackage{soul}
\usepackage{euscript}

 \usepackage{tikz}                     
 \usepackage{slashed}
 \usepackage{tikz-cd}                     
 \usepackage{verbatim}

\usepackage{todonotes}

\crefname{equation}{}{}

\def\tocsection#1{\section*{#1}\addcontentsline{toc}{section}{#1}}

\newcounter{comments}
\newenvironment{displaycomment}{\begin{list}{}{\rightmargin=1cm\leftmargin=1cm}\item\sf\begin{small}}{\end{small}\end{list}}

 \numberwithin{equation}{section}

\usepackage[
  backend=biber,
  style=numeric,
  doi=false,
  url=false,
  isbn=false,  
  eprint=true,
  sortcites=true
]{biblatex}
\hypersetup{
    linktocpage = true,
    colorlinks,
    linkcolor=[RGB]{0,0,120},
    citecolor=[RGB]{0,0,120},
    urlcolor=[RGB]{0,0,120}
}

\def\pdfdaten{
  \hypersetup{
    pdftitle = {\@title},
    pdfauthor = {\@author},
    pdfkeywords = {\@keywords},    
    bookmarksopen = true,
    bookmarksopenlevel = 1
  }}   
 
\usepackage{changepage}

\usepackage[normalem]{ulem}

\usepackage{amsthm}

\numberwithin{equation}{section}

\theoremstyle{plain}  % Italic body, bold header
\newtheorem{theorem}{Theorem}[section]
\newtheorem{proposition}[theorem]{Proposition}
\newtheorem{lemma}[theorem]{Lemma}
\newtheorem{corollary}[theorem]{Corollary}

\theoremstyle{definition}  % Normal font, bold header
\newtheorem{definition}[theorem]{Definition}
\newtheorem{example}[theorem]{Example}

\theoremstyle{remark}  % Normal font, italic header
\newtheorem{remark}[theorem]{Remark}

\DeclareMathOperator*{\Der}{\mathfrak{Der}}

\newcommand{\AUT}{\mathrm{AUT}}
\newcommand{\Ad}{\mathrm{Ad}}
\newcommand{\ad}{\mathrm{ad}}
\newcommand{\R}{\mathbb{R}}
\newcommand{\N}{\mathbb{N}}

\newcommand{\Z}{\mathbb{Z}}
\newcommand{\C}{\mathbb{C}}

\newcommand{\tr}{\mathrm{tr}}

\newcommand{\Out}{\mathrm{Out}}

\newcommand{\Lin}{\mathrm{Lin}}

\newcommand{\id}{\mathrm{id}}

\newcommand{\Hom}{\mathrm{Hom}}

\newcommand{\Aut}{\mathrm{Aut}}

\newcommand{\Spin}{\mathrm{Spin}}
\newcommand{\String}{\mathrm{String}}

\newcommand{\Alt}{\mathrm{Alt}}
\newcommand{\Sym}{\mathrm{Sym}}

\newcommand{\grpd}{\text{\textsc{grpd}}}

\newcommand{\U}{\mathrm{U}}

\newcommand{\anti}{\mathrm{a}}
\newcommand{\sym}{\mathrm{s}}
\newcommand{\lact}{\triangleright}
\newcommand{\ract}{\triangleleft}

\def\KL{\mathrm{KL}}
\def\Adj{\mathrm{\mathrm{Adj}}}

\def\T{\mathbb{T}}

\def\ev{\mathrm{ev}}
\def\quot#1{``#1''}
\def\ADJ{\EuScript{A}\mathrm{dj}}
\def\cw{\mathrm{cw}}
\def\CM{\EuScript{C}\mathrm{r}\EuScript{M}\mathrm{od}}
\def\cm{\mathfrak{CrMod}}
\def\BUT{\EuScript{B}\mathrm{ut}}
\def\but{\mathfrak{But}}
\def\grpd{\EuScript{G}\mathrm{rpd}}
\def\AUT{\EuScript{A}\mathrm{ut}}

\title{Classification of adjustments on central crossed modules}
\author{ Matthias Ludewig \& Konrad Waldorf} %Universit\"atsstra{\ss}e 31, 93053 Regensburg}
\date{}

\usepackage{color}
\usepackage{amstext} 

\begin{document} 

\maketitle

\begin{abstract}
Adjustments are additional structures on crossed modules of Lie groups, serving as a tool in higher gauge theory to circumvent the fake flatness of connections on 2-bundles. In this article, we investigate the existence and classification of adjustments, as well as their covariance under weak equivalences. Our  approach is based on a differentiation/integration correspondence with an infinitesimal version of adjustments on the associated crossed module of Lie algebras, which we then study using Lie algebra techniques. Our main result is that infinitesimal adjustments exist if and only if the Kassel-Loday class of the crossed module lies in the image of the (Lie algebraic) Chern-Weil homomorphism. 
\end{abstract}

%%%%%%%%%%%%%%%%%%%%%%%%%%%%%%%%%%%%%%%%%%%%%%%%%%%%%%%%%%%%%%%%

\setlength{\parskip}{0ex}
\tableofcontents
\setlength{\parskip}{1ex}
 
\sloppy 

\section{Introduction}

Crossed modules of Lie groups provide a  convenient model for Lie 2-groups, which are the analogues of Lie groups in categorified differential geometry. Lie 2-groups  naturally arise in theoretical physics - first in string theory and, more recently, in condensed matter physics. 
Pushing the analogy with ordinary gauge theory as far as considering  connections on categorified principal bundles reveals a crucial insight: categorification introduces a novel feature absent in ordinary gauge theory. Namely, in order to define consistent notions of parallel transport along surfaces -- and holonomy around them -- a condition called \quot{fake-flatness} must be imposed \cite{schreiber2, Schreiber2011, Schreiber2016}. 
However, in other contexts, fake-flatness becomes an undesirable constraint \cite{Fiorenza2012,Sati2009,Sati2012, Sati2009, Rist2022, Kim, Saemann2020, Kim2020}.

Avoiding fake-flatness requires an \quot{adjustment} of the  theory. This was noticed by Sati, Schreiber, and Stasheff \cite{Sati2009,Sati2012}, who introduced  \quot{transgression elements} for $L_{\infty}$-algebras of string extension type. Fiorenza, Schreiber, and Stasheff later revisited this idea in \cite{Fiorenza2012}, referring to them as \quot{Chern-Simons elements}. Sämann, together with Kim and Schmidt, restricted the framework to Lie 2-algebras and simultaneously extended it to Lie 2-\emph{groups} \cite{Saemann2020, Kim2020}, using the term \emph{adjustment}. Recently, Tellez-Dominguez \cite{TellezDominguez2023} proposed a more specific and simplified definition, which nevertheless encompasses most known examples.

Crossed modules may or may not allow adjustments, and when they do, there may be different choices.  The goal of this article is to explore the existence and the classification of adjustments on a subclass of crossed modules of Lie groups, so called \emph{central} ones.

In order to delve into  some details, let  $\Gamma =(H \stackrel{t}{\to} G \stackrel{\alpha}{\to} \Aut(H))$  be a central crossed module of Lie groups, with (possibly infinite-dimensional) Lie groups $G$ and $H$, and corresponding Lie algebras $\mathfrak{g}$ and $\mathfrak{h}$, respectively. An adjustment on $\Gamma$ is a map
\begin{equation*}
\kappa: G \times \mathfrak{g} \to \mathfrak{h}
\end{equation*} 
satisfying a number of conditions, including a non-linear one that does not allow $\kappa$ to be identically zero;  see \cref{definition-adjustment}.

Our main tool for studying adjustments on a crossed module of Lie groups $\Gamma$ is the induced crossed module $\mathfrak{G}=(\mathfrak{h} \stackrel{t_{*}}\to \mathfrak{g} \stackrel{\alpha_{*}}\to \mathfrak{der}(\mathfrak{h}))$ of Lie algebras, obtained by differentiating all structure of $\Gamma$, as well as a corresponding notion of infinitesimal adjustment
\begin{equation*}
\kappa_*: \mathfrak{g} \times \mathfrak{g} \to \mathfrak{h}
\end{equation*} 
obtained by differentiating $\kappa$ in its first argument, see \cref{DefInfAdj}. Thus, we study the sets $\Adj(\Gamma)$ of all adjustments on $\Gamma$ and $\Adj(\mathfrak{G})$ of all infinitesimal adjustments on $\mathfrak{G}$. 

Moreover, we consider the \quot{homotopy} Lie algebras \[
\mathfrak{a} := \mathrm{ker}(t_{*})\subset \mathfrak{h}, \qquad \text{and} \qquad \mathfrak{f} := \mathfrak{g}/t_{*}\mathfrak{h}
\]
 of $\mathfrak{G}$. 
A \emph{section} of $\mathfrak{G}$ is a linear section $s:\mathfrak{f} \to \mathfrak{g}$  against the projection $\mathfrak{g} \to \mathfrak{f}$. 
Such a section selects   subsets $\Adj^{s}(\Gamma) \subset \Adj(\Gamma)$ and  $\Adj^{s}(\mathfrak{G}) \subset \Adj(\mathfrak{G})$ of \emph{adapted} (infinitesimal) adjustments which turned out to be  important    \cite{TellezDominguez2023}. 
Our first main result relates (adapted) adjustments on $\Gamma$ to (adapted) infinitesimal adjustments on $\mathfrak{G}$.

\begin{theorem}
\label{theorem-1}
Let   $\Gamma =(H \stackrel{t}{\to} G \stackrel{\alpha}{\to} \Aut(H))$   be a central crossed module of Lie groups, let $\mathfrak{G}$ be the corresponding crossed module of Lie algebras, and let $s$ be any section of $\mathfrak{G}$. Differentiation constitutes maps
\begin{equation*}
\Adj(\Gamma) \to \Adj(\mathfrak{G})
\quad\text{ and }\quad
\Adj^{s}(\Gamma) \to \Adj^{s}(\mathfrak{G})\text{, }
\end{equation*}
which are injective when $G$ is connected, and bijections when $G$ is connected and simply-connected and $H$ is connected.
\end{theorem}

\cref{theorem-1} is stated in the main text as \cref{integration-of-adjustments}.
It allows us to reduce the classification of (adapted) adjustments on $\Gamma$ to the classification of (adapted) infinitesimal adjustments on $\mathfrak{G}$, at least under the connectedness assumptions stated in \cref{theorem-1}.

Our next result provides a complete classification of infinitesimal adjustments. 
For preparation, we recall that any central crossed module $\mathfrak{G}$ of Lie algebras has a  \emph{Kassel-Loday class} $\mathrm{KL}(\mathfrak{G}) \in H^3(\mathfrak{f}, \mathfrak{a})$, which classifies crossed modules of Lie algebras (with fixed homotopy Lie algebras $\mathfrak{f}$ and $\mathfrak{a}$) up to weak equivalence  \cite{Kassel1982}. 
Moreover, we let $T(\mathfrak{f}, V)$ be the vector space of continuous bilinear forms $\eta:\mathfrak{f} \times \mathfrak{f} \to V$ satisfying the condition
\begin{equation*}
\eta([X, Y], Z) + \eta(Y, [X, Z]) = \eta(X, [Y, Z])
\end{equation*}
for all $X, Y, Z\in \mathfrak{f}$.
Finally, we denote by $\Sym^2(\mathfrak{f}, \mathfrak{a})^{\mathrm{ad}} $ the space of symmetric $\mathrm{Ad}$-invariant bilinear forms on $\mathfrak{f}$ with values in $\mathfrak{a}$.

\begin{theorem}
\label{theorem-2}
Let $\mathfrak{G}$ be a central crossed module of Lie algebras,  with a section  $s$. 
\begin{enumerate}
\item 
The following are equivalent:
\begin{enumerate}[(i)]

\item 
$\mathfrak{G}$ admits infinitesimal adjustments.

\item
$\mathfrak{G}$ admits infinitesimal adjustments adapted to  $s$.

\item
The Kassel-Loday class $\KL(\mathfrak{G})$ of $\mathfrak{G}$ lies in the image of the Chern-Weil homomorphism
\begin{equation*}
\cw: \Sym^2(\mathfrak{f}, \mathfrak{a})^{\mathrm{ad}} \to H^3(\mathfrak{f}, \mathfrak{a}).
\end{equation*} 

\end{enumerate}

\item
$\Adj(\mathfrak{G})$ is an affine space over $T(\mathfrak{f}, \mathfrak{h})$, and $\Adj^{s}(\mathfrak{G})$ is an affine space over $T(\mathfrak{f}, \mathfrak{a})$.

\item
There is a canonical map
\begin{equation}
\label{CanonicalMap}
\KL^{\mathrm{adj}} : \Adj^{s}(\mathfrak{G}) \to \Sym^2(\mathfrak{f}, \mathfrak{a})^{\mathrm{ad}},
\end{equation}
called the \emph{adjusted Kassel-Loday class}, that assigns to each adapted adjustment a preimage of the Kassel-Loday class under the Chern-Weil homomorphism:
\[
\bigl[\cw\bigl(\KL^{\mathrm{adj}}(\mathfrak{G},\eta)\bigr)\bigr] = \KL(\mathfrak{G}). 
\]

The fibre of the map \eqref{CanonicalMap} over $B\in \Sym^2(\mathfrak{f}, \mathfrak{a})^{\mathrm{ad}}$ is an affine space for $\Alt^2_{\mathrm{cl}}(\mathfrak{f}, \mathfrak{a}) \subset T(\mathfrak{f}, \mathfrak{a})$ if $[\cw (B)]=\KL(\mathfrak{G})$, and it is empty else.

\item
If $\mathfrak{G}$ is weakly equivalent to another crossed module $\mathfrak{G}'$, and $s'$ is a section in $\mathfrak{G}'$, then there exists a bijection 
\begin{equation*}
\Adj^{s}(\mathfrak{G}) \cong   \Adj^{s'}(\mathfrak{G}')\text{.}
\end{equation*}

\end{enumerate}
\end{theorem}

This is obtained in the main text as a combination of \cref{existence-of-infinitesimal-adjustments,affine-space-structure-on-adjustments,bijectivity-adjustment-transfer}.
We remark that \cref{theorem-2} provides a new way to compute the Kassel-Loday class of (certain) crossed modules: choose an adapted adjustment $\eta$ and compute the image of its adjusted Kassel-Loday class under the Chern-Weil homomorphism.  
This is conceptually similar to how the Chern classes of a vector bundle may be computed using Chern-Weil theory.

The classification of sets of adjustments provided by \cref{theorem-2} can be improved by considering the \emph{groupoid} $\ADJ(\Gamma)$ of adjustments, where the objects are pairs $(s, \kappa)$ of a section $s$ and an adjustment $\kappa$ adapted to $s$, and, analogously, the groupoid $\ADJ(\mathfrak{G})$ of infinitesimal adjustments, with pairs $(s, \eta)$ of a section $s$ and an infinitesimal adjustment adapted to $s$. This organization in groupoids follows Tellez-Dominguez \cite{TellezDominguez2023}. 
Our previous results find the following reformulation, which are stated in the main text in \cref{homotopy-groups-of-ADJ}.

\begin{theorem}
\label{theorem-3}
Let $\Gamma =(H \stackrel{t}{\to} G \stackrel{\alpha}{\to} \Aut(H))$ be a central crossed module of Lie groups, and $\mathfrak{G}$ be the corresponding crossed module of Lie algebras, with homotopy Lie algebras $\mathfrak{f}$ and $\mathfrak{a}$.
\begin{enumerate}

\item 
Differentiation is a faithful functor $\ADJ(\Gamma) \to \ADJ(\mathfrak{G})$; it is full if $G$ is connected, and it is essentially surjective if $G$ is connected and simply connected  and   $H$ is connected. 

\item
The adjusted Kassel-Loday class induces a well-defined
 map
\begin{equation*}
\pi_0\ADJ(\mathfrak{G}) \to \Sym^2(\mathfrak{f}, \mathfrak{a})^{\mathrm{ad}},\quad (s, \eta) \mapsto \KL^{\mathrm{adj}}(\mathfrak{G}, \eta)
\end{equation*}
on the set of isomorphism classes of objects, 
whose fibre over $B\in \Sym^2(\mathfrak{f}, \mathfrak{a})^{\mathrm{ad}}$ is an affine space over $H^2(\mathfrak{f}, \mathfrak{a})$ if $\cw (B)=\KL(\mathfrak{G})$ and  empty else.  

\item
The group $\pi_1\ADJ(\mathfrak{G})$ of automorphisms (of any object) is 
\begin{equation*}
\pi_1\ADJ(\mathfrak{G})\cong H^1(\mathfrak{f}, \mathfrak{a})\text{.}
\end{equation*}

\end{enumerate}
Moreover, the assignment of the groupoid $\ADJ(\mathfrak{G})$ of adjustments to a crossed module $\mathfrak{G}$ of Lie algebras extends to a strict 2-functor
\begin{equation*}
\ADJ: \cm^{\mathrm{sec}} \to \grpd\text{.}
\end{equation*}
In particular, weakly equivalent crossed modules have equivalent groupoids of adjustments. 
\end{theorem}

Above,  $\cm^{\mathrm{sec}}$ denotes the bicategory of central crossed modules of Lie algebras and weak equivalences (realized as butterflies with a selected section), and $\grpd$ denotes the bicategory of groupoids.
Performing the Grothendieck construction with the functor $\ADJ: \cm^{\mathrm{sec}} \to \grpd$ provides systematically  a bicategory $\cm^{\mathrm{adj}}$ of central crossed modules equipped with adjustments, and adjustment-preserving butterflies. 
We denote by 
\[
\cm^{\mathrm{adj}}(\mathfrak{f}, \mathfrak{a}) \subset \cm^{\mathrm{adj}}\qquad \text{ and} \qquad \cm(\mathfrak{f}, \mathfrak{a})\subset \cm
\]
 the full subcategories of crossed modules with fixed homotopy Lie algebras $\mathfrak{f}$ and $\mathfrak{a}$.
Our final result classifies these bicategories.

\begin{theorem}
\label{Thm1-4}
The adjusted Kassel-Loday class induces a well-defined bijective map $\mathrm{KL}^{\mathrm{adj}}$ in the top row of a commutative diagram
\begin{equation*}
\xymatrix@C=3em{\pi_0\cm^{\mathrm{adj}}(\mathfrak{f}, \mathfrak{a}) \ar[d] \ar[r]^-{\KL^{\mathrm{adj}}} & \Sym^2(\mathfrak{f}, \mathfrak{a})^{\ad} \ar[d]^{\mathrm{cw}} \\ \pi_0\cm(\mathfrak{f}, \mathfrak{a}) \ar[r]_-{\KL} & H^3(\mathfrak{f}, \mathfrak{a})\text{.} }
\end{equation*}
In particular, adjusted crossed modules are classified up to weak equivalence by their adjusted Kassel-Loday class. 
\end{theorem}

\noindent
This is proved in the main text as \cref{adjusted-KL}.  
Here, the bottom horizontal map is the previously mentioned classification of crossed modules of Lie algebras by their Kassel-Loday class \cite{Kassel1982}.

For each pair $\mathfrak{f}$, $\mathfrak{a}$ of finite-dimensional Lie algebras with $\mathfrak{a}$ abelian, and each $B \in \Sym^2(\mathfrak{f}, \mathfrak{a})^{\ad}$, we give in \cref{SectionConstructionCrossedModule1} a canonical construction of a crossed module $\mathfrak{G}_B$ of Lie algebras together with an  infinitesimal adjustment $\eta_{B,s}$ on $\mathfrak{G}_B$, adapted to any section $s$, such that $\KL^{\mathrm{adj}}(\mathfrak{G}_B, \eta_{B,s}) = B$.
In \cref{SectionConstructionCrossedModule2} we lift this construction to a crossed module $\Gamma_B$ of Lie groups and an adjustment on $\Gamma_B$. 
This construction is a generalization of the construction of the string group in \cite{LudewigWaldorf2Group}.
In particular,  the string group admits an adjustment, and our classification shows that this adjustment is unique up to unique isomorphism (see \cref{string-2-group}).

As further examples, we discuss adjustments on product crossed modules (\cref{Products}), on categorical tori (\cref{ex:TD}) and on automorphism 2-groups of algebras (\cref{ex:AUTA}).

\paragraph{Acknowledgements.}

ML acknowledges support of SFB 1085 ``Higher invariants'', funded by  the German Research Foundation (DFG). KW was supported  by the DFG under project code WA 3300/5-1.

\section{Adjustments}

We define adjustments on crossed modules of Lie groups and infinitesimal adjustments on crossed modules of Lie algebras, and explore the relation between them under differentiation. 

\subsection{Adjustments on central crossed modules of Lie groups}

Let $\Gamma=(H \stackrel{t}{\to} G \stackrel{\alpha}{\to} \Aut(H))$ be a  crossed module of (possibly infinite-dimensional) Lie groups. 
We assume it to be \emph{central}, meaning that the action $\alpha$ is trivial on $A :=\pi_1(\Gamma):= \ker(t) \subseteq H$.
We also assume it to be \emph{smoothly separable}.
This means that  the quotients $F :=\pi_0(\Gamma):= G/t(H)$ and $H/A$ have  Lie group structures, such that the projections $p:G \to F$ and $H \to H/A$  have smooth local sections, and the map $t: H/A \to t(H)$ is a diffeomorphism \cite[Def. II.1, Def. III.1]{Neeb2005}.
In particular, $\mathfrak{a} \subseteq \mathfrak{h}$ is topologically complemented.
Smooth separability is always satisfied when the groups involved are finite-dimensional. 
A smoothly separable crossed module can  be seen as a   categorical Lie group extension 
\begin{equation*}
1\to BA \to \Gamma \to F_{dis}\to 1\text{,}
\end{equation*}
and centrality means that this extension is central.  

For each $h \in H$, we define a smooth map $\tilde{\alpha}_h : G \to H$ by the formula
\begin{equation}
\label{alpha-tilde}
\tilde{\alpha}_h(g) := h^{-1} \alpha_g(h);
\end{equation}
this map appears in the definition of adjustments (see below).

A    crossed module $\mathfrak{G}=(\mathfrak{h} \stackrel{t_*}{\to} \mathfrak{g} \stackrel{\alpha_*}{\to} \Der(\mathfrak{h)})$ of  Lie algebras consists of Lie algebras  $\mathfrak{g}$ and $\mathfrak{h}$, a Lie algebra homomorphism $t_*:\mathfrak{h} \to \mathfrak{g}$, and an action $\alpha_*$ of   $\mathfrak{g}$ on $\mathfrak{h}$ by derivations such that  $t$ is $\mathfrak{g}$-equivariant with respect to the adjoint action of $\mathfrak{g}$ on itself, and  the infinitesimal version 
\begin{equation}
\label{InfinitesimalPeiffer}
\alpha_*(t_*(x), y) = [x, y]
\end{equation}
of the Peiffer identity holds  \cite[Def. A.1]{Kassel1982}. 
All Lie algebras may be infinite-dimensional, and are assumed to carry locally convex topologies such that the Lie brackets and all other structure maps are continuous.
A crossed module of Lie algebras is called  \emph{central} if the action $\alpha$ of $\mathfrak{g}$ on $\mathfrak{h}$ restricts to the trivial action on $\pi_1(\mathfrak{G})=\mathfrak{a}:=\mathrm{ker}(t_{*})$. 
We also form the quotient Lie algebra $\pi_0(\mathfrak{G})=\mathfrak{f} := \mathfrak{g}/t_*(\mathfrak{h})$, and we say that $\mathfrak{G}$ is smoothly separable if $\mathfrak{a} \subseteq \mathfrak{h}$ and $t_* \mathfrak{h} \subseteq \mathfrak{g}$ are complemented subspaces.
In particular, the projection $\mathfrak{g} \to \mathfrak{f}$ admits a section $s: \mathfrak{f} \to \mathfrak{g}$.

Applying the Lie algebra functor to a (central, smoothly separable) crossed module $\Gamma$ of Lie groups yields a (central, smoothly separable) crossed module of Lie algebras, in such a way that $\mathfrak{a}$ and  $\mathfrak{f}$ are the Lie algebras of $A$ and $F$, respectively. 
In the following, we assume that all crossed modules of Lie groups and Lie algebras are smoothly separable, without explicit mentioning.

Associated to a crossed module  $\mathfrak{G}$ of Lie algebras is the following four-term exact sequence of Lie algebras, 
\begin{equation}
\label{four-term-sequence}
\begin{tikzcd}
0 \ar[r] & \mathfrak{a} \ar[r, "\iota"] & \mathfrak{h} \ar[r, "t"] & \mathfrak{g} \ar[r, "p"] & \mathfrak{f} \ar[r] & 0.
\end{tikzcd}
\end{equation}
A \emph{section} of $\mathfrak{G}$ is a linear map $s: \mathfrak{f} \to \mathfrak{g}$ such that $p s=\id_{\mathfrak{f}}$. 
This is the same information as the vector space complement $s(\mathfrak{f})$ of $t(\mathfrak{h})\subset \mathfrak{g}$ and thus the same information as the idempotent projection 
\begin{equation}
\label{RhoForS}
\rho_s :=\id_{\mathfrak{g}}-s p: \mathfrak{g} \to \mathfrak{g}
\end{equation}
onto $t(\mathfrak{h})$.

As described in the introduction, the notion of an adjustment underwent several developments \cite{Fiorenza2012,Sati2009,Sati2012,Saemann2020, Kim2020, Rist2022, Kim, TellezDominguez2023}, from which we synthesized the following definition (suitable for central crossed modules).

\begin{definition}
\label{definition-adjustment}
Let   $\Gamma=(H \stackrel{t}{\to} G \stackrel{\alpha}{\to} \Aut(H))$ be a central crossed module of Lie groups with associated crossed module $\smash{\mathfrak{G}=(\mathfrak{h} \stackrel{t_{*}}\to\mathfrak{g} \stackrel{\alpha_{*}}\to \Der(\mathfrak{h}))}$ of Lie algebras. An \emph{adjustment} of $\Gamma$  is a map
\[
\kappa : G \times \mathfrak{g} \longrightarrow \mathfrak{h}
\]
that is linear and continuous in $\mathfrak{g}$, smooth in $G$, 
and satisfies the following conditions:
\begin{align}
\label{adjustment-condition-1}
\kappa(g_1g_2, X) &= \kappa\big(g_1, \Ad_{g_2}(X)\big) + \kappa(g_2, X) 
\\
\label{adjustment-condition-2}
\kappa\big(t(h), X\big) &= (\tilde{\alpha}_{h^{-1}})_* (X)
\\
\label{adjustment-condition-3}
\kappa(g, t_{*}x) &= \alpha_g(x) - x
\end{align}
for all $g, g_1, g_2\in G$, $h\in H$, $X\in \mathfrak{g}$, and $x\in \mathfrak{h}$.
We say that $\kappa$ is \emph{adapted} to a section $s$ of $\mathfrak{G}$ if it additionally satisfies
\[
t_* \kappa(g, X) = \rho_s\big(\Ad_g(X) - X\big) 
\]
for all $g\in G$ and $X\in \mathfrak{g}$.
Here $\rho_s$ is the idempotent \eqref{RhoForS} associated to the section $s$.
\end{definition}

We denote by $\Adj(\Gamma)$ the set of adjustments on $\Gamma$, and by $\Adj^{s}(\Gamma)$ the subset of adjustments that are adapted to a given section $s$. In \cite{Rist2022}, adapted adjustments are called \emph{special}.

\begin{remark}
We restrict our attention  to \emph{central} crossed modules, because non-central crossed modules $\Gamma$ do not admit adjustments in the above sense (at least when $G$ is  connected). 
To see this, assume that $\Gamma$ admits an adjustment $\kappa$, and suppose $h, h'\in H$ with $t(h)=t(h')$. 
Thus, there exists $a\in A$ such that $h'=ha$. 
Then, 
\begin{align*}
-\mathrm{Ad}_{h}((\tilde\alpha_{h})_{*}(v))
&=\kappa(t(h)^{-1}, v) 
%\\&
=\kappa(t(h')^{-1}, v) 
%\\&
=-\mathrm{Ad}_{h'}((\tilde\alpha_{h'})_{*}(v))
%\\&
=-\mathrm{Ad}_{ha}((\tilde\alpha_{ha})_{*}(v)).
\end{align*}
We recall that $A\subset H$ is always central, and note that 
\begin{equation*}
(\tilde\alpha_{ha})_{*}=\mathrm{   Ad}_{a}^{-1} \circ (\tilde\alpha_{h} )_{*}+(\tilde\alpha_{a})_{*}= (\tilde\alpha_{h} )_{*}+(\tilde\alpha_{a})_{*}\text{.}
\end{equation*}
This implies 
\begin{equation}
\label{necessary-condition-for-adjustments}
(\tilde\alpha_{a})_{*}(v)=0
\end{equation}
for all $v\in \mathfrak{g}$ and all $a\in \pi_1\Gamma$. 
Condition \eqref{necessary-condition-for-adjustments} is a necessary condition for the existence of adjustments, and it is satisfied when $\Gamma$ is central.  
Conversely, we assume that \eqref{necessary-condition-for-adjustments} holds and  consider for each $a\in A$ the map $G \to H: g \mapsto \tilde\alpha_a(g)$. 
\eqref{necessary-condition-for-adjustments} implies that it is locally constant, and hence constant on the identity component $G_0$. Its value at $g=1$ is $1$, and so $\alpha(g, a)=a$ for all $g\in G_0$ and $a\in A$. 
If $G$ is connected, this shows that $\Gamma$ is central. 
\end{remark}

\subsection{Infinitesimal adjustments and their integration}
\label{SectionInfinitesimalAdjustments}

We consider a central crossed module $\mathfrak{G}=(\mathfrak{h} \stackrel{t_*}{\to} \mathfrak{g} \stackrel{\alpha_*}{\to} \Der(\mathfrak{h}))$ of Lie algebras.

\begin{definition}
\label{DefInfAdj}
An \emph{infinitesimal adjustment} for $\mathfrak{G}$ is a continuous bilinear map 
\begin{equation*}
\eta:\mathfrak{g} \times \mathfrak{g} \to  \mathfrak{h}
\end{equation*}
satisfying the  conditions
\begin{align}
\label{EtaIdentities-0}
\eta([X, Y], Z) + \eta(Y, [X, Z]) &= \eta(X, [Y, Z])
\\
\label{EtaIdentities-1}
\eta(t_*x, Y) &= -\alpha_*(Y, x)
\\
\label{EtaIdentities-2}
\eta(X, t_*y) &= \alpha_*(X, y)
\end{align}
for all $X, Y, Z \in \mathfrak{g}$ and all $x, y\in \mathfrak{h}$.
If $s$ is a section of $\mathfrak{G}$, an infinitesimal adjustment $\eta$ is called \emph{adapted to $s$} if it satisfies
\begin{equation}
\label{AdaptedEquationKappastar}
  t_* \eta(X, Y) = \rho_s([X, Y]), 
\end{equation}
where $\rho_s$ is the idempotent \eqref{RhoForS} associated to the section $s$.
\end{definition}

We denote by $\Adj(\mathfrak{G})$ the  set of all infinitesimal adjustments on $\mathfrak{G}$, and by $\Adj^{s}(\mathfrak{G})$ the set of all infinitesimal adjustments on $\mathfrak{G}$ that are adapted to a section $s$.

%\begin{remark}
%Clearly, the notion of being $u$-adapted only depends on the idempotent $t_*u$ and not $u$ itself.
%In particular, if $u$ and $u'$ are two generalized splittings such that $u - u'$ takes values in $\mathfrak{a}$, then any $u$-adapted adjustment %is also $u'$-adapted. 
%Moreover, if $u$ is any generalized splitting, then $u' := u t_* u$ is a full splitting and the difference $u - u'$ takes values in $\mathfrak{a}$.
%Hence any adjustment that is adapted to a generalized splitting $u$ is also adapted to the full splitting $u' = ut_*u$.
%\end{remark}

\begin{lemma}
Let   $\Gamma=(H \stackrel{t}{\to} G \stackrel{\alpha}{\to} \Aut(H))$ be a central crossed module of Lie groups with associated crossed module of Lie algebras $\mathfrak{G}$.
 For any adjustment $\kappa$ on $\Gamma$, the bilinear map
\[
  \kappa_* : \mathfrak{g} \times \mathfrak{g} \to \mathfrak{h}
\]
obtained by differentiating the first entry of $\kappa$ at the unit element of $G$ is an infinitesimal adjustment on $\mathfrak{G}$. 
Moreover, if $\kappa$ is adapted to a section $s$, then $\kappa_*$ is also adapted to $s$.
\end{lemma}

\begin{proof}
Conditions \eqref{EtaIdentities-1} \& \eqref{EtaIdentities-2} are obtained from differentiating \eqref{adjustment-condition-2} \& \eqref{adjustment-condition-3}. 
To obtain condition \eqref{EtaIdentities-0}, first observe that applying the cocycle condition for $\kappa$ twice yields the identity 
\begin{equation}
\label{ConjugationIdentity}     
\begin{aligned}
\kappa\big(g g' g^{-1} , \Ad_g(Y)\big) 
&= \kappa(g g', Y) + \kappa\big(g^{-1}, \Ad_g(Y)\big)
\\
&= \kappa\big(g, \Ad_{g'}(Y)\big) + \kappa(g', Y) + \kappa\big(g^{-1}, \Ad_g(Y)\big).
\end{aligned}
\end{equation}
Inserting $g' = e^{tX}$ and
differentiating
at $t=0$ yields
\begin{equation*}
\kappa_*\big(\Ad_g(X), \Ad_g(Y)\big) = \kappa\big(g, [X, Y]\big) + \kappa_*(X, Y) .
\end{equation*}  
Finally, setting $g = e^{tZ}$ and differentiating at $t=0$, we obtain that $\kappa_*$ satisfies  \eqref{EtaIdentities-0}. 
\end{proof}

Under certain conditions, infinitesimal adjustments  can be ``integrated'' to obtain adjustments on the given crossed module of Lie groups.

\begin{theorem}
\label{integration-of-adjustments}
Let   $\Gamma=(H \stackrel{t}{\to} G \stackrel{\alpha}{\to} \Aut(H))$ be a central crossed module of Lie groups with associated crossed module of Lie algebras $\mathfrak{G}$. 
Suppose that $G$ is a locally exponential Lie group.
If $G$ is connected, then two adjustments $\kappa$ and $\kappa'$ on $\Gamma$ agree if and only if the corresponding infinitesimal adjustments $\kappa_*$ and $\kappa_*'$ on $\mathfrak{G}$ agree. 
Conversely, if $\eta$ is any infinitesimal adjustment on $\mathfrak{G}$, $G$ is additionally simply connected and $H$ is connected, then there exists a (necessarily unique) adjustment $\kappa$ on $\Gamma$ with $\kappa_* = \eta$.
\end{theorem}

The proof uses the notion of crossed homomorphisms, which we briefly recall.
Let $G$ be a Lie group and let $L$ be a Lie group with a right action of $G$. 
We assume $L$ to be abelian for simplicity.
 Recall that a \emph{crossed homomorphism} from $G$ to $L$ is a smooth map $\varphi : G \to L$ such that
 \[
 \varphi(gh) = \varphi(g) \cdot h + \varphi(h)
 \]
 holds for all $g, h \in G$.
 The differential of $\varphi$ at the identity is a \emph{crossed homomorphism of Lie algebras} $\varphi_* : \mathfrak{g} \to \mathfrak{l}$, meaning that it satisfies
 \[
   \varphi_*([X, Y]) =  \varphi_*(X) \cdot Y - \varphi_*(Y) \cdot X,
 \]
 where $\mathfrak{g}$ acts on $\mathfrak{l}$ by differentiating the action of $G$ on $L$.
 %
 \begin{comment} 
This is easiest seen using the identification with the corresponding Lie group homomorphism $\tilde{\varphi} = (\id_G, \varphi) : G \to G \ltimes L$ discussed below: Namely, 
\begin{align*}
([X, Y], \varphi_*([X, Y])) 
&= \tilde{\varphi}_*([X, Y])
\\
&= [\tilde{\varphi}_*(X), \tilde{\varphi}_*(Y)] 
\\
&= [(X, \varphi_*(X)), (Y, \varphi_*(Y)]
\\
&= ([X, Y], \varphi_*(X) \cdot Y - \varphi_*(Y) \cdot X).
\end{align*}
        If $L$ is non-abelian, one obtains the more general formula
        \[
           \varphi_*([X, Y]) = [\varphi_*(X), \varphi_*(Y)] + \varphi_*(X) \cdot Y - \varphi_*(Y) \cdot X,
        \]
        but this is not relevant for us.
 \end{comment}
 %
Crossed homomorphism of Lie groups may be identified with Lie group homomorphism $G \to G \ltimes L$ whose second component is the identity.
 %
 \begin{comment}
        Indeed, setting $\tilde{\varphi} = (\id_G, \varphi) : G \to G \ltimes L$ with $\varphi$ a crossed homomorphism, we have
        \[
        \tilde{\varphi}(gh) = \bigl(gh, \varphi(gh)\bigr) = \bigl(gh, (\varphi(g) \ract h) \cdot \varphi(h)\bigr) = (g, \varphi(g)) \cdot (h, \varphi(h)) = \tilde{\varphi}(g) \cdot \tilde{\varphi}(h),
        \]
        hence $\tilde{\varphi}$ is a Lie group homomorphism.
        Conversely, if $\tilde{\varphi} = (\id_G, \varphi) : G \to G \ltimes K$ is a Lie group homomorphism, then
        \[
        (gh, \varphi(gh)) = \tilde{\varphi}(gh) = \tilde{\varphi}(g)\tilde{\varphi}(h) = (g, \varphi(g)) \cdot (h, \varphi(h)) = (gh, (\varphi(g) \lact h) \cdot \varphi(h)), 
        \]
        hence $\varphi$ is a crossed homomorphism.
 \end{comment}
 % 
 In particular,  usual  statements for Lie group homomorphisms also hold for crossed homomorphisms:
if $\varphi, \psi : G \to L$ are two crossed homomorphisms whose differentials agree at the identity and $G$ is connected, then $\varphi = \psi$ (this needs that $G$ is a locally exponential Lie group).
 And conversely, if $G$ is additionally simply connected, then any crossed homomorphism of Lie algebras integrates (uniquely) to a crossed homomorphism of Lie groups.

\begin{proof}[Proof of \cref{integration-of-adjustments}]
Condition \eqref{adjustment-condition-1} for adjustments can be reformulated to say that every adjustment $\kappa$ yields a crossed homomorphism $\tilde{\kappa} : G \to \Lin(\mathfrak{g}, \mathfrak{h})$ (the space of continuous linear maps $\mathfrak{g} \to \mathfrak{h}$) by setting $\tilde{\kappa}(g)(X) = \kappa(g, X)$, where $G$ acts on $\Lin(\mathfrak{g}, \mathfrak{h})$ from the right by pre-composition with the adjoint action.
Then, the corresponding infinitesimal adjustment is just its differential at $1\in G$. 
Thus, the first claim follows from the fact that crossed homomorphisms on a connected and locally exponential domain $G$ are determined by their induced Lie algebra homomorphisms.  

Conversely, if $\eta$ is an infinitesimal adjustment and $\tilde{\eta}: \mathfrak{g} \to \Lin(\mathfrak{g}, \mathfrak{h})$ is defined by $\tilde{\eta}(X)(Y) = \eta(X, Y)$, then the identity \eqref{EtaIdentities-0} may be rewritten as
\[
\tilde{\eta}([X, Y]) = \tilde{\eta}(X) \circ \ad_Y - \tilde{\eta}(Y) \circ \ad_X. 
\]
In other words, an infinitesimal adjustment yields a crossed homomorphism of Lie algebras $\mathfrak{g} \to \Lin(\mathfrak{g}, \mathfrak{h})$, where $\mathfrak{g}$ acts on $\Lin(\mathfrak{g}, \mathfrak{h})$ from the right by post-composing with the adjoint action.
If now $G$ is connected and simply connected (as well as locally exponential), there exists a unique crossed homomorphism $\tilde{\kappa} : G \to \Lin(\mathfrak{g}, \mathfrak{h})$ of Lie groups with differential $\tilde\eta$, and it corresponds to a map $\kappa : G \times \mathfrak{g} \to \mathfrak{h}$ satisfying condition \eqref{adjustment-condition-1} for adjustments.

We need to check that $\kappa$ satisfies conditions \eqref{adjustment-condition-2} 
and \eqref{adjustment-condition-3}.
To see condition \eqref{adjustment-condition-2}, we observe that both sides of this equation are crossed homomorphisms of Lie groups $H \to \Lin(\mathfrak{g}, \mathfrak{h})$, and \eqref{adjustment-condition-2} differentiates to \eqref{EtaIdentities-1} (where both sides are crossed homomorphisms $\mathfrak{h} \to \Lin(\mathfrak{g}, \mathfrak{h})$ of Lie algebras).
Since crossed homomorphisms of Lie groups on a connected Lie group are determined by their differentials at the identity, condition \eqref{adjustment-condition-2} follows from \eqref{EtaIdentities-1}, provided that $H$ is connected.
Condition \eqref{adjustment-condition-3} is checked similarly, but does not need any further connectedness requirements.
\end{proof}

One may also give an explicit formula for the integration of an infinitesimal adjustment.
To this end, we use the differential equation 
\begin{equation}
\label{ODEkappa}
\frac{d}{ds} \kappa(e^{sX}, Y) = \kappa\big(e^{sX}, [X, Y]\big) + \kappa_*(X, Y)
\end{equation}
satisfied by any adjustment $\kappa$ and its corresponding infinitesimal adjustment $\kappa_{*}$ (this follows directly from differentiating \eqref{adjustment-condition-1}).
%
\begin{comment}
We have
\[
\frac{d}{ds} \kappa(e^{sX}, Y) = \frac{d}{d\varepsilon}\Big|_{\varepsilon=0} \kappa(e^{sX} e^{\varepsilon X}, Y) = \frac{d}{d\varepsilon}\Big|_{\varepsilon=0} \kappa(e^{sX}, \mathrm{Ad}_{e^{\varepsilon X}}(Y))+\kappa(e^{\varepsilon X}, Y)
\]
\end{comment}
%
Passing to higher derivatives, one can prove the estimates necessary to show that $s \mapsto \kappa(e^{sX}, Y)$ is  an analytic function, and that in a neighborhood of the identity, the adjustment is given by the formula
\[
\kappa(e^{X}, Y) 
= \sum_{n=1}^\infty \frac{1}{n!} \sum_{k=0}^{n-1} \kappa_*\big(X, \ad_X^k(Y)\big).
\]

\begin{remark}
\label{RemarkTrivialG}
In certain cases, the connectedness assumptions from \cref{integration-of-adjustments} may be achieved by replacing the crossed module with a weakly equivalent one.
Indeed, if $\smash{\Gamma=(H \stackrel{}{\to} G \stackrel{}{\to} \Aut(H))}$ is a central crossed module of Lie groups 
with $F:=\pi_0\Gamma$ connected, then 
there exists a new crossed module $\smash{\tilde{\Gamma}=(\tilde H \stackrel{}{\to} \tilde G \stackrel{}{\to} \Aut(\tilde H))}$ with $\tilde{G}$ the  universal cover of the identity component of $G$ and  $\tilde{H} := \tilde{G} \times_G H$, together with a strict morphism $\tilde{\Gamma} \to \Gamma$ which induces an identities on $A := \pi_1\Gamma$ and $F$; hence, a weak equivalence of crossed modules.
The group $\tilde{G}$ is then connected and simply connected. 
However, $\pi_0(\tilde{H})$ is an extension of $\pi_1(F)$ by $\pi_0(A)$, and hence may be non-trivial, still not meeting the assumption of  \cref{integration-of-adjustments} on $H$ unless these groups vanish.
\end{remark}

\section{Classification of central crossed modules of Lie algebras}

\label{kassel-loday}

We review the classification of crossed modules of Lie algebras obtained by Kassel and Loday  \cite{Kassel1982}, in a slightly modified form, restricting it to \emph{central} crossed modules and using butterflies as a model for weak equivalences. 

\subsection{The Kassel-Loday class}

\label{kassel-loday-class}

Let $\mathfrak{g}$ be a locally convex Lie algebra and let $V$ be a locally convex vector space.
We  consider the complex $\mathrm{Alt}^*(\mathfrak{g}, V)$ of continuous alternating multi-linear maps on $\mathfrak{g}$ taking values in $V$, equipped with the Chevalley-Eilenberg differential
\[
\delta \omega(X_1, \dots, X_{p+1}) = \sum_{i < j} (-1)^{i+j} \omega([X_i, X_j], X_1, \dots, \widehat{X_i}, \dots, \widehat{X_j}, \dots, X_{p+1}), 
\]
where the hat indicates that the argument is omitted.
The corresponding cohomology groups are denoted by $H^k(\mathfrak{g}, V)$.

A \emph{splitting} of a crossed module $\mathfrak{G}=(\mathfrak{h} \stackrel{t_*}{\to} \mathfrak{g} \stackrel{\alpha_*}{\to} \Der(\mathfrak{h)})$ of Lie algebras is a linear map $u: \mathfrak{g} \to \mathfrak{h}$ such that
\begin{equation}
\label{splitting}
  t_* ut_* = t_*  \qquad \text{and} \qquad ut_* u = u.
\end{equation}
A splitting is the same datum as the choice of vector space complements of $\mathfrak{a} \subseteq \mathfrak{h}$ and $t_*(\mathfrak{h}) \subseteq \mathfrak{g}$. 
A \emph{half splitting} is a linear map $u:\mathfrak{g} \to \mathfrak{h}$ that satisfies only the first condition in \eqref{splitting}.
Every (half) splitting determines an idempotent linear map $\rho_u:\mathfrak{g} \to \mathfrak{g}$ with image $t(\mathfrak{h})$, by 
\[
\rho_u:=t_* u.
\] 
Conversely, for every such idempotent $\rho$ there exists a splitting $u$ with $\rho=\rho_u$.  

As explained before, idempotents of $\mathfrak{g}$ with image $t(\mathfrak{h})$ are the same as sections  $s: \mathfrak{f} \to \mathfrak{g}$ against the projection $p: \mathfrak{g} \to \mathfrak{f}$, the relation being $\rho_s + s p =\id_{\mathfrak{g}}$. If $u$ is a splitting or half splitting, we denote the corresponding section by $s_u$. 

Summarizing, we have  maps
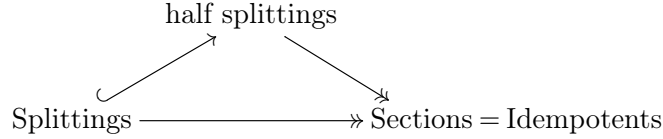
\begin{equation*}
\xymatrix@C=0.6em{& \text{half splittings} \ar@{->>}[rd]\\\text{Splittings} \ar@{->>}[rr] \ar@{^(->}[ur] &  & \text{Sections} \ar@{=}[r] & \text{Idempotents}
}
\end{equation*}
with the indicated injectivity and surjectivity behaviour.

\begin{comment}
Any other half splitting has the form $u' = u + u \xi + \psi$, where $\xi : \mathfrak{g} \to t_* \mathfrak{h}$ vanishes on $t_* \mathfrak{h}$ and $\psi : \mathfrak{g} \to \mathfrak{a}$.
Then $v = u \xi + \psi$.
Indeed, this is a half splitting as
\[
t_*u't_* = t_*ut_* + t_*u \underbrace{\xi t_*}_{=0} + \underbrace{t_*\psi}_{=0} t_* = t_*.
\]
Conversely, any other linear map $u' : \mathfrak{g} \to \mathfrak{h}$  can be written as $u' = u + u \xi + \psi$ with $\psi$ $\mathfrak{a}$-valued but a priori not necessarily with $\xi t_* = 0$. 
But then we have $t_*u \xi t_* = \xi t_*$, and the above calculation shows that to get a half splitting, $\xi$ must vanish on $t_*\mathfrak{h}$

We then have $\rho' = t_* u' = \rho + \xi$.
This is indeed a projection because $(\rho')^2 = \rho^2 + \rho\xi + \xi \rho + \xi^2 = \rho + \xi = \rho'$, as $\xi^2 = \xi \rho = 0$ and $\rho \xi = \xi$.
We conclude that $u'$ induces the same section as $u$ if and only if $\xi = 0$.

If $u$ is a (full) splitting, then $u'$ is a full splitting if and only if
\[
u + u \xi + \psi = u' \stackrel{!}{=} u' t_* u' 
= (u + \psi)t_* (u + u \xi) 
= u + \psi t_*u + u \xi + \psi \xi
\quad \Longleftrightarrow \quad \psi = \psi(t_*u + \xi).
\]
\end{comment}

We consider  a central  crossed module $\mathfrak{G}=(\mathfrak{h} \stackrel{t_*}{\to} \mathfrak{g} \stackrel{\alpha_*}{\to} \Der(\mathfrak{h}))$ of Lie algebras, and a half splitting $u$.
We let $\rho := t_*u$ be the corresponding idempotent  and set $\rho^\perp := \id-\rho$.
We moreover define
\begin{equation}
\label{omegaU}
\omega_u(X, Y) := \alpha_*\big(X, u(Y)\big) - \alpha_*\big(Y, u(X)\big) - \big[u(X), u(Y)\big] + u\big([\rho^\perp(X), \rho^\perp(Y)]\big), 
\end{equation}
an element of $\Alt^2 (\mathfrak{g}, \mathfrak{h})$.

\begin{lemma}
\label{LemmaAdjustmentIdentities2}
$\omega_u$ satisfies the identities \eqref{EtaIdentities-1} and \eqref{EtaIdentities-2} of an infinitesimal adjustment.
\end{lemma}

\begin{proof}
Since $\rho^\perp(t_*y) = 0$, the last term cancels and we have
\begin{align*}
\omega_u(X, t_*y) &= \alpha_*\big(X, u(t_{*}y)\big) - \alpha_*\big(t_*y, u(X)\big) - \big[u(X), u(t_*y)\big]
\\
&=\alpha_*\big(X, u(t_*y) - y\big) + \alpha_*(X, y)  - \big[y, u(X)\big] + \big[u(t_*y), u(X)\big]
\\
&=\alpha_*(X, y). 
\end{align*}
Here, we used twice that the difference between $y$ and $u(t_{*}y)$ lies in the  ideal $\mathfrak{a} \subset \mathfrak{h}$, on which $\mathfrak{g}$  acts trivially because our crossed module is central.
The second identity follows from the first by exchanging $X$ and $y$ by $Y$ and $x$ and using anti-symmetry of $\omega_u$.
%
\begin{comment}
We may also calculate
\begin{align*}
\omega_u(t_*x, Y) &= \alpha_*\big(t_*x, u(Y)\big) - \alpha_*\big(Y, u(t_*x)\big) - \big[u(t_*x), u(Y)\big]
\\
&= \big[x, u(Y)\big] - \alpha_*\big(Y, u(t_*x)\big) - \big[u(t_*x), u(Y)\big]
\\
&=- \alpha_*(Y, x)
\end{align*}
The first and last term in the second row cancel, and the central term is equal to the third row, just by the same argument as before.
\end{comment}
\end{proof}

\begin{lemma}
\label{differential-and-pullback}
Suppose $\mathfrak{g}'$ is another Lie algebra, $\phi: \mathfrak{g}' \to \mathfrak{g}$ is a linear map, and $\lambda\in \Alt^2(\mathfrak{g}', \mathfrak{h})$ such that
\begin{equation*}
\big[\phi(X), \phi(Y)\big] - \phi([X, Y]) =t\lambda(X, Y)\text{.}
\end{equation*} 
Then, 
\begin{equation*}
\phi^{*}(\delta\omega_{u})= \delta(\phi^{*}\omega_{u})+\gamma_{\phi, \lambda}
\end{equation*}
holds for
\begin{equation*}
\gamma_{\phi, \lambda}(X, Y, Z):=\alpha\big(\phi(X), \lambda(Y, Z)\big)+\alpha\big(\phi(Y), \lambda(Z, X)\big)+\alpha\big(\phi(Z), \lambda(X, Y)\big)\text{.}
\end{equation*}
\end{lemma}

\begin{proof} We calculate
\begin{align*}
\phi^{*}(\delta\omega_{u_2})(X, Y, Z) 
&=\delta\omega_{u_2}\big(\phi(X), \phi(Y), \phi(Z)\big)
\\
&= -\omega_{u_2}\big([\phi(X), \phi(Y)], \phi(Z)\big)+\omega_{u_2}\big([\phi(X), \phi(Z)], \phi(Y)\big)
\\
&\hspace{0.46cm} -\omega_{u_2}\big([\phi(Y), \phi(Z)], \phi(X)\big)
\\
&= -\omega_{u_2}\big(\phi([X, Y])+t_2(\lambda(X, Y)), \phi(Z)\big)
\\
&\hspace{0.46cm}+\omega_{u_2}\big(\phi([X, Z])+t_2(\lambda(X, Z)), \phi(Y)\big)
\\
&\hspace{0.46cm}-\omega_{u_2}\big(\phi([Y, Z])+t_2(\lambda(Y, Z)), \phi(X)\big)
\\
%&= -\omega_{u_2}\big(\phi([X, Y]), \phi(Z)\big)+\omega_{u_2}\big(\phi([X, Z]), \phi(Y)\big)
%\\
%&\hspace{0.46cm}-\omega_{u_2}\big(\phi([Y, Z]), \phi(X)\big)-\omega_{u_2}\big(t_2(\lambda(X, Y)), \phi(Z)\big)
%\\
%&\hspace{0.46cm} +\omega_{u_2}\big(t_2(\lambda(X, Z)), \phi(Y)\big)-\omega_{u_2}\big(t_2(\lambda(Y, Z)), \phi(X)\big)
%\\
&= \delta(\phi^{*}\omega_{u_2})(X, Y, Z)
-\omega_{u_2}\big(t_2(\lambda(X, Y)), \phi(Z)\big)
\\
&\hspace{0.46cm} +\omega_{u_2}\big(t_2(\lambda(X, Z)), \phi(Y)\big)-\omega_{u_2}\big(t_2(\lambda(Y, Z)), \phi(X)\big)\text{.}
\end{align*}
Applying \cref{LemmaAdjustmentIdentities2} yields the claimed identity.
\end{proof}

\begin{lemma}
\label{LemmaOmegaUadapted}
We have $t_* \omega_u = -t_* \delta u$, or explicitly
\[
t_*\omega_u(X, Y) = t_*u([X, Y]) = \rho([X, Y]).
\]      
\end{lemma}

\begin{proof}
Since $t_*$ is a Lie algebra homomorphism that intertwines $\alpha_*$ with the commutator, we have
\begin{align*}
t_*\omega_u(X, Y)
&= [X, \rho(Y)] - [Y, \rho(X)] - [\rho(X), \rho(Y)] + \rho [\rho^\perp(X), \rho^\perp(Y)]
\\
&= \rho \Big([X, \rho(Y)] + [\rho(X), Y] - [\rho(X), \rho(Y)] + [\rho^\perp(X), \rho^\perp(Y)]\Big)
\\
&= \rho([X, Y]).
\end{align*}
Here in the second step, we used that the first three terms are all contained in $t_*\mathfrak{h}$, where $\rho$ acts as the identity.
\end{proof}

\begin{lemma}
\label{LemmaDeltaOmegaDescends}
The cocycle 
 $\delta \omega_u\in \Alt^3 _{\mathrm{cl}}(\mathfrak{g}, \mathfrak{h})$ vanishes on $t_* \mathfrak{h}$ and is $\mathfrak{a}$-valued.
\end{lemma}

\noindent
Here, we mean that $\delta \omega_u$ vanishes as soon as \emph{one} of the three entries is contained in $t_*\mathfrak{h}$.

\begin{proof}
We have
\begin{equation}
\label{DeltaOfOmegaU}
\begin{aligned}
\delta\omega_u(X, Y, Z)
= &-\omega_u([X, Y], Z) + \omega_u([X, Z], Y) - \omega([Y, Z], X)
\\
= &-\alpha_*\big([X, Y], u(Z)\big) + \alpha_*\big(Z, u([X, Y])\big) + \big[u([X, Y]), u(Z)\big]
\\
 &+ \alpha_*\big([X, Z], u(Y)\big) - \alpha_*\big(Y, u([X, Z])\big) - \big[u([X, Z]), u(Y)\big]
\\
&- \alpha_*\big([Y, Z], u(X)\big) + \alpha_*\big(X, u([Y, Z])\big) + \big[u([Y, Z]), u(X)\big]
\\
& - u\big(\big[\rho^\perp([X, Y]), \rho^\perp(Z)\big]\big) + u\big(\big[\rho^\perp([X, Z]), \rho^\perp(Y)\big]\big) 
\\
&- u\big(\big[\rho^\perp([Y, Z]), \rho^\perp(X)\big]\big).
\end{aligned}
\end{equation}
Setting $X = t_*x$, the last three terms vanish, as in each case, $\rho^\perp$ is applied to an element of $t_*\mathfrak{h}$ (for the first two of these three terms, we use that $t_*\mathfrak{h}$ is an ideal).
Hence we get
\begin{align*}
\delta\omega_u(t_*x, Y, Z)
= &-\alpha_*\big([t_*x, Y], u(Z)\big) + \alpha_*\big(Z, u([t_*x, Y])\big) + \big[u([t_*x, Y]), u(Z)\big]
\\
 &+ \alpha_*\big([t_*x, Z], u(Y)\big) - \alpha_*\big(Y, u([t_*x, Z])\big) - \big[u([t_*x, Z]), u(Y)\big]
\\
&- \alpha_*\big([Y, Z], u(t_*x)\big) + \alpha_*\big(t_*x, u([Y, Z])\big) + \big[u([Y, Z]), u(t_*x)\big]
\\
= &~~~~\cancel{\big[\alpha_*(Y, x), u(Z)\big]} - \alpha_*\big(Z, \alpha_*(Y, x)\big) - \cancel{\big[u([Y, t_*x]), u(Z)\big]}
\\
 &- \xcancel{\big[\alpha_*(Z, x), u(Y)\big]} + \alpha_*\big(Y, \alpha_*(Z, x)\big) + \xcancel{\big[u([Z, t_*x]), u(Y)\big]}
\\
&- \alpha_*\big([Y, Z], x\big) + \bcancel{\big[x, u([Y, Z])\big]} - \bcancel{\big[u(t_*x), u([Y, Z])\big]}
\\
= & ~~~\, \alpha_*\big(Y, \alpha_*(Z, x)\big) - \alpha_*\big(Z, \alpha_*(Y, x)\big) - \alpha_*\big([Y, Z], x\big)
\\
= & ~0. 
\end{align*}
Here, in the second step, we used that the differences $u([X, t_*y]) - \alpha_*(X, y)$, $u([t_*x, Y]) + \alpha_*(Y, x)$ and $u(t_*x) - x$ all lie in $\mathfrak{a}$, on which the action is trivial.
By anti-symmetry, $\delta \omega_u(X, Y, Z)$ also vanishes when one of $Y$ and $Z$ lies in $t_*\mathfrak{h}$.

To see that $\delta\omega_u$ is $\mathfrak{a}$-valued, we calculate, using \cref{LemmaOmegaUadapted}, 
\begin{equation*}
t_*\delta\omega_u(X, Y) = \delta t_*\omega_u(X, Y) = -\delta \delta \rho([X, Y]) = 0.
\qedhere
\end{equation*}
%
\begin{comment}
Explicit calculation that $\delta \omega_u$ is $\mathfrak{a}$-valued:  Applying $t_*$ to the first nine terms of \cref{DeltaOfOmegaU} 
gives
\begin{align*}
&-\big[[X, Y], \rho(Z)\big] + \big[Z, \rho([X, Y])\big] + \big[\rho([X, Y]), \rho(Z)\big]
\\
 &+ \big[[X, Z], \rho(Y)\big] - \big[Y, \rho([X, Z])\big] - \big[\rho([X, Z]), \rho(Y)\big]
\\
&- \big[Y, Z], \rho(X)\big] + \big[X, \rho([Y, Z])\big] + \big[\rho([Y, Z]), \rho(X)\big].
\end{align*}
Now an application of the Jacobi identity shows that this equals
\begin{equation*}
\big[\rho^\perp([X, Y]), \rho^\perp(Z)\big] - \big[\rho^\perp([X, Z]), \rho^\perp(Y)\big] + \big[\rho^\perp([Y, Z]), \rho^\perp(X)\big].
\end{equation*}
Hence this sum is in fact contained in $t_*\mathfrak{h}$. 
Notice now that the last three terms of \cref{DeltaOfOmegaU} are precisely $u$ applied to this expression. 
Since $t_*u = \rho$ is the identity on $t_* \mathfrak{h}$, we therefore see that the last three terms of \cref{DeltaOfOmegaU} cancel precisely with the first nine after applying $t_*$.
\end{comment}
%
\end{proof}

By \cref{LemmaDeltaOmegaDescends}, $\delta \omega_u$ descends to $\mathfrak{f}$, i.e., there exists a unique 3-cocycle $C_u\in\Alt^3_{\mathrm{cl}}(\mathfrak{f}, \mathfrak{a})$ such that $\delta \omega_u = p^*C_u$. 
It  defines a cohomology class
\[
  [C_u] \in H^3(\mathfrak{f}, \mathfrak{a}).
\]
\begin{comment}
This element is  in general not exact, as $\omega_u$ itself does \emph{not} descend to the quotient.
\end{comment}

\begin{lemma}
\label{LemmaChangeSplitting}
If $u, u': \mathfrak{g} \to \mathfrak{h}$ are half-splittings, then the difference $\omega_{u'}-\omega_{u}$  descends to $\mathfrak{f}$. Moreover, we have $[C_u]=[C_{u'}]$. 
\end{lemma}

\begin{proof}
Let $u, u' : \mathfrak{g} \to \mathfrak{h}$ be two half splittings and write $v := u' - u$.
Since $t_* u' t_* = t_*$, we see that $t_* v t_* = 0$.
Therefore, $v$ is $\mathfrak{a}$-valued on $t_* \mathfrak{h}$.
Then, 
\begin{align*}
\omega_{u'}(X, Y) - \omega_{u}(X, Y)
=~& \alpha_*\big(X, v(Y)\big) - \alpha_*\big(Y, v(X)\big)
\\
& - \big[v(X), u(Y)\big] - \big[u(X), v(Y)\big]
- [v(X), v(Y)]
\\
&+ u'\big([(\rho')^\perp(X), (\rho')^\perp(Y)]\big) - u\big([\rho^\perp(X), \rho^\perp(Y)]\big).
\end{align*}
The last two terms vanish on $t_*\mathfrak{h}$.
Taking $X = t_* x$, the vector $v(X) = v(t_*x)$ is $\mathfrak{a}$-valued, hence the second, third and fifth term vanish as well.
We get
\begin{equation*}
\omega_{u'}(t_*x, Y) - \omega_{u}(t_*x, Y) = \big[x, v(Y)\big] - \big[ut_*x, v(Y)\big] = \big[x - ut_*x, v(Y)\big]= 0, 
\end{equation*}
again using that $x - ut_*x \in \mathfrak{a}$.
That also $\omega_{u'}(X, t_*y) - \omega_{u}(X, t_*y) =0$ follows from  anti-symmetry.
We conclude that the difference $\omega_{u'} - \omega_u$ vanishes on $t_*\mathfrak{h}$ and hence descends to $\mathfrak{f}$, showing the first claim.

We remark that the difference $\omega_{u'} - \omega_u$  is, however, not necessarily $\mathfrak{a}$-valued.
We claim that $\omega_{u'} - \omega_u + \delta(ut_*v)$ is $\mathfrak{a}$-valued and also vanishes on $t_*\mathfrak{h}$.
Indeed, $u t_*v$ vanishes on $t_*\mathfrak{h}$ as $t_*vt_* = 0$, and
\[
t_*\big(\omega_{u'} - \omega_u + \delta(ut_*v)\big) = - t_*\delta u' - t_*\delta u + tut_* \delta v = 0, 
\]
hence the form is indeed $\mathfrak{a}$-valued.
We conclude that there exists $\theta_{u, u'} \in \Alt(\mathfrak{f}, \mathfrak{a})$ such that 
\[
p^*\theta_{u, u'} = \omega_{u'} - \omega_u + \delta(ut_* v).
\]
With this definition, we have
\[
p^* \delta \theta_{u, u'} = \delta(p^* \theta_{u, u'})  = \delta\omega_{u'} - \delta\omega_u, 
\]
hence
\[
\delta \theta_{u, u'} = C_{u'} - C_u, 
\]
so that $C_{u}$ and $C_{u'}$ define the same cohomology class.
\end{proof}

By \cref{LemmaChangeSplitting}, the cohomology class $[C_u]$ only depends on the crossed module $\mathfrak{G}$ of Lie algebras.
The cocycle $C_u$ was considered by Kassel-Loday in the appendix of \cite{Kassel1982}  (however, our form $\omega_u$ differs from theirs by the exact term $\delta u$).

\begin{definition}
We call
\[
\KL(\mathfrak{G}):=[C_u]\in H^3(\mathfrak{f}, \mathfrak{a})
\]
the \emph{Kassel-Loday class} of $\mathfrak{G}$.
\end{definition}

\subsection{Invariance under butterflies}

In this section we re-examine the existing result that the Kassel-Loday class of a crossed module of Lie algebras is invariant under weak equivalences, which we realize here as butterflies (see \cref{butterflies}).
%In this section, we omit the index ``$*$'' that usually indicates an induced Lie algebra homomorphism. 

We consider two central crossed modules  $\mathfrak{G}_i=(\mathfrak{h}_i \stackrel{t_i}{\to} \mathfrak{g}_i \stackrel{\alpha_i}{\to} \Der(\mathfrak{h}_i))$ of Lie algebras, for $i=1, 2$, equipped with half splittings $u_i: \mathfrak{g}_i \to \mathfrak{h}_i$ inducing the corresponding 2-cochains $\omega_{u_i}\in \Alt^2(\mathfrak{g}_i, \mathfrak{h}_i)$, and a butterfly $\mathfrak{k}:\mathfrak{G}_1\to \mathfrak{G}_2$,   
\begin{equation}
\begin{aligned}
\xymatrix{\mathfrak{h}_1 \ar[dd]_{t_1} \ar[dr]^{i_1} && \mathfrak{h}_2 \ar[dd]^{t_2}\ar[dl]_{i_2} \\ & \mathfrak{k} \ar[dl]^{r_1}\ar[dr]_{r_2} \\ \mathfrak{g}_1 && \mathfrak{g}_2\text{. }}
\end{aligned}
\end{equation}
% \[
% \begin{tikzcd}
% \mathfrak{h}_1 \ar[dd, "t_1"'] \ar[dr, "i_1"] && \mathfrak{h}_2 \ar[dd, "t_2"] \ar[dl, "i_2"'] \\ & 
% \mathfrak{k} 
% \ar[ur, dashed, bend right =25, "j"']
% \ar[dl, "r_1"] \ar[dr, "r_2"']  \\ \mathfrak{g}_1 
% \ar[ur, dashed, bend left=25, "q"]&& \mathfrak{g}_2\text{, }
% \end{tikzcd}
% \]
We choose a section $q: \mathfrak{g}_1 \to \mathfrak{k}$ against $r_1$, obtaining the linear maps $\phi_q := r_2q: \mathfrak{g}_1 \to \mathfrak{g}_2$ and $f_q: \mathfrak{h}_1 \to \mathfrak{h}_2$, as well as the cochain $\lambda_q \in \Alt^2(\mathfrak{g}_1,\mathfrak{h}_2)$, as described in detail in \cref{butterflies}.
Since $\phi_q$ is  not a Lie algebra homomorphism, the Chevalley-Eilenberg differential $\delta$ does not commute with pullback along $\phi_q$.
Instead, we have the following lemma.

\begin{lemma}
\label{calc-butterfly-1}
The following equality of elements of $\Alt^3(\mathfrak{g}_1, \mathfrak{h}_2)$ holds:
\begin{equation*}
\phi_q^{*}(\delta\omega_{u_2})= \delta(\phi_q^{*}\omega_{u_2})-\delta \lambda_q\text{.}
\end{equation*}
\end{lemma}

\begin{proof}
By \cref{differential-and-pullback} we have 
\begin{equation*}
\phi_q^{*}(\delta\omega_{u_2})= \delta(\phi_q^{*}\omega_{u_2})+\gamma_{\phi_q, \lambda_q}\text{.}
\end{equation*} 
The expression for $\gamma_{\phi_q, \lambda_q}$ obtained in  \cref{differential-and-pullback} coincides with the one for $-\delta\lambda_q$, as computed in  \cref{cyclic-identity}. This shows the claim.
\end{proof}

Next we consider the 2-cochain $R'_q\in \Alt^2(\mathfrak{g}_1, \mathfrak{h}_2)$ defined by
\begin{equation}
\label{cochain-of-butterfly}
R_q' := \phi_q^* \omega_{u_2} - f_q(\omega_{u_1}) - \lambda_q, 
%R'_q(X, Y):=\omega_{u_2}(\phi_q(X), \phi_q(Y)) -f_q(\omega_{u_1}(X, Y))-\lambda_q(X, Y)\text{, }
\end{equation}
obtaining by \cref{calc-butterfly-1} 
\begin{equation*}
\delta R'_q = \phi_q^{*}(\delta\omega_{u_2})-f_q(\delta\omega_{u_1})\text{.} 
\end{equation*}
%
\begin{comment}
\[
\delta R'_q = \delta(\phi^*\omega_{u_2}) - f_q(\delta\omega_{u_1}) - \delta \lambda_q = \phi_1^*(\delta \omega_{u_2}) - f_q(\delta\omega_{u_1})
\]
\end{comment}
We note that $R'_q$ descends to $\mathfrak{f}_1$, as it is skew-symmetric and \cref{LemmaAdjustmentIdentities2} (together with \cref{PhiFIntertinest,exchange-of-actions}) implies
\begin{align*}
R'_q\big(t_1(y), X\big) 
&= \omega_{u_2}\big(\phi_q(t_1(y)), \phi_q(X)\big) -f_q\big(\omega_{u_1}(t_1(y), X)\big)-\lambda_q\big(t_1(y), X\big) 
\\&= -\alpha_2\big(\phi_q(X), f_q(y)\big) +f_q\big(\alpha_1(X, y)\big)-\lambda_q\big(t_1(y), X\big)
\\&=0\text{.} 
\end{align*}
We denote the descended bilinear form by $R_q\in \Alt^2(\mathfrak{f}_1, \mathfrak{h}_2)$. 
Thus, we obtain an equality
\begin{equation*}
\delta {R_q}= (\pi_0\mathfrak{k)}^{*}{C}_{u_2}-(\pi_1\mathfrak{k)}_{*}{C}_{u_1}.
\end{equation*}
Setting $\tilde R_q :=(\id-u_2t_2)(R_q)$ and using that both $(\pi_1\mathfrak{k})_*C_{u_1}$ and $C_{u_2}$ take values in $\mathfrak{a}_2$, we finally obtain
\begin{equation*}
\delta {\tilde R_q} = (\pi_0\mathfrak{k)}^{*}{C}_{u_2}-(\pi_1\mathfrak{k)}_{*}{C}_{u_1}\text{.} 
\end{equation*}
This shows the following result.

\begin{proposition}
\label{invariance-of-KL}
The Kassel-Loday classes of  two central crossed modules of Lie algebras related by a butterfly $\mathfrak{k}:\mathfrak{G}_1 \to \mathfrak{G}_2$ satisfy the following pull-push-relation:
\begin{equation*}
(\pi_0\mathfrak{k})^{*}\KL(\mathfrak{G}_2) = (\pi_1\mathfrak{k})_{*}\KL(\mathfrak{G}_1)\text{.}
\end{equation*}
\end{proposition}

\begin{remark}
\label{splittings-assumption}
Given a butterfly $\mathfrak{k}:\mathfrak{G}_1 \to \mathfrak{G}_2$, and sections $s_i: \mathfrak{f}_i \to \mathfrak{g}_i$ in both crossed modules, then a section $q$ in $\mathfrak{k}$ is called \emph{neat} if the corresponding idempotents $\rho_{s_i}$ satisfy
\begin{equation}
\label{Neatness}
\rho_{s_2}\phi_q=\phi_q \rho_{s_1}\text{.}      
\end{equation}
\begin{comment}
In terms of just the sections, this is equivalent to
\begin{equation*}
s_2p_2\phi_q=\phi_qs_1p_1\text{.}
\end{equation*}
This is again equivalent to
\begin{equation*}
s_2\Phi = \phi_q s_1\text{.}
\end{equation*}
\end{comment}
Neatness can always be achieved: for some section $q$, we have $p_2(s_2(\pi_0\mathfrak{k})-\phi_q s_1)=0$, and thus there exists a linear map $f: \mathfrak{f}_1\to \mathfrak{h}_2$ such that $s_2(\pi_0\mathfrak{k} )= \phi_q s_1 + t_2f$. Then, $q' := q+i_2 fp_1$ is a neat section.  
\begin{comment}
Indeed, we have
\begin{equation*}
r_1q' = r_1q+r_1i_2fp_1=\id+0=\id
\end{equation*}
so that $q'$ is a section, and
\begin{equation*}
\phi_q s_1 = r_2q's_1=r_2 q s_1+r_2 i_2 f p_1 s_1 =r_2 q s_1 +  t_2 f=r_2 q s_1+s_2\pi_0\mathfrak{k} -r_2 q s_1=s_2\pi_0\mathfrak{k} 
\end{equation*}
shows it is neat. 
\end{comment}
\begin{comment}
When $\pi_0\mathfrak{k}$ is an isomorphism, one could also keep $q$  and change the section $s_2$; indeed, $s_2 := r_2 q s_1 \pi_0\mathfrak{k}^{-1}$ is a section in $\mathfrak{G}_2$ such that $q$ is neat. 
\end{comment}
In the situation above, we may assume that $q$ is neat with respect to the sections $s_1=s_{u_1}$ and $s_2=s_{u_2}$ induced by the chosen half splittings. 
Then 
we have
\begin{align*}
&\hspace{-1em}t_2(R_q(X, Y))
\\&=t_2\omega_{u_2}(\phi_q(X), \phi_q(Y))  -t_2f_q(\omega_{u_1}(X, Y))-t_2\lambda_q(X, Y)
\\&=t_2 u_2([\phi_q(X), \phi_q(Y)]) - \phi_q(t_1(\omega_{u_1}(X, Y)))-t_2u_2t_2\lambda_q(X, Y) 
& & \text{\cref{LemmaOmegaUadapted},~\eqref{PhiFIntertinest}}
\\&=t_2 u_2\big([\phi_q(X), \phi_q(Y)]-t_2\lambda_q(X, Y)\big)-\phi_q( t_1(\omega_{u_1}(X, Y))) 
\\&=t_2 u_2(\phi_q([X, Y]))-\phi_q( t_1u_1([X, Y])) 
& & \text{from \eqref{tofLambda}}
\\&=t_2 u_2(\phi_q([X, Y])- t_2u_2(\phi_q([X, Y]))
& & \text{from \eqref{Neatness}}
\\&=0 
\end{align*}
so that $R_q$ is automatically $\mathfrak{a}_2$-valued, and  $\tilde R_q=R_q$.
\begin{comment}
Relation \cref{relation-between-splittings} is equivalent to a formulation using the splitting $j$, namely 
\begin{equation*}
t_2 u_2 r_2 =t_2j -t_2  j i_1  u_1  r_1\text{.}
\end{equation*}
\end{comment}
\end{remark}

\subsection{The classification of crossed modules}
\label{classification-result}

We may now fix Lie algebras $\mathfrak{f}$ and $\mathfrak{a}$,  and  consider  central crossed modules $\mathfrak{G}$ of Lie algebras together with fixed isomorphisms $\pi_0\mathfrak{G}\cong\mathfrak{f}$ and $\pi_1\mathfrak{G}\cong\mathfrak{a}$. Moreover, we consider only those butterflies $\mathfrak{k}:\mathfrak{G}_1 \to \mathfrak{G}_2$ that induce -- under the fixed isomorphisms -- the identities on $\mathfrak{f}$ and $\mathfrak{a}$. 
Together with all 2-morphisms between butterflies, this yields a bicategory $\cm(\mathfrak{f}, \mathfrak{a})$. We remark that all 1-morphisms in this bicategory are invertible, by \cref{invertibility-of-butterflies}. 

The Kassel-Loday classes of crossed modules in $\cm(\mathfrak{f}, \mathfrak{a})$ can  be identified canonically with classes in $H^3(\mathfrak{f}, \mathfrak{a})$; under this identification, 
  \cref{invariance-of-KL} shows  that isomorphic crossed modules have the same Kassel-Loday class. Thus, the Kassel-Loday class establishes a well-defined map
\begin{equation*}
\mathrm{KL}: \pi_0\cm(\mathfrak{f}, \mathfrak{a}) \to H^3(\mathfrak{f}, \mathfrak{a})\text{.}
\end{equation*} 

The following result has been proved in \cite{Kassel1982} (not using butterflies but a different model for the localization at weak equivalences). 

\begin{theorem}
\label{classification-of-crossed-modules}
$\mathrm{KL}$ is a bijection
\begin{equation*}
\pi_0 \cm(\mathfrak{f}, \mathfrak{a}) \cong H^3(\mathfrak{f}, \mathfrak{a})\text{.}
\end{equation*}
\end{theorem}

The proof of surjectivity in \cref{classification-of-crossed-modules} does not use any specific model for weak equivalences, and so the original proof applies without changes. 
We re-examine in the following the proof of injectivity using butterflies.

\begin{proof}[Proof of injectivity]
Suppose we have two crossed modules of Lie algebras with the same Kassel-Loday class. Let $u_1$ and $u_2$ be  splittings, and let $C_{u_1}$ and $C_{u_2}$ be the corresponding cocycles. 
Thus, by our assumption, there exists $R\in \Alt^2(\mathfrak{f}, \mathfrak{a})$ such that 
\begin{equation*}
\delta R= C_{u_1}-C_{u_2}\text{.}
\end{equation*}
We  denote by $j_i: \mathfrak{h}_i \to \mathfrak{a}$, $i=1, 2$, the unique linear maps such that $\iota_i j_i + u_i t_i=\id_{\mathfrak{h}_i}$. 
Consider the map $\phi := s_2p_1 : \mathfrak{g}_1 \to \mathfrak{g}_2$ and
\begin{equation*}
\lambda := (\iota_2)_{*}(p_1^{*}R -(j_1)_{*}\omega_{u_1})+\phi^{*}\omega_{u_2} \in \Alt^2(\mathfrak{g}_1, \mathfrak{h}_2) \text{.}
\end{equation*}
Then, we obtain, via  \cref{LemmaOmegaUadapted}, 
\begin{equation}
\label{butterfly-1}
\begin{aligned}
t_2\lambda (X, Y) &=t_2\omega_{u_2}(\phi(X), \phi(Y))
\\&=t_2u_2([\phi(X), \phi(Y)])
\\&=[\phi(X), \phi(Y)] -s_2p_2([s_2p_1(X), s_2p_1(Y)])
\\&=[\phi(X), \phi(Y)] -s_2([\underbrace{p_2s_2}_{=\id_{\mathfrak{f}_2}}p_1(X), \underbrace{p_2s_2}_{=\id_{\mathfrak{f}_2}}p_1(Y)])
\\&=[\phi(X), \phi(Y)] -s_2([p_1(X), p_1(Y)])
\\&=[\phi(X), \phi(Y)] -\phi([X, Y])\text{.}
\end{aligned}   
\end{equation}
Moreover, we have 
\begin{equation}
\label{butterfly-2} 
\begin{aligned}
\delta\lambda &= (\iota_2)_{*}(p_1^{*}\delta R - j_1\delta\omega_{u_1})+\delta\phi^{*}\omega_{u_2}
\\&= (\iota_2)_{*}(p_1^{*}C_{u_1}-\phi^{*}p_2^{*}C_{u_2} - j_1\delta\omega_{u_1})+\delta\phi^{*}\omega_{u_2}  
\\&= -(\iota_2)_{*}j_2\phi^{*}\delta\omega_{u_2}+\delta\phi^{*}\omega_{u_2} 
\\&= -\phi^*\delta\omega_{u_2}+\delta\phi^{*}\omega_{u_2}
\\&= -\gamma_{\phi, \lambda} %-\alpha_2\big(\phi(X), \lambda(Y, Z)\big)-\alpha_2\big(\phi(Y), \lambda(Z, X)\big)-\alpha_2\big(\phi(Z), \lambda(X, Y)\big)
\end{aligned}
\end{equation}
where the last step is \cref{differential-and-pullback} and $\gamma_{\phi, \lambda}$ is defined there.
Finally, we define  $f := \iota_2j_1  : \mathfrak{h}_1 \to \mathfrak{h}_2$. 
We have
\begin{equation}
\label{butterfly-3}
t_2f =0=\phi t_1 
\end{equation}
and, using \cref{LemmaAdjustmentIdentities2},
\begin{equation}\label{butterfly-4}
\begin{aligned}
\alpha_2(\phi(X), f(y)) &= \alpha_2(s_2p_1(X), \iota_2j_1(y))
\\&=0
\\&=\iota_2j_1(\alpha_1(X, y))-\iota_2 j_1\omega_{u_1}(X, t_1(y))
\\&=f(\alpha_1(X, y))+\lambda(X, t_1(y)) .
\end{aligned}
\end{equation}
Identities \eqref{butterfly-1}, \eqref{butterfly-2}, \eqref{butterfly-3} \eqref{butterfly-4} allow us to apply \cref{constructing-butterflies} to the data of $(\phi, f, \lambda)$; this yields a butterfly $\mathfrak{k}:\mathfrak{G}_1 \to \mathfrak{G}_2$ with $\pi_0\mathfrak{k}$ and $\pi_1\mathfrak{k}$ identities. By \cref{invertibility-of-butterflies}, it is hence invertible. 
\end{proof}

\begin{remark}
\label{smooth-group-cohomology-classification}
Central (and smoothly separable) crossed modules of Lie \emph{groups} with fixed homotopy Lie groups $A$ and $F$ have a similar classification, by Lie group cohomology $H^3(F, A)$, via a bijection
\begin{equation*}
\CM(F, A) \cong H^3(F, A)\text{.}
\end{equation*}
Here, it is important to not use the straightforward  smooth version of group cohomology, but rather to include certain local resolutions. This can be achieved using derived functors (\quot{Segal-Mitchison}, see \cite{Lane1975a,pries2}),  using a \v Cech resolution \cite{brylinski3}, or by considering only locally smooth cochains \cite{Neeb2005}. Wagemann and Wockel set up a unified framework  and also prove the classification of crossed modules of Lie groups \cite[Thm. V.4]{Wagemann2015} claimed above; also see \cite[Lemma III.6]{Neeb2005}.
Differentiation is  a map
\begin{equation*}
H^3(F, A) \to H^3(\mathfrak{f}, \mathfrak{a})\text{,}
\end{equation*}
and produces the Kassel-Loday class of the corresponding crossed module of Lie algebras. The kernel of this differentiation map is $H^3(BF, A^{\delta})$, i.e., the singular cohomology of the classifying space of $F$ with values in $A$ (considered as a discrete abelian group)  \cite[Rem. V.13]{Wagemann2015}. 
\end{remark}

One may also compute the higher homotopy groups of the bicategory $\cm(\mathfrak{f}, \mathfrak{a})$. Since we have not found them listed in the literature, while having all necessary methods available we present them here. 

\begin{theorem}
\label{ThmClassificationOfButterflies}
The automorphism 2-group $\Aut(\mathfrak{G})$ of each $\mathfrak{G}$ in $\cm(\mathfrak{f}, \mathfrak{a})$ has
\[
\pi_0\Aut(\mathfrak{G}) = H^2(\mathfrak{f}, \mathfrak{a})\text{,}
\]
and, for any automorphism $\mathfrak{k} : \mathfrak{G} \to \mathfrak{G}$ in $\cm(\mathfrak{f}, \mathfrak{a})$, we have
\[
\Aut(\mathfrak{k}) \cong H^1(\mathfrak{f}, \mathfrak{a}).
\]
\end{theorem}

\begin{proof}
Consider an arbitrary automorphism $\mathfrak{k} : \mathfrak{G} \to \mathfrak{G}$  in the bicategory $\cm(\mathfrak{f}, \mathfrak{a})$. 
We first claim that we may choose a section $q$ such that the induced maps $\phi_q$ and $f_q$ from \eqref{phiqfq} are identities.
Indeed, since $\mathfrak{k}$ induces the identity on $\mathfrak{f}$, we have $p(\phi_q-\id_{\mathfrak{g}})=0$; hence, there exists a linear map $\gamma: \mathfrak{g} \to \mathfrak{h}$ such that $\phi_q-\id=t\gamma$. 
From \eqref{equivalence-of-cocycle-data}, the new section $q':= q-i_2\gamma$ has $\phi_{q'}=\id_{\mathfrak{g}}$ and $f_{q'} = f_q + \gamma t$.
We now construct a section $q''$ such that also $f_{q''} = \id_{\mathfrak{h}}$:
Since $\phi_{q'} = r_2 q'= \id_{\mathfrak{g}}$, we have 
\begin{align*}
t( f_{q'}-\id_{\mathfrak{h}})
&=- r_2 i_2 j i_1  - t
\\
&= - r_2(\id - q' r_1)i_1 - r_1 i_1
\\
&= - r_2 i_1 + r_2 q' r_1 i_1 - r_1 i_1 = 0.
 \end{align*}
Moreover, since $\mathfrak{k}$ induces the identity on $\mathfrak{a}$, $f_{q'} - \id_{\mathfrak{h}}$ vanishes on $\mathfrak{a}$.
Hence, there exists   a linear map $\tilde{\gamma}: \mathfrak{g} \to \mathfrak{h}$ such that $\tilde{\gamma} t =f_{q'}-\id_{\mathfrak{h}}$ and $t\tilde{\gamma}=0$. 
Now, again by \eqref{equivalence-of-cocycle-data}, the new section $q'':= q'-i_2\tilde{\gamma}$ produces $\phi_{q''}=\phi_{q'} = \id_{\mathfrak{g}}$ and $f_{q''}=\id_{\mathfrak{h}}$.

Given the claim, we may choose a section $q$ such that $\phi_q$ and $f_q$ are identities.
As, in particular, $\phi_q$ is a Lie algebra homomorphism, \eqref{tofLambda} then implies that the corresponding cocycle $\lambda \in \Alt^2(\mathfrak{g}, \mathfrak{h})$ is $\mathfrak{a}$-valued.
\cref{exchange-of-actions} implies that it vanishes on $t_*\mathfrak{h}$.
Hence $\lambda = \iota_*p^*\xi$ for some $\xi \in \Alt^2(\mathfrak{f}, \mathfrak{a})$.
Since the crossed module is central, \cref{cyclic-identity} implies that $\lambda$ (and hence also $\xi$) is closed. 

If $q$ and $q'$ are two sections such that $\phi_q = \phi_{q'} = \id_{\mathfrak{g}}$ and $f_q = f_{q'} = \id_{\mathfrak{h}}$, then by \cref{equivalence-of-cocycle-data}, their difference $\gamma= q' - q$ is $\mathfrak{a}$-valued and vanishes on $t_*\mathfrak{h}$, hence we have $\gamma = \iota_*p^* \zeta$ for some $\zeta \in \Alt^1(\mathfrak{f}, \mathfrak{a})$.
Again by \cref{equivalence-of-cocycle-data} and the fact that the crossed module is central, we have
\[
\lambda_{q'} = \lambda_q + \delta \gamma.
\]
We therefore conclude that there is a well-defined map 
\begin{equation}
\label{MapPi1}
\pi_1 \cm(\mathfrak{f}, \mathfrak{a}) \to H^2(\mathfrak{f}, \mathfrak{a}),
\qquad \mathfrak{k} \mapsto [\lambda_q],
\end{equation} 
where $q$ is any section such that $\phi_q$ and $f_q$ are identities.

Conversely, given a closed $\xi \in \Alt^2(\mathfrak{f}, \mathfrak{a})$, the triple $(\id_{\mathfrak{g}}, \id_{\mathfrak{h}}, \xi)$ constitutes cocycle data for a butterfly $k_\xi$ whose class in $H^2(\mathfrak{f}, \mathfrak{a})$ associated by \eqref{MapPi1} is precisely $[\xi]$.
Hence \eqref{MapPi1} is surjective.

To see that \eqref{MapPi1} is injective, let now $\xi, \xi' \in \Alt^2_{\mathrm{cl}}(\mathfrak{f}, \mathfrak{a})$ and let $\ell : \mathfrak{k}_{\xi} \to \mathfrak{k}_{\xi'}$ be a 2-isomorphism, given by a linear endomorphism of $\mathfrak{g} \oplus \mathfrak{h}$.
That $\ell$ intertwines the butterfly maps of $\mathfrak{k}_\xi$ and $\mathfrak{k}_{\xi'}$ yields that $\ell$ must be of the form
\[
\ell = \begin{pmatrix} \id & 0 \\ \tilde{\zeta} & \id \end{pmatrix}
\]
with $\tilde{\zeta} : \mathfrak{g} \to \mathfrak{h}$ satisfying $t\tilde{\zeta} = \tilde{\zeta} t = 0$.
In other words, $\tilde{\zeta} = \iota_*p^* \zeta$ for some $\zeta \in \Alt^1(\mathfrak{f}, \mathfrak{a})$.
That $\ell$ must intertwine the $\xi$-Lie bracket with the $\xi'$-Lie bracket on $\mathfrak{g} \oplus \mathfrak{h}$ yields the equation $\xi' - \xi = - \delta \zeta$.
Hence $\xi$ and $\xi'$ define the same class in $H^2(\mathfrak{f}, \mathfrak{a})$.
Conversely, if $\xi$ and $\xi'$ define the same class, then choosing $\zeta$ with $\xi' - \xi = - \delta \zeta$ and defining $\ell$ as above with $\tilde{\zeta} = \iota_* p^* \zeta$ yields a 2-isomorphism $\ell : \mathfrak{k}_{\xi} \to \mathfrak{k}_{\xi'}$.
Since we may modify $\zeta$ by any element of $\Alt^1(\mathfrak{f}, \mathfrak{a}) = H^1(\mathfrak{f}, \mathfrak{a})$, we get that the set of 2-isomorphisms $\mathfrak{k}_{\xi} \to \mathfrak{k}_{\xi'}$ is a torsor for this group.

In particular, setting $\xi = \xi'$ in the above discussion, we obtain the desired identification $\Aut(\mathfrak{k}_\xi) = H^1(\mathfrak{f}, \mathfrak{a})$, independent of $\xi$.
Notice here also that any butterfly is isomorphic to one of the form $\mathfrak{k}_\xi$, hence this finishes the proof.
\end{proof}

\section{Classification of adjustments}

In \cref{SectionForms} we recall  the Lie-algebraic version of the Chern-Weil homomorphism, and describe its kernel as well as its role in the classification of infinitesimal adjustments. In \cref{sec-existence-of-infinitesimal-adjustments} we use the Chern-Weil homomorphism to state and prove one of our main results: the existence of infinitesimal adjustments in dependence of the Kassel-Loday class of the crossed module. Finally, in \cref{covariance-under-butterflies}, we establish a bijection between the sets of adjustments of weakly equivalent crossed modules.

\subsection{The Chern-Weil homomorphism}
\label{SectionForms}

Let $V$ be a topological vector space and $\mathfrak{g}$ a topological Lie algebra.
We write $T(\mathfrak{g}, V)$ for the vector space of $V$-valued continuous bilinear forms $\eta$ on $\mathfrak{g}$ satisfying the condition
\begin{equation}
\label{Tcondition}
\eta([X, Y], Z) + \eta(Y, [X, Z]) = \eta(X, [Y, Z]).
\end{equation}
We have $ \Alt^2_{\mathrm{cl}}(\mathfrak{g}, V) \subset T(\mathfrak{g}, V)$; 
this follows directly from observing that for an anti-symmetric bilinear form $\eta$, the identity \eqref{Tcondition} means that $\eta$ is closed.
\begin{comment}
Nur für's Protokoll, 
\begin{equation*}
T(\mathfrak{g},V) \cong \Alt^1_{\mathrm{cl}}(\mathfrak{g},\Lin(\mathfrak{g},V))\text{,}
\end{equation*}
unter dem Isomorphismus $\mathrm{Bil}(\mathfrak{g},V) \cong \Lin(\mathfrak{g},\Lin(\mathfrak{g},V))$. 
\end{comment}

Moreover, we denote by $\Sym^2(\mathfrak{g}, V)^\ad$ the space of  $\Ad$-invariant symmetric bilinear forms on $\mathfrak{g}$.
Explicitly, such a bilinear form $\beta$ satisfies
\[
\beta([X, Y], Z) = \beta(X, [Y, Z]).
\]
We consider the Chern-Weil homomorphism 
\begin{equation}
\label{ChernWeilHomomorphism}
\Sym^2(\mathfrak{g}, V)^\ad \to \Alt^3_{\mathrm{cl}}(\mathfrak{g}, V), \qquad \beta \mapsto \cw(\beta)
\end{equation}
that sends an $\mathrm{Ad}$-invariant symmetric bilinear form $\beta$ to the  Lie algebra $3$-cocycle $\cw(\beta)$ given by
\begin{equation}
        \label{Associated3form}
\cw(\beta)(X, Y, Z) := \beta([X, Y], Z) = \beta(X, [Y, Z]).
\end{equation}
The following crucial lemma connects the two spaces defined above.

\begin{lemma}
\label{decomposition-lemma}
\label{EquivalentEquationTgV}   
Let $\eta$ be a continuous $V$-valued bilinear form on $\mathfrak{g}$, and let $\eta = \eta^\anti + \eta^\sym$ be its decomposition into its skew-symmetric and its symmetric part.
Then, $\eta\in T(\mathfrak{g}, V)$ if and only if $\eta^\sym$ is $\mathrm{Ad}$-invariant and  
\begin{equation}
\label{deltaetaaplusetashat}
\delta \eta^\anti + \cw({\eta^\sym}) = 0.
\end{equation}
\end{lemma}

\begin{proof}
$(\Longrightarrow)$
Let $\eta \in T(\mathfrak{g}, V)$.
The symmetric part of $\eta$ is $\ad$-invariant by the calculation
\begin{align*}
&\hspace{-2em}2\big(\eta^\sym([X, Y], Z) + \eta^\sym(Y, [X, Z])\big) \\&= \eta([X, Y], Z) + \eta(Z, [X, Y]) + \eta(Y, [X, Z]) + \eta([X, Z], Y)
\\
&=  \eta(X, [Y, Z])  + \eta(X, [Z, Y]) 
\\&= 0.
\end{align*}
The anti-symmetric part $\eta^\anti$  satisfies
\begin{align*}
-2 \delta \eta^\anti(X, Y, Z) 
&= 2\big(\eta^\anti([X, Y], Z) - \eta^\anti([X, Z], Y) + \eta^\anti([Y, Z], X)\big)
\\
&=2\big(\eta^\anti([X, Y], Z) + \eta^\anti(Y, [X, Z]) - \eta^\anti(X, [Y, Z])\big)
\\
&= - \eta(Z, [X, Y]) - \eta([X, Z], Y) + \eta([Y, Z], X)
\\
&= - \eta(X, [Z, Y]) + \eta([Y, Z], X)
\\
&= \eta(X, [Y, Z]) + \eta([Y, Z], X) 
\\&= 2 \eta^\sym(X, [Y, Z]).
\end{align*}
This is the claimed equation.

$(\Longleftarrow)$ We check the relevant equation:
\begin{align*}
\eta([X, Y], Z) + \eta(Y, [X, Z]) 
&=\eta^\sym([X, Y], Z) + \eta^\sym(Y, [X, Z])+\eta^\anti([X, Y], Z) + \eta^\anti(Y, [X, Z])
\\&= \eta^\sym(X, [Y, Z]) + \eta^\sym([Y, X], Z)-\delta\eta^\anti(X, Y, Z) - \eta^\anti([Y, Z], X)
\\&= \eta^\sym(X, [Y, Z]) - \eta^\sym([X, Y], Z)+ \cw( \eta^\sym)(X, Y, Z) - \eta^\anti([Y, Z], X)
\\&= \eta^\sym(X, [Y, Z]) + \eta^\anti(X, [Y, Z])
\\&= \eta(X, [Y, Z])\text{.}
\end{align*} 
This shows that $\eta\in T(\mathfrak{g}, V)$. 
\end{proof}

We obtain an exact sequence
\begin{equation}
\label{FunnyFourTermSequence}
0\longrightarrow\Alt^2_{\mathrm{cl}} (\mathfrak{g}, V) \longrightarrow T(\mathfrak{g}, V) \longrightarrow \Sym^2(\mathfrak{g}, V)^\ad \longrightarrow H^3(\mathfrak{g}, V)
\end{equation}
where the third map sends an element of $T(\mathfrak{g}, V)$ to its symmetric part.
Clearly, an element of the kernel of this map is anti-symmetric, and Lemma~\ref{decomposition-lemma} shows that it is  closed.
The fourth map is the Chern-Weil-Homomorphism \eqref{ChernWeilHomomorphism}, whose kernel equals the image of the third map by \cref{decomposition-lemma}; also see \cite[Prop.~7.2]{NeebWagemannCurrent}.

\begin{lemma}
\label{SymmetricPart}
If $\eta$ is an infinitesimal adjustment, then there exists a unique $B \in \Sym^2(\mathfrak{f}, \mathfrak{h})^\ad$ such that
\[
\eta^\sym = -p^*B,
\]
where $p : \mathfrak{g} \to \mathfrak{f}$ is the projection.
Moreover, if $\eta$ is adapted to some section $s$, then $B$ takes values in $\mathfrak{a}$.
\end{lemma}

\begin{proof}
By \cref{decomposition-lemma}, the symmetric part $\eta^\sym$ is $\mathrm{Ad}$-invariant, and the identities \eqref{EtaIdentities-2} and \eqref{EtaIdentities-1} show that $\eta^\sym$ vanishes as soon as one argument is in the image of $t_{*}$.
%
\begin{comment}
\begin{align*}
\eta^\sym(t_*x, Y) 
&= \frac{1}{2}\big(\eta(t_*x, Y) + \eta(Y, t_*x)\big)
= \frac{1}{2}\big(- \alpha_*(Y, x) + \alpha_*(Y, x)\big)  = 0.
\end{align*}
\end{comment}
%
This implies the claim.
The $s$-adaptedness of $\eta$ implies $t_* \eta$ is skew-symmetric, i.e., $\eta^\sym$ must take values in $\mathfrak{a}$.
\end{proof}

\begin{definition}
Let $\eta$ be an infinitesimal adjustment, adapted to some section $s$.
We call the unique $B \in \Sym^2(\mathfrak{f}, \mathfrak{a})^{\ad}$ such that $\eta^{\mathrm{s}} = - p^*B$ the \emph{adjusted Kassel-Loday class} of $\eta$ and denote it by $\KL^{\mathrm{adj}}(\mathfrak{G}, \eta)$.
\end{definition}

By the above lemma, sending an adjustment to the corresponding adjusted Kassel-Loday class  yields a commutative diagram
\begin{equation}
\label{DiagramAdjToSym}
\xymatrix{
\Adj^{s}(\mathfrak{G}) \ar[r] \ar[d]_{\KL^{\mathrm{adj}}} & \Adj(\mathfrak{G}) \ar[d]^{\KL^{\mathrm{adj}}} \\ \mathrm{Sym}^2(\mathfrak{f}, \mathfrak{a})^{\mathrm{ad}} \ar[r] & \mathrm{Sym}^2(\mathfrak{f}, \mathfrak{h})^{\mathrm{ad}}. }
\end{equation}

\begin{proposition}
\label{affine-space-structure-on-adjustments}
For any central crossed module $\mathfrak{G}$ of Lie algebras, 
$\Adj(\mathfrak{G})$ is an affine space over $T(\mathfrak{f}, \mathfrak{h})$, and for any section $s$, $\Adj^{s}(\mathfrak{G})$ is an affine space over $T(\mathfrak{f}, \mathfrak{a})$.
Moreover, the fibres of the vertical maps in the diagram \eqref{DiagramAdjToSym} are affine spaces over $\Alt^2_{\mathrm{cl}}(\mathfrak{f}, \mathfrak{a})$, respectively $\Alt^2_{\mathrm{cl}}(\mathfrak{f}, \mathfrak{h})$.
\end{proposition}

\begin{proof}
If $\eta$ is an infinitesimal adjustment, and $\beta\in T(\mathfrak{f}, \mathfrak{h})$, then $\eta':=\eta+p^{*}\beta\in T(\mathfrak{g}, \mathfrak{h})$, where $p: \mathfrak{g} \to \mathfrak{f}$ is the projection, is again an infinitesimal adjustment. 
Indeed, both $\eta$ and $p^*\beta$ satisfy the linear condition \eqref{EtaIdentities-0}, hence so does $\eta'$, while the conditions \eqref{EtaIdentities-2} \& \eqref{EtaIdentities-1} still hold as $p^*\beta = \eta'- \eta$ vanishes on $t_*\mathfrak{h}$.
If $\eta$ is adapted to $s$, and $\beta\in T(\mathfrak{f}, \mathfrak{a})$, then the adjustment $\eta':=\eta+p^{*}\beta\in T(\mathfrak{g}, \mathfrak{a})$ is adapted to $s$ as well, as $t_*\eta' = t_*\eta$. 

If $\eta$ and $\eta'$ are infinitesimal adjustments, then $\tilde{\beta}:=\eta'-\eta$ also satisfies \eqref{EtaIdentities-0}  and therefore is contained in
$T(\mathfrak{g}, \mathfrak{h})$.
Since both $\eta$ and $\eta'$ satisfy \eqref{EtaIdentities-2} and \eqref{EtaIdentities-1}, we get that $\tilde{\beta}(t_{*}x, Y)=\tilde{\beta}(X, t_{*}y)=0$, and hence $\tilde{\beta}$ descends to an element $\beta$ of $T(\mathfrak{f}, \mathfrak{h})$. 
If $\eta$ and $\eta'$ are both adapted to the same $s$, then $t_{*}\tilde{\beta}= 0$ and so $\beta$ takes values in $\mathfrak{a}$. 

The claim on the fibres follows from the  exact sequence \eqref{FunnyFourTermSequence}.
\end{proof}

\begin{comment}
Let $s$ and $s'$ be two different sections and let $\eta$ and $\eta'$ be infinitesimal adjustments adapted to $s$, respectively $s'$.
By \cref{affine-space-structure-on-adjustments}, we have $\eta - \eta' = p^*\theta$ for some $\theta \in  \Alt^2_{\mathrm{cl}}(\mathfrak{f}, \mathfrak{h})$.
We have
\[
t_*\bigl(\eta(X, Y) - \eta'(X, Y)\bigr) = (\rho_s - \rho_{s'})([X, Y]).
\]
Since
\[
(\rho_{s'} - \rho_{s})(t_*x) = - (s'p_* - sp_*)t_* x = 0,
\]
we see that $\rho_{s'} - \rho_{s} = p^*\xi_{s, s'}$ for a unique $\xi_{s, s'} \in \Alt^1(\mathfrak{f}, \mathfrak{g})$, which satisfies
\[
t_*\theta = \delta \xi_{s, s'}.
\]
\end{comment}

\subsection{Existence of infinitesimal adjustments}

\label{sec-existence-of-infinitesimal-adjustments}

\begin{theorem}
\label{existence-of-infinitesimal-adjustments}
A central crossed module $\mathfrak{G}=(\mathfrak{h} \stackrel{t}{\to} \mathfrak{g} \stackrel{\alpha}{\to} \Der(\mathfrak{h}))$ of Lie algebras admits an infinitesimal adjustment if and only if its Kassel-Loday class $\KL(\mathfrak{G})$  lies in the image of the  Chern-Weil homomorphism
\[
\cw:  \Sym^2(\mathfrak{f}, \mathfrak{a})^\ad  \to  H^3(\mathfrak{f}, \mathfrak{a}).
\]
More precisely, any adapted adjustment  $\eta$ on $\mathfrak{G}$ with adjusted Kassel-Loday class $\KL^{\mathrm{adj}}(\mathfrak{G}, \eta)$ satisfies
\[
\bigl[\cw\bigr(\KL^{\mathrm{adj}}(\mathfrak{G}, \eta)\bigr)\bigr] = \KL(\mathfrak{G}).
\]
Conversely, given any section $s$ of $\mathfrak{G}$ and any $B \in \Sym^2(\mathfrak{f}, \mathfrak{a})^{\ad}$ such that $[\cw(B)] = \KL(\mathfrak{G})$, there exists an $s$-adapted adjustment $\eta$ such that $\KL^{\mathrm{adj}}(\mathfrak{G}, \eta) = B$.
\end{theorem}
\begin{proof}
$(\Longleftarrow)$
Assume that there exists $B \in \Sym^2(\mathfrak{f}, \mathfrak{a})^\ad$ with $[\cw(B)] = \KL(\mathfrak{G})$.
By Lemma~\ref{LemmaChangeSplitting}, this means that for any half splitting $u$, the difference $\cw(B) - C_u$ is exact. 
In other words, there exists  $\xi \in \Alt^2 (\mathfrak{f}, \mathfrak{a})$ such that 
\[
\cw(B) - C_u =  \delta \xi.
\]
Set now
\[
\eta := \omega_u + p^*\xi  - p^*B.
\]
Since $\omega_u$  satisfies the identities  \eqref{EtaIdentities-2} and \eqref{EtaIdentities-1} by Lemma~\ref{LemmaAdjustmentIdentities2}, and the difference $\eta - \omega_u$ vanishes on $t_*\mathfrak{h}$, 
$\eta$ also satisfies these identities. 
It remains to show that $\eta\in T(\mathfrak{g}, \mathfrak{h})$. 
Since $\omega_u$ and $\xi$ are skew-symmetric, we have $\eta^\sym=-p^*B$ and $\eta^\anti=\omega_u+p^*\xi$. 
By assumption, $B$ is $\mathrm{Ad}$-invariant and 
\[
\delta\eta^\anti+\cw(\eta^\sym)= \delta \omega_u +p^*\delta\xi -p^*\cw(B) = p^* (C_u - \cw(B) + \delta \xi) =0.
\]
By \cref{EquivalentEquationTgV}, this shows that $\eta\in T(\mathfrak{g}, \mathfrak{h})$.   
The infinitesimal adjustment $\eta$ is  adapted to the section $s_u$ determined by $u$, by Lemma~\ref{LemmaOmegaUadapted}. Since any section $s$ is of the form $s=s_u$ for some half splitting $u$, this proves the assertion about adaptedness.  
Finally, the equality $\KL^{\mathrm{adj}}(\mathfrak{G}, \eta) = B$ holds by construction.

$(\Longrightarrow)$
Let $\eta$ be an infinitesimal adjustment.
Choose a  half splitting $u : \mathfrak{g} \to \mathfrak{h}$ and set $\beta := \eta - \omega_u$.
Because $\eta=\beta+\omega_u \in T(\mathfrak{g}, \mathfrak{h})$, \cref{decomposition-lemma} yields that
\begin{equation}
\label{StructureEquation}
\delta \omega_u +  \delta \beta^\anti + \cw(\beta^\sym) = 0.
\end{equation}
The symmetric part $\beta^\sym=\eta^\sym$ of $\beta$ vanishes on $t_*\mathfrak{h}$ by \cref{SymmetricPart}. 
The skew-symmetric part vanishes, too, by \cref{LemmaAdjustmentIdentities2}, as
\begin{equation*}
2\beta^\anti(t_*x, Y)=2\eta^\anti(t_*x, Y)-2\omega_u(t_*x, Y)=\eta(t_*x, Y)-\eta(Y, t_*x)+2 \alpha_*(Y, x)=0
\end{equation*}
In general, $\beta^\sym$ and $\beta^\anti$ will  be $\mathfrak{h}$-valued.
We set $\tilde{\beta} := \beta - ut_*\beta$, which is now $\mathfrak{a}$-valued. 
(We remark that, if $\eta$ is adapted to the section determined by $u$, then $\tilde\beta=\beta$ and this step is unnecessary.)
Now because $\delta \omega_u$ is $\mathfrak{a}$-valued, we get
\begin{align*}
\delta \omega_u + \delta \tilde{\beta}^\anti + \cw(\tilde{\beta}^\sym) 
&= \delta \omega_u + (\id - ut) \big(\delta \beta^\anti + \cw(\beta^\sym)\big) 
\\
&= (\id - ut) \big(\delta \omega_u+ \cw({\beta^\sym}) +  \delta \beta^\anti\big) = 0
\end{align*}
using \eqref{StructureEquation}, so $\delta \omega_u + \cw({\tilde{\beta}^\sym})$ is exact. 
Since both $\delta \omega_u$ and $\cw({\tilde{\beta}^\sym})$ vanish on $t_*\mathfrak{h}$, this relation descends to $\mathfrak{f}$, hence the Kassel-Loday class lies in the image of the Chern-Weil homomorphism. If $\eta$ was adapted, without loss of generality to the section $s$ determined by $u$, then $\eta^{\mathrm{s}} = \beta = \tilde{\beta} = - \KL^{\mathrm{adj}}(\mathfrak{G}, \eta)$, so
\[
\bigl[\cw\bigl(\KL^{\mathrm{adj}}(\mathfrak{G}, \eta)\bigr)\bigr] = [\delta\omega_u] = \KL(\mathfrak{G}),
\]
as claimed.
\end{proof}

\begin{remark}
\label{construction-of-adjustments}
The proof of \cref{existence-of-infinitesimal-adjustments} reveals that, in order to construct an infinitesimal adjustment, one may choose:
\begin{enumerate}[(1)]
\item 
a half splitting $u:\mathfrak{g} \to \mathfrak{h}$, 
\item
a bilinear form $B \in \Sym^2(\mathfrak{f}, \mathfrak{a})^\ad $, 
and
\item
a cochain $\xi \in \Alt^2 (\mathfrak{f}, \mathfrak{a})$ such that $\cw(B) - C_u   = \delta \xi$
\end{enumerate}
and then set $\eta := \omega_u + p^*\xi - p^*B$; this yields an infinitesimal adjustment  adapted to $s_u$. 
\end{remark}

\begin{remark}
Passing the constructions in the proof of   \cref{existence-of-infinitesimal-adjustments}  back and forth, we see that if $\eta$ is an infinitesimal adjustment, and $u$ is any half splitting of $\mathfrak{G}$, then
\begin{equation*}
\mathrm{adpt}^{u}(\eta) :=  \eta+ut_*( \omega_u-\eta)
\end{equation*} 
is another infinitesimal adjustment and adapted to $s_u$.
 If $\eta$ was already adapted to $s_u$, then $\mathrm{adpt}^{u}(\eta)=\eta$, as one can see from \cref{LemmaOmegaUadapted}.
This shows that $\mathrm{adpt}^{u}$
is an idempotent projection in $\Adj(\mathfrak{G})$ onto $\Adj^{s_{u}}(\mathfrak{G})$.
\end{remark}

\subsection{Covariance under butterflies}

\label{covariance-under-butterflies}

From \cref{existence-of-infinitesimal-adjustments} and \cref{classification-of-crossed-modules} it is already clear that if a crossed module of Lie algebras $\mathfrak{G}_1$ admits  infinitesimal adjustments, then any other crossed module $\mathfrak{G}_2$ related to $\mathfrak{G}_1$ by an invertible butterfly also admits infinitesimal adjustments.

Here we want to establish the stronger result that given two central crossed modules of Lie algebras,  $\mathfrak{G}_i=(\mathfrak{h}_i \stackrel{t_i}{\to} \mathfrak{g}_i \stackrel{\alpha_i}{\to} \Der(\mathfrak{h}_i))$, $i=1,2$, and  any invertible butterfly $\mathfrak{k}:\mathfrak{G}_1 \to \mathfrak{G}_2$, one may construct a bijection
\begin{equation*}
\Adj^{s_1}(\mathfrak{G}_1) \to \Adj^{s_2}(\mathfrak{G}_2)\text{, }
\end{equation*}
for any choice of sections  $s_i: \mathfrak{f}_i \to \mathfrak{g}_i$ in $\mathfrak{G}_i$. 
This construction depends on the choice of a neat section $q$ of the butterfly $\mathfrak{k}$ (see \cref{splittings-assumption}).
Throughout, we write $\mathfrak{f}_i := \pi_0(\mathfrak{G}_i)$ and $\mathfrak{a}_i:= \pi_1(\mathfrak{G}_i)$, as well as $F:=\pi_1\mathfrak{k}:\mathfrak{a}_1 \to \mathfrak{a}_2$ and $\Phi := \pi_0\mathfrak{k}: \mathfrak{f}_1 \to \mathfrak{f}_2$. As before, we also denote by $f_q: \mathfrak{h}_1 \to \mathfrak{h}_2$ and $\phi_q : \mathfrak{g}_1 \to \mathfrak{g}_2$ the linear maps determined by $q$, inducing $F$ and $\Phi$, respectively.

Suppose $\eta_1$ is an infinitesimal adjustment on $\mathfrak{G}_1$ and adapted to $s_1$. We choose  half splittings $u_1$  and $u_2$ extending the given sections $s_1$ and $s_2$. 
Following the proof of \cref{existence-of-infinitesimal-adjustments}, there exists a unique $\mathfrak{a}_1$-valued bilinear form $\beta$ on $ \mathfrak{f}_1$ such that $p_1^{*}\beta := \eta_1 - \omega_{u_1}$, and we split this into $\beta=\beta^\sym+\beta^\anti$. 
Then, from \eqref{StructureEquation}, we get
\begin{equation*}
 C_{u_1}=-\cw(\beta^\sym)-\delta\beta^\anti
\end{equation*}   
in $\Alt^3(\mathfrak{f}_1, \mathfrak{a}_1)$. 
Next, we consider the cochain $R_q\in \Alt^2(\mathfrak{f}_1, \mathfrak{h}_2)$ descended from \eqref{cochain-of-butterfly}, which has values in $\mathfrak{a}_2$ by \cref{splittings-assumption} and satisfies
\begin{equation*}
\Phi^{*}{C}_{u_2} =\delta R_q+F_{*}({C}_{u_1})=-\cw(F_{*}(\beta^{\sym})) - \delta\big(F_{*}(\beta^{\anti})-R_q\big) \text{.}
\end{equation*}
Further following the  proof of \cref{existence-of-infinitesimal-adjustments}, 
\begin{equation}
\label{eta-2-definition}
\eta_2 := \omega_{u_2}+p_2^{*}(\Phi^{-1})^{*}(F_{*}({\beta})-{R_q})
\end{equation}
is an infinitesimal adjustment on $\mathfrak{G}_2$ and adapted to $s_2$.

\begin{lemma}
\label{butterfly-map-depends-only-on-idempotent}
The adjustment $\eta_2$  defined in \cref{eta-2-definition} is independent of the choices of the half splittings $u_1$ and $u_2$. 
\end{lemma}

\begin{proof}
Let $u_1'$ and $u_2'$ be other half splittings inducing the same sections ${s_1}$ and ${s_2}$, respectively. If $p_1^{*}\beta = \eta_1 - \omega_{u_1}$ and $p_1^{*}\beta' = \eta_1-\omega_{u_1'}$ then $p_1^{*}(\beta'-\beta )=\omega_{u_1}-\omega_{u_1'}$.
Moreover, if $R_q'$ is the cochain determined by $u_1'$ and $u_2'$, we see from \cref{cochain-of-butterfly} that
\begin{equation*}
p_1^*R_q'=p_1^*R_q+\phi_q^{*}(\omega_{u'_2}-\omega_{u_2})-(f_q)_{*}(\omega_{u_1'}-\omega_{u_1})\text{.}
\end{equation*}
The difference $\omega_{u'_2}-\omega_{u_2}$ descends to $\mathfrak{f}_2$ by \cref{LemmaChangeSplitting}; i.e., $\omega_{u_2'}-\omega_{u_2}=p_2^{*}\tilde\omega_2$ for a unique $\tilde\omega_2\in \Alt^2(\mathfrak{f}_2,\mathfrak{h}_2)$, so that
\begin{equation*}
R_q'=R_q+\Phi^{*}\tilde\omega_2+F_{*}(\beta'-\beta)\text{.}
\end{equation*} 
Now, the adjustment $\eta_2'$ determined by $\beta'$ and $R_q'$ is
\begin{align*}
\eta_2'
&=\omega_{u'_2}+p_2^{*}(\Phi^{-1})^{*}(F_{*}{(\beta')}-{R'_q}) 
\\&= \omega_{u'_2}+p_2^{*}(\Phi^{-1})^{*}(F_{*}(\beta)-\Phi^{*}\tilde\omega_2-R_q)
\\&= \omega_{u_2}+p_2^{*}(\Phi^{-1})^{*}(F_{*}({\beta})-{R_q})
\\&=\eta_2\text{.}\qedhere
\end{align*}
\end{proof}

By \cref{butterfly-map-depends-only-on-idempotent}, we have established 
a map 
\begin{equation}
\label{covariance-under-butterflies-map}
\Adj^{q}(\mathfrak{k}):\Adj^{s_1}(\mathfrak{G}_1) \to \Adj^{s_2}(\mathfrak{G}_2)
\end{equation}
for every invertible butterfly $\mathfrak{k}:\mathfrak{G}_1 \to \mathfrak{G}_2$, arbitrary sections $s_i$ in $\mathfrak{G}_i$ and a neat section $q$ in $\mathfrak{k}$.
Since the construction of the map $\Adj^{q}(\mathfrak{k})$ is quite involved, we offer with the next result a simple way of checking if a given infinitesimal adjustment $\eta_2$ on $\mathfrak{G}_2$ is the image of $\eta_1$ under $\Adj^{q}(\mathfrak{k})$.

\begin{proposition}
\label{adjustment-respect}
Let $\mathfrak{k}:\mathfrak{G}_1 \to \mathfrak{G}_2$ be an invertible butterfly, let $s_i$ be  arbitrary sections in $\mathfrak{G}_i$ and let  $q$ be  a neat section in $\mathfrak{k}$, inducing  maps $\phi_q: \mathfrak{g}_1 \to \mathfrak{g}_2$ and $f_q: \mathfrak{h}_1 \to \mathfrak{h}_2$ and the cochain $\lambda_q\in \Alt^2(\mathfrak{g}_1,\mathfrak{h}_2)$.   Let $\eta_i \in \mathrm{\Adj}^{s_i}(\mathfrak{G}_i)$ be adapted infinitesimal adjustments. Then, 
\begin{equation*}
\eta_2=\Adj^{q}(\mathfrak{k})(\eta_1) \iff \phi_q^{*}\eta_2=(f_q)_{*}(\eta_1)+\lambda_q\text{.}
\end{equation*} 
\end{proposition}

\begin{proof}
We evaluate the left hand side, using the construction of the map $\Adj^{q_0}(\mathfrak{k})$. Thus, we choose splittings $u_1$ and $u_2$ extending the sections $s_1$ and $s_2$, respectively, and consider $\beta\in \Alt^2(\mathfrak{f}_1,\mathfrak{a}_1)$ defined by $p_1^{*}\beta := \eta_1 - \omega_{u_1}$. 
From \eqref{eta-2-definition} and \eqref{cochain-of-butterfly}, we get
\begin{align*}
\phi_q^*\eta_2 
&= \phi_q^*\omega_{u_2}+\phi_q^*p_2^{*}(\Phi^{-1})^{*}((f_q)_{*}({\beta})-{R_q})
\\
&= \phi_q^*\omega_{u_2} + p_1^*(f_q)_*\beta  - p_1^* R_q
\\
&= (p_1)^*(f_q)_*\beta + (f_q)_*\omega_{u_1} + \lambda_q
\\
&= (f_q)_*\eta_1 + \lambda_q.
\end{align*}
Conversely, suppose $(f_q)_{*}(\eta_1)+\lambda_q=\phi_q^{*}\eta_2$ holds. Then, we get: 
\begin{align*}
\eta_2&=p_2^{*}s_2^{*}\eta_2+u_2^{*}t_2^{*}\eta_2 && \text{since $s_2p_2+t_2u_2=\id_{\mathfrak{h}_2}$}
\\&=p_2^{*}(\Phi^{-1})^{*}s_1^{*}\phi_q^{*}\eta_2-\delta u_2 && \text{by \cref{EtaIdentities-2} and neatness }
\\&=p_2^{*}(\Phi^{-1})^{*}s_1^{*}((f_q)_{*}(\eta_1)+\lambda_q)-\delta u_2
\\&=p_2^{*}(\Phi^{-1})^{*}s_1^{*}((f_q)_{*}(p_1^{*}\beta+\omega_{u_1})+\lambda_q)-\delta u_2
\\&=p_2^{*}(\Phi^{-1})^{*}(f_q)_{*}(\beta)+p_2^{*}(\Phi^{-1})^{*}s_1^{*}(\phi_q^{*}\omega_{u_2}-p_1^{*}R_{q})-\delta u_2
\\&=p_2^{*}(\Phi^{-1})^{*}((f_q)_{*}(\beta)-R_{q})+p_2^{*}s_2^{*}\omega_{u_2}-\delta u_2 &&\text{again using neatness}
\\&=\omega_{u_2}+p_2^{*}(\Phi^{-1})^{*}((f_q)_{*}(\beta)-R_{q_0})-u_2^{*}t_2^{*}\omega_{u_2}-\delta u_2 
\\&= \Adj^{q_0}(\mathfrak{k})(\eta_1)&&\text{by \cref{LemmaAdjustmentIdentities2}.}
\end{align*}
This completes the proof. 
\end{proof}

\begin{remark}
\label{strict-intertwiners-adjustments}
\cref{adjustment-respect} applies, in particular, to strict intertwiners $(\phi, f): \mathfrak{G}_1 \to \mathfrak{G}_2$ (see \cref{strict-intertwiners}). A strict intertwiner induces a butterfly $\mathfrak{k}$ with a section $q$ inducing the given maps, $\phi=\phi_q$ and $f=f_q$, whereas $\lambda_q=0$. Thus, if the strict intertwiner is a weak equivalence, we have 
\begin{equation*}
\Adj^{q}(\mathfrak{k})(\eta_1)=\eta_2 \iff f_{*}(\eta_1) =\phi^{*}\eta_2
\end{equation*} 
 for infinitesimal adjustments   $\eta_i \in \mathrm{\Adj}^{s_i}(\mathfrak{G}_i)$ adapted to any sections $s_i$ in $\mathfrak{G}_i$ such that $q$ is neat. We remark the latter condition can be achieved by taking $s_1$ arbitrary and putting $s_2:= \phi s_1\Phi^{-1}$.  
\end{remark}

Next we recall from \cref{affine-space-structure-on-adjustments} that $\Adj^{s_1}(\mathfrak{G}_1)$ and $\Adj^{s_2}(\mathfrak{G}_2)$ are affine spaces over $T(\mathfrak{f}_1, \mathfrak{a}_1)$ and $T(\mathfrak{f}_2, \mathfrak{a}_2)$, respectively. We consider the linear isomorphism 
\begin{equation*}
\varphi_{\mathfrak{k}}: \mathrm{Lin}(\mathfrak{f}_1 \otimes \mathfrak{f}_1, \mathfrak{a}_1) \to \mathrm{Lin}(\mathfrak{f}_2 \otimes \mathfrak{f}_2, \mathfrak{a}_2)
\end{equation*}
given by $\eta \mapsto F_{*}(\Phi^{-1})^{*}\eta$, and notice that it restricts to  isomorphisms $T(\mathfrak{f}_1, \mathfrak{a}_1)\cong T(\mathfrak{f}_2, \mathfrak{a}_2)$ and $\Sym^2(\mathfrak{f}_1, \mathfrak{a}_1)^{\ad} \cong \Sym^2(\mathfrak{f}_2, \mathfrak{a}_2)^{\ad}$. 

\begin{proposition}
\label{affine-linear-map}
The map \eqref{covariance-under-butterflies-map} is affine along $\varphi_{\mathfrak{k}}$, and fits into the commutative diagram 
\begin{equation*}
\xymatrix@C=3em{\Adj^{s_1}(\mathfrak{G}_1) \ar[r]^{\Adj^{q}(\mathfrak{k})} \ar[d]_{\KL^{\mathrm{adj}}} & \Adj^{s_2}(\mathfrak{G}_2) \ar[d]^{\KL^{\mathrm{adj}}} \\ \Sym^2(\mathfrak{f}_1, \mathfrak{a}_1) \ar[r]_{\varphi_{\mathfrak{k}}} & \Sym^2(\mathfrak{f}_2, \mathfrak{a}_2)\text{.}}
\end{equation*} 
\end{proposition}

\begin{proof}
Suppose $\eta_2=\Adj^q(\mathfrak{k})(\eta_1)$; thus $(f_q)_{*}(\eta_1)+\lambda_q =\phi_q^{*}\eta_2$ by \cref{adjustment-respect}. For $\rho \in T(\mathfrak{f}_1, \mathfrak{a}_1)$, we have
\begin{equation*}
(f_q)_{*}(\eta_1 + p_1^{*}\rho)+\lambda_q= \phi_q^{*}\eta_2 + p_1^{*}(f_q)_{*}\rho=\phi_q^{*}(\eta_2+p_2^{*}(\Phi^{-1})^{*}(f_q)_{*}\rho)=\phi_q^{*}(\eta_2+\varphi_{\mathfrak{k}}(\rho))\text{;}
\end{equation*}
hence, again by \cref{adjustment-respect}, $\Adj^{q}(\mathfrak{k})(\eta_1 + p_1^{*}\rho) =\Adj^{q}(\mathfrak{k})(\eta_1)+\varphi_{\mathfrak{k}}(\rho)$. 
\begin{comment}
Old proof not using \cref{adjustment-respect}:
For $\rho \in T(\mathfrak{f}_1, \mathfrak{a}_1)$
and $\eta_1'=\eta_1+p_1^{*}\rho$ we obtain $p_1^{*}\beta' = \eta_1'-\omega_{u_1}=p_1^{*}\beta+p_1^{*}\rho$; hence, 
\begin{align*}
\Adj^{q}(\mathfrak{k})(\eta_1 + p_1^{*}\rho) &=  \omega_{u_2}+p_2^{*}(\Phi^{-1})^{*}(F(\beta')+R_q)
\\&=\omega_{u_2}+p_2^{*}(\Phi^{-1})^{*}(F(\beta)+R_q)+p_2^{*}(\Phi^{-1})^{*}F_{*}\rho
\\&= \Adj^{q}(\mathfrak{k})(\eta_1)+\varphi_{\mathfrak{k}}(\rho)\text{.}
\end{align*}
\end{comment}
This shows that $\Adj^{q}(\mathfrak{k})$ is affine along $\varphi_{\mathfrak{k}}$. 
Let $B := \KL^{\mathrm{adj}}(\mathfrak{G}_1,\eta_1)$,  i.e., $\eta_1^{\sym}=-p_1^{*}B$. Thus, upon writing $p_1^{*}\beta = \eta_1 - \omega_{u_1}$, we find $p_1^{*}\beta^{\sym}=\eta_1^{\sym}$ and hence $B=-\beta^\sym$. Since $\omega_{u_2}$ and $R_q$ are skew-symmetric, 
we obtain\begin{equation*}
\eta_2^{\sym}=-p_2^{*}(\Phi^{-1})^{*}(F_{*}B)=-p_2^{*}\varphi_{\mathfrak{k}}(B)\text{.}
\end{equation*} 
This shows that $\varphi_{\mathfrak{k}}(B)=\KL^{\mathrm{adj}}(\mathfrak{G}_2,\eta_2)$.  
\end{proof}

Because $\varphi_{\mathfrak{k}}$ restricts to an isomorphism $T(\mathfrak{f}_1, \mathfrak{a}_1)\cong T(\mathfrak{f}_2, \mathfrak{a}_2)$, we obtain:

\begin{corollary}
\label{bijectivity-adjustment-transfer}
Suppose $\mathfrak{k}:\mathfrak{G}_1 \to \mathfrak{G}_2$ is an invertible butterfly. For any choice of sections $s_i$ in $\mathfrak{G}_i$ and neat section $q$ in $\mathfrak{k}$, 
the map
\begin{equation*}
\Adj^q(\mathfrak{k}): \Adj^{s_1}(\mathfrak{G}_1) \to \Adj^{s_2}(\mathfrak{G}_2)
\end{equation*}
is a bijection.
\end{corollary}

Reducing this to strict intertwiners, and using \cref{strict-intertwiners-adjustments} we obtain:
\begin{corollary}
\label{adjustemnts-and-strict-intertwiners}
Suppose $(\phi, f): \mathfrak{G}_1 \to \mathfrak{G}_2$ is a strict intertwiner and a weak equivalence, and suppose $s_1$ and $s_2$ are sections satisfying $s_2\Phi=\phi s_1$. Then, there is a unique bijection
\begin{equation*}
\Adj^{s_1}(\mathfrak{G}_1) \cong \Adj^{s_2}(\mathfrak{G}_2)
\end{equation*} 
under which infinitesimal adjustments $\eta_i$ correspond to each other if and only if $f_{*}\eta_1=\phi^{*}\eta_2$. 
\end{corollary}

\begin{remark}
\label{butterfly-map-dependence-on-splitting}
The map $\Adj^{q}(\mathfrak{k})$ depends on the choice of the section $q$. Considering another section, $q' = q+i_2\gamma$, for a linear map $\gamma: \mathfrak{g}_1 \to \mathfrak{h}_2$, and the corresponding changes of $\phi_q$, $f_q$, and $\lambda_q$ described in \cref{equivalence-of-cocycle-data}, we obtain 
\begin{equation*}
R_{q'}-R_q =-\delta(d s_1)\text{.}
\end{equation*} 
\begin{comment}
Indeed, 
\begin{align*}
R_{q'}'(X, Y)&=\omega_{u_2}(\phi'(X), \phi'(Y)) -f'(\omega_{u_1}(X, Y))-\lambda_{j'}(X, Y)
\\&=\omega_{u_2}(\phi(X)+t_2(d(X)), \phi(Y)+t_2(d(Y)))
\\&\quad  -f(\omega_{u_1}(X, Y))+d(t_1(\omega_{u_1}(X, Y)))
\\&\quad -\lambda_{j}(X, Y)-\alpha_2(\phi (X) , d(Y))+\alpha_2(\phi(Y), d(X))-[d(X), d(Y)]-d([X, Y])
\\&=\omega_{u_2}(\phi(X), \phi(Y))+\alpha_2(\phi(X), d(Y))-\alpha_2(\phi(Y), d(X))+[d(X), d(Y)]
\\&\quad  -f(\omega_{u_1}(X, Y))+d(t_1(u_1([X, Y])))
\\&\quad -\lambda_{j}(X, Y)-\alpha_2(\phi (X) , d(Y))+\alpha_2(\phi(Y), d(X))-[d(X), d(Y)]-d([X, Y])
\\&=\omega_{u_2}(\phi(X), \phi(Y)) -f(\omega_{u_1}(X, Y))-\lambda_{j}(X, Y) +d(t_1(u_1([X, Y]))-[X, Y])
\\&=  R_q'(X, Y)+d(t_1(u_1([X, Y]))-[X, Y])\text{.}
\end{align*}
In other words, 
\begin{equation*}
R'_{q'}-R'_q =- \delta(d  \rho_1^{\perp}) =-\delta(ds_1p_1) =-p_1^{*}\delta(d s_1) \text{.}
\end{equation*}
\end{comment}
This shows that
\begin{equation*}
\Adj^{q'}(\mathfrak{k})-\Adj^{q}(\mathfrak{k}) =p_2^{*}(\Phi^{-1})^{*}(- R_{q'}+R_q)=p_2^{*}(\Phi^{-1})^{*}\delta(d  s_1).
\end{equation*}
\begin{comment}
The calculation without the assumption of \cref{splittings-assumption} is here, for completeness:
\begin{align*}
\tilde R_{j'}(X, Y)&=(1-u_2t_2)\big (\omega_{u_2}(\phi'(X), \phi'(Y)) -f'(\omega_{u_1}(X, Y))-\lambda_{j'}(X, Y)\big )
\\&=(1-u_2t_2)\big (\omega_{u_2}(\phi(X)+t_2(d(X)), \phi(Y)+t_2(d(Y)))
\\&\quad  -f(\omega_{u_1}(X, Y))+d(t_1(\omega_{u_1}(X, Y)))
\\&\quad -\lambda_{j}(X, Y)-\alpha_2(\phi (X) , d(Y))+\alpha_2(\phi(Y), d(X))-[d(X), d(Y)]-d([X, Y])\big )
\\&=(1-u_2t_2)\big (\omega_{u_2}(\phi(X), \phi(Y))+\alpha_2(\phi(X), d(Y))-\alpha_2(\phi(Y), d(X))+[d(X), d(Y)]
\\&\quad  -f(\omega_{u_1}(X, Y))+d(t_1(u_1([X, Y])))
\\&\quad -\lambda_{j}(X, Y)-\alpha_2(\phi (X) , d(Y))+\alpha_2(\phi(Y), d(X))-[d(X), d(Y)]-d([X, Y])\big )
\\&=(1-u_2t_2)\big (\omega_{u_2}(\phi(X), \phi(Y)) -f(\omega_{u_1}(X, Y))-\lambda_{j}(X, Y)
\\&\quad +d(t_1(u_1([X, Y]))-[X, Y])\big )
\\&= \tilde R_q(X, Y)+(1-u_2t_2)(d(t_1(u_1([X, Y]))-[X, Y]))\text{.}
\end{align*}
\end{comment}
This expression will find a natural explanation in the groupoid formalism we discuss in \cref{groupoids-of-adjustments}. 
\end{remark}

We conclude with three lemmas about the compatibility of the map $\Adj^{q}(\mathfrak{k})$ with identity butterflies, morphisms between butterflies,   and the composition of butterflies, and we compute $\Adj^{q}(\mathfrak{k})$ when $\mathfrak{k}$ is induced from a strict intertwiner.

\begin{lemma}
With $q$  the canonical section of the identity butterfly, $\Adj^{q}(\id_\mathfrak{G})=\id_{\Adj^{s}(\mathfrak{G})}$. 
\end{lemma}

\begin{proof}
The identity butterfly  $\id_\mathfrak{G}:\mathfrak{G} \to \mathfrak{G}$ of  $\mathfrak{G}$ has the canonical section  $q: \mathfrak{g} \to \mathfrak{g} \ltimes \mathfrak{h}$,  $q(x) := (x, 0)$, which is neat with respect to an arbitrary section $s$ in $\mathfrak{G}$.
$q$ is a Lie algebra homomorphism and thus $\lambda_q=0$.
Moreover, since $\phi=\id_{\mathfrak{g}}$ and $f=\id_{\mathfrak{a}}$, \cref{adjustment-respect} shows the claim. 
\end{proof}

\begin{lemma}
\label{compatible-sections}
Let $k: \mathfrak{k} \to \mathfrak{k}'$ be a morphism between butterflies  $\mathfrak{k}, \mathfrak{k}':\mathfrak{G}_1\to \mathfrak{G}_2$, each equipped with a section $s_i$. 
We assume that $k$ is compatible with neat sections $q$ and $q'$ in the sense  that $q = q' \circ k$. Then, $\Adj^{q}(\mathfrak{k})=\Adj^{q'}(\mathfrak{k}')$.
\end{lemma}

\begin{proof}
We have $\phi=\phi'$ and $f=f'$.
Moreover, $\lambda_q=\lambda_{q'}$, and hence \cref{adjustment-respect} shows the claim.   
\end{proof}

\begin{lemma}
\label{composition-of-butterflies-and-neat-sections}
Let $\mathfrak{G}_i=(\mathfrak{h}_i \stackrel{t_{i}}{\to} \mathfrak{g}_i \stackrel{\alpha_{i}}{\to} \Der(\mathfrak{h}_i))$ be central crossed modules of Lie algebras, $i=1, 2, 3$, each equipped with a section $s_i: \mathfrak{f}_i \to \mathfrak{g}_i$.  
Let\begin{equation*}
\begin{aligned}
\xymatrix{\mathfrak{h}_1 \ar[dd]_{t_1} \ar[dr]^{i_1} && \mathfrak{h}_2 \ar[dd]^{t_2}\ar[dl]_{i_2} \\ & \mathfrak{k} \ar[dl]^{r_1}\ar[dr]_{r_2} \\ \mathfrak{g}_1 && \mathfrak{g}_2}
\end{aligned}
\quad\text{ and }\quad
\begin{aligned}
\xymatrix{\mathfrak{h}_2 \ar[dd]_{t_2} \ar[dr]^{i_2'} && \mathfrak{h}_3 \ar[dd]^{t_3}\ar[dl]_{i'_3} \\ & \mathfrak{k}' \ar[dl]^{r_2'}\ar[dr]_{r'_3} \\ \mathfrak{g}_2 && \mathfrak{g}_3}
\end{aligned}
\end{equation*}
be  invertible butterflies equipped with neat sections $q$ and $q'$, respectively. Then, the formula
\begin{equation*}
\tilde q: \mathfrak{g}_1 \to \widetilde{\mathfrak{k}}:=(\mathfrak{k} \times_{\mathfrak{g}_2} \mathfrak{k}')/\mathfrak{h}_2, \quad x_1 \mapsto [q(x_1), (q'  r_2 q)(x_1)]\text{.}
\end{equation*} 
defines a neat section of the composed butterfly $\mathfrak{k}' \circ \mathfrak{k}$, and the equality
\begin{equation*}
\Adj^{\tilde q}(\mathfrak{k}' \circ \mathfrak{k}) = \Adj^{q'}(\mathfrak{k}') \circ \Adj^{q}(\mathfrak{k})
\end{equation*}
holds.  
\end{lemma} 

\begin{proof}
We remark that the retract corresponding to $\tilde q$ is
\begin{equation*}
\tilde j: (\mathfrak{k} \times_{\mathfrak{g}_2} \mathfrak{k}')/\mathfrak{h}_2 \to \mathfrak{h}_3: [x, x']\mapsto j'(x'-i_2'(j(x)))\text{.}
\end{equation*} 
\begin{comment}
This is easily checked to be well-defined and a section against $\tilde r_1: \widetilde{\mathfrak{k}} \to \mathfrak{g}_1$. 
\end{comment}
The induced maps $\phi_{\tilde q}: \mathfrak{g}_1 \to \mathfrak{g}_3$ and $f_{\tilde q}: \mathfrak{h_1} \to \mathfrak{h}_3$ are the compositions of the separate ones, i.e. $\phi_{\tilde q}=\phi_{q'}  \phi_q$ and $f_{\tilde q} = f_{q'}  f_q$.
\begin{comment}
Indeed, 
\begin{equation*}
\phi_{\tilde j}(x_1) = \tilde r_3q_{\tilde j}(x_1)=\tilde r_3([q(x_1), (q' \circ r_2\circ q)(x_1)])=r_3'((q' \circ r_2\circ q)(x_1))=\phi_{j'}(\phi_q(x_1))
\end{equation*}
\end{comment}
Since $q$ and $q'$ are neat, it follows that $\tilde q$ is neat, too.
\begin{comment}
Indeed, 
\begin{align*}
t_3u_3 \phi_{\tilde j} =t_3u_3\phi_{j'}\phi_q=\phi_{j'}t_2u_2\phi_q=\phi_{j'}\phi_q t_1 u_1= \phi_{\tilde j}  t_1  u_1
\end{align*}
\end{comment}
We compute the cochain $\lambda_{\tilde q}$ defined in \eqref{definition-of-lambda}: 
\begin{align*}
\lambda_{\tilde q}(x_1, x_1') &= \tilde j([q_{\tilde j}(x_1), q_{\tilde j}(x_1')])
\\&= \tilde j([[q(x_1), q(x_1')], [q'r_2q(x_1), q'r_2q(x_1')]])
\\&= j'([q'r_2q(x_1), q'r_2q(x_1')]-i_2'(j([q(x_1), q(x_1')])))
\\&= \lambda_{q'}(\phi_q(x_1), \phi_q(x_1'))-f_{q'}(\lambda_q(x_1, x_1')))\text{;}
\end{align*}
thus,
\begin{equation*}
\lambda_{\tilde q} = \phi_q^{*}\lambda_{q'}+(f_{q'})_{*}\lambda_q\text{.}
\end{equation*}
Using \cref{adjustment-respect},  our assumption is that
\begin{align*}
(f_q)_{*}(\eta_1)+\lambda_q =\phi_q^{*}\eta_2
\quad\text{ and }\quad
(f_{q'})_{*}(\eta_2)+\lambda_{q'} =\phi_{q'}^{*}\eta_3\text{,}
\end{align*}
and the claim is proved by the following calculation:
\begin{equation*}
(f_{\tilde q})_{*}(\eta_1)+\lambda_{\tilde q}=f_{q'}  (f_q(\eta_1)+\lambda_q)+\phi_q^{*}\lambda_{q'}=\phi_{q}^{*}((f_{q'})_{*}\eta_2+\lambda_{q'})=\phi_q^{*}\phi_{q'}^{*}\eta_3=\phi_{\tilde q}^{*}\eta_3\text{.}
\qedhere
\end{equation*}
\end{proof}

\section{Groupoids of adjustments}

\label{groupoids-of-adjustments}

We pick up an idea of Tellez-Dominguez \cite{TellezDominguez2023} to understand pairs of sections and adapted adjustments as objects of a groupoid. We show in \cref{functors-from-butterflies} that this setting allows to interpret the covariance results of \cref{covariance-under-butterflies} as a functor on a bicategory of crossed modules. Finally, in \cref{adjusted-crossed-modules} we define and classify a bicategory of crossed modules with adjustments.   

\subsection{The groupoid of adjustments}

Let $\Gamma=(H \stackrel{t}{\to} G \stackrel{\alpha}{\to} \Aut(H))$ be a central crossed module of Lie groups, let  $\mathfrak{G}=(\mathfrak{h} \stackrel{t_*}{\to} \mathfrak{g} \stackrel{\alpha_*}{\to} \Der(\mathfrak{h}))$ be the induced crossed module of Lie algebras, and let $\mathfrak{f} := \pi_0(\mathfrak{G})$. 

\begin{definition}
The \emph{groupoid $\ADJ(\Gamma)$ of  adjustments} is the following:
its objects are pairs $(s, \kappa)$ consisting of a section $s$ in $\mathfrak{G}$ and an adjustment $\kappa$ on $\Gamma$ that is adapted to $s$.
Morphisms $(s, \kappa) \to (s', \kappa')$ are linear maps $\phi : \mathfrak{f} \to \mathfrak{h}$ such that 
\[
s-s' = t_*\phi \qquad \text{and} \qquad \kappa'(g, X) - \kappa(g, X) = \phi  p_{*}( \mathrm{Ad}_g(X)-X).
\]
Composition is the addition of linear maps. 
%
\begin{comment}
These conditions are consistent because
\begin{align*}
(\rho_{s'} - \rho_s)(\Ad_g(X) - X) &= t_*\kappa'(g, X) - t_*\kappa(g, X) 
\\
&= t_*\phi p_*(\Ad_g(X) - X)
\\
&= (s - s')p_*(\Ad_g(X) - X) 
\\
&= \big((1 - s'p_*) - (1- sp_*)\big)(\Ad_g(X) - X).
\end{align*}
Here the first equation uses the Adaptedness and the second one the property of $\phi$.
\end{comment}
%
\end{definition}

Tellez-Dominguez defines in \cite[Def. 3.8]{TellezDominguez2023} another groupoid $\ADJ_{\mathrm{TD}}(\Gamma)$ with objects the \quot{strong adjustments}: pairs $(j, \tilde \kappa)$ of a retract $j: \mathfrak{h} \to \mathfrak{a}$ (i.e., a linear map $j$ such that $j|_{\mathfrak{a}}=\id_{\mathfrak{a}}$) and maps
\begin{equation*}
\tilde\kappa: G \times \mathfrak{g} \to \mathfrak{a}
\end{equation*} 
satisfying
\begin{align*}
\tilde \kappa(g_1g_2, X) &= \tilde \kappa\big(g_1, \Ad_{g_2}(X)\big) + \tilde \kappa(g_2, X) 
\\
\tilde \kappa\big(t(h), X\big) &= j((\tilde{\alpha}_{h^{-1}})_* X)
\\
\tilde \kappa(g, t_{*}x) &= j(\alpha_g(x) - x)\text{.}
\end{align*}
The morphisms $(j, \tilde\kappa) \to (j', \tilde\kappa')$ are linear maps $\psi: \mathfrak{g} \to \mathfrak{a}$ such that
\begin{equation*}
j'-j=\psi t_{*}
 \qquad \text{and} \qquad 
 \tilde \kappa(g, X) - \tilde \kappa'(g, X) = \psi( \mathrm{Ad}_g(X)-X)\text{.}
\end{equation*}
In order to compare
ours and Tellez-Dominguez's groupoids it is helpful to introduce a third groupoid $\smash{\widetilde{\ADJ}(\Gamma)}$. 
The objects of this third groupoid are pairs $(u, \kappa)$ where $u:\mathfrak{g} \to \mathfrak{h}$ is a splitting and $\kappa$ is an adjustment that is adapted to the section $s_u$ induced by $u$. The morphisms $(u, \kappa) \to (u', \kappa')$ are pairs $(\phi, \psi)$ of linear maps $\phi: \mathfrak{f}\to \mathfrak{h}$ and $\psi: \mathfrak{g} \to \mathfrak{a}$ such that 
\[
u'-u = \phi p_{*} - \psi \qquad \text{and} \qquad \kappa'(g, X) - \kappa(g, X) =\phi p_{*}( \mathrm{Ad}_g(X)-X).
\]
%
\begin{comment}
Again, we check that these conditions are consistent.
\begin{align*}
(\rho_{s_{u'}} - \rho_{s_u})(\Ad_g(X) - X) 
&= t_*\kappa'(g, X) - t_* \kappa(g, X)
\\
&=t_* \phi p_*(\Ad_g(X) - X)
\\
&= t_*(u' - u + \psi)(\Ad_g(X) - X)
\\
&= (t_*u' - t_*u )(\Ad_g(X) - X)
\end{align*}
Again the first equation uses Adaptedness and the second one the property of $\phi$.
Recall that $s_up_* = 1 - t_*u$, hence $\rho_{s_u} = 1 - s_u p_* = t_* u$, so this the left and the right hand side are indeed equal.
\end{comment}
%
There is a span of functors between the adjustment categories defined above, 
\begin{equation*}
\ADJ_{\mathrm{TD}}(\Gamma) \leftarrow \widetilde{\ADJ}(\Gamma) \to \ADJ(\Gamma)\text{.}
\end{equation*}
given by
\begin{equation}
\label{FunctorsBetweenAdjustmentCategories}
\begin{aligned}
        \widetilde{\ADJ}(\Gamma) &\to \ADJ(\Gamma), 
        \\ (u, \kappa) &\mapsto (s_u, \kappa), 
        \\ (\phi, \psi)&\mapsto \phi  
\end{aligned}
\qquad 
\begin{aligned}
        \widetilde{\ADJ}(\Gamma) &\to \ADJ_{\mathrm{TD}}(\Gamma), 
        \\ (u, \kappa) &\mapsto (j_u, j\kappa), 
        \\ (\phi, \psi)&\mapsto \psi  
\end{aligned}
\end{equation}
\begin{comment}
If $u'-u=\phi p_{*} - \psi$, then 
\begin{equation*}
(s_{u}-s_{u'}) p_{*}=(1-t_{*}u)-(1-t_{*}u')=t_{*}(u'-u)=t_{*}(\phi p_{*} - \psi)=t_{*}\phi p_{*}\text{, }
\end{equation*}
and since $ p_{*}$ is surjective, this implies $s_{u}-s_{u'}=t_{*}\phi$.
So $\phi$ is indeed a morphism $(s_u, \kappa) \to (s_{u'}, \kappa')$, as claimed.

Let $(\phi, \psi):(u, \kappa) \to (u', \kappa')$ be a morphism.
We calculate
\begin{equation*}
j_{u'}-j_u =1-u't_{*}-(1-ut_{*})=(u-u')t_{*}= (\psi - \phi p_*)t_* = \psi t_{*}
\end{equation*}
and
\begin{align*}
j_{u'} \kappa'(g, X) - j_{u}\kappa(g, X) &=j_{u'}\kappa(g, X)+j_{u'}\phi p_{*}(\mathrm{Ad}_g(X)-X)-j_{u}\kappa(g, X)
\\&=\psi t_{*}\kappa(g, X)+j_{u'}\phi p_{*}(\mathrm{Ad}_g(X)-X)
\\&=(\psi t_{*}u+j_{u'}(u'-u+\psi))(\mathrm{Ad}_g(X)-X)
\\&=(\psi t_{*}u - j_{u'}u+\psi)(\mathrm{Ad}_g(X)-X)
\\&=\psi(\mathrm{Ad}_g(X)-X), 
\end{align*}
where in the last step, we used the previous identity to obtain
\[
\psi t_* u - j_{u'}u = (\psi t_* - j_{u'})u = - j_u u = 0.
\]
\end{comment}

\begin{proposition}
The functors in \eqref{FunctorsBetweenAdjustmentCategories} are equivalences of categories. 
In particular, the groupoids $\ADJ(\Gamma)$ and $\ADJ_{\mathrm{TD}}(\Gamma)$ are equivalent. 
\end{proposition}

\begin{proof}
The left functor in \cref{FunctorsBetweenAdjustmentCategories} is (essentially) surjective since every section $s$ is of the form $s=s_u$ for a splitting $u$. It is faithful since $\psi$ is uniquely determined by $u-u' = \phi p_{*} + \psi$. To see that it is full, suppose $\phi:\mathfrak{f}\to \mathfrak{h}$ satisfies $s_{u'}-s_u = t_*\phi$. Then, 
\begin{equation*}
t_{*}(u-u')=1-t_{*}u'-(1-t_{*}u)=(s_{u'}-s_{u}) p_{*}=t_{*}\phi p_{*}\text{.}
\end{equation*}
Thus, there exists a unique linear map $\psi: \mathfrak{g} \to \mathfrak{a}$ such that $u-u'=\phi p_{*} - \psi$, and $(\phi, \psi)$ is a preimage of $\phi$.  Hence, the left functor in \eqref{FunctorsBetweenAdjustmentCategories} is an equivalence of groupoids.

The right functor in \cref{FunctorsBetweenAdjustmentCategories} is faithful because $\phi$ is uniquely determined by $u'-u = \phi  p_{*} - \psi$, due to the surjectivity of $p_*$. 
In order to see that it is full, suppose we have objects $(u, \kappa)$ and $(u', \kappa')$ and a linear map $\psi:\mathfrak{g} \to \mathfrak{a}$  that is a morphism $(j_u, j_u\kappa) \to (j_{u'}, j_{u'}\kappa')$. 
%
\begin{comment}
Then $\psi$ satisfies the equations
\begin{equation*}
j_{u'}-j_u=\psi t_{*}
 \qquad \text{and} \qquad 
 j_{u}\kappa(g, X) - j_{u'}\kappa'(g, X) = \psi( \mathrm{Ad}_g(X)-X)\text{.}
\end{equation*}
\end{comment}
%
Then, 
\begin{equation*}
(u-u')t_{*} = 1-j_u-(1-j_{u'})=j_{u'}-j_u=\psi t_{*}\text{.}
\end{equation*}
Thus, there exists a linear map $\phi:\mathfrak{f} \to \mathfrak{h}$ such that $u'-u=\phi p_{*} - \psi$. 
It satisfies
\begin{align*}
\kappa'(g, X) - \kappa(g, X) &=(j_{u'}+u't_{*})\kappa'(g, X)-(j_u+ut_{*})\kappa(g, X)
\\
&=(j_{u'}+u't_{*})\kappa'(g, X)-(j_u+ut_{*})\kappa(g, X)
\\
&=(\psi +u't_{*}u'-ut_{*}u)(\mathrm{Ad}_g(X)-X)
\\&=(\psi+u'-u)(\mathrm{Ad}_g(X)-X)
\\&=\phi p_{*}( \mathrm{Ad}_g(X)-X), 
\end{align*}
hence $(\phi, \psi)$ is indeed a morphism in $\widetilde{\ADJ}(\Gamma)$ that is sent to the morphism $\psi : (j_u, j_u \kappa) \to (j_{u'}, j_{u'}\kappa')$ under the right functor.

Finally, the right functor is (essentially) surjective: if $(j, \tilde\kappa)$ is a strong adjustment, we choose a splitting $u$ such that $j=j_u$ and  set
\begin{equation*}
\kappa(g, X) := \tilde\kappa(g, X) + u(\mathrm{Ad}_g(X)-X)\text{.}
\end{equation*}
It is straightforward to check that this is an adjustment and adapted to $s_u$; moreover, $j_u\kappa=\tilde\kappa$. 
Hence, $(u, \kappa)$ is sent to $(j, \tilde{\kappa})$ under the right functor.
\end{proof}

\subsection{The groupoid of infinitesimal adjustments}
 
 Let $\mathfrak{G}=(\mathfrak{h} \stackrel{t_*}{\to} \mathfrak{g} \stackrel{\alpha_*}{\to} \Der(\mathfrak{h}))$ be a central crossed module of Lie algebras, with $\mathfrak{f} := \pi_0(\mathfrak{G})$ and $\mathfrak{a} := \pi_1(\mathfrak{G})$.       
        
\begin{definition}
The \emph{groupoid $\ADJ(\mathfrak{G})$ of  infinitesimal adjustments} is defined as follows.
The objects are pairs $(s, \eta)$ consisting of a section $s$ and an infinitesimal  adjustment $\eta$ on $\mathfrak{G}$ that is adapted to $s$.
Morphisms $(s, \eta) \to (s', \eta')$ are linear maps $\phi : \mathfrak{f} \to \mathfrak{h}$ such that 
\[
s-s' = t_*\phi \qquad \text{and} \qquad \eta'- \eta = \delta ( p^*\phi).
\]
Composition is again the addition of linear maps.
\end{definition}

We note that \cref{integration-of-adjustments} has the following reformulation:
If $\Gamma = (H \to G \to \Aut(H))$ is a central crossed module of Lie groups and $\mathfrak{G}$ is the corresponding crossed module of Lie algebras, then differentiation yields a faithful functor 
\begin{equation*}
\ADJ(\Gamma) \to \ADJ(\mathfrak{G}).
\end{equation*}
It is full if the Lie group $G$ in the crossed module $\Gamma$  is connected, and essentially surjective if $G$ is connected and simply connected and $H$ is connected. 
The following result yields a complete classification result for the groupoid $\ADJ(\mathfrak{G})$ of  adjustments.

\begin{theorem}
\label{homotopy-groups-of-ADJ}
The set $\pi_0(\ADJ(\mathfrak{G}))$ of isomorphism classes of objects of $\ADJ(\mathfrak{G})$ admits a map 
\[
\pi_0(\ADJ(\mathfrak{G})) \longrightarrow \Sym^2(\mathfrak{f}, \mathfrak{a})^{\ad}
\]
whose fibre over $B\in \Sym^2(\mathfrak{f}, \mathfrak{a})^{\ad}$  is an affine space over $H^2(\mathfrak{f}, \mathfrak{a})$ if $[\cw(B)]=\KL(\mathfrak{G})$, and is empty else.
The automorphism group of each object $(s, \eta)$ of $\ADJ(\mathfrak{G})$ is given by
\[
\pi_1(\ADJ(\mathfrak{G}))=\Aut(s, \eta) = H^1(\mathfrak{f}, \mathfrak{a}) = \Lin(\mathfrak{f}/[\mathfrak{f}, \mathfrak{f}], \mathfrak{a}).
\]
\end{theorem}

\begin{proof}
Observe that if two adapted adjustments $\eta$ and $\eta'$ are isomorphic in $\ADJ(\mathfrak{G})$, then they have the same symmetric part.
This yields the desired map
\[
\pi_0(\ADJ(\mathfrak{G})) \longrightarrow \Sym^2(\mathfrak{f}, \mathfrak{a})^{\ad}.
\]
We show that the group $H^2(\mathfrak{f}, \mathfrak{a})$ acts on the fibres of this map. 
First we already know that $\Alt^2_{\mathrm{cl}}(\mathfrak{f}, \mathfrak{a})$ acts on $\Adj^u(\mathfrak{G})$:
If $\eta$ is an $s$-adapted adjustment and $\xi \in \Alt^2_{\mathrm{cl}}(\mathfrak{f}, \mathfrak{a})$, then $\eta +  p^*\xi$ is again an $s$-adapted adjustment.
Now, if $\xi - \xi' = \delta \varphi$ for some linear map $\varphi \in \Alt^1(\mathfrak{f}, \mathfrak{a}) = \Lin(\mathfrak{f}, \mathfrak{a})$, then $\phi := \iota_* \varphi$ is an isomorphism between $(s, \eta +  p^*\xi)$ and $(s, \eta +  p^*\xi')$; hence, this action descends to an action of $H^2(\mathfrak{a}, \mathfrak{f})$ on $\pi_0(\ADJ(\mathfrak{G}))$.

We show that this action is transitive. 
So let $(s, \eta)$ and $(s', \eta')$ be two objects such that $\eta$ and $\eta'$ have the same symmetric part. 
We have to find an element of $H^2(\mathfrak{f}, \mathfrak{a})$ that sends the isomorphism class of $(s, \eta)$ to the isomorphism class of $(s', \eta')$.
By Prop.~\ref{affine-space-structure-on-adjustments}, we have $\eta - \eta' =  p^*\xi$ for some $\xi \in \Alt^2_{\mathrm{cl}}(\mathfrak{f}, \mathfrak{h})$.
On the other hand, $ p(s'-s)=0$, and so there exists $\phi:\mathfrak{f} \to \mathfrak{h}$ with $s'-s=t\phi$. Now, by adaptedness, we have
\[
t\xi p=t  p^*\xi = t_*(\eta - \eta') = \delta (\rho_s - \rho_{s'}) =- \delta( s'-s) =-t_{*}\delta\phi.
\]
We therefore obtain that the element $\xi' := \xi + \delta \phi$ takes values in $\mathfrak{a}$, hence defines a cohomology class $[\xi'] \in H^2(\mathfrak{f}, \mathfrak{a})$.
Acting by $\xi'$ sends  $(s', \eta')$ to  $(s', \eta' +  p^*\xi')$.
We claim that $\phi$ is an isomorphism between $(s, \eta)$ and $(s', \eta' +  p^*\xi')$: $s'-s=t\phi$ was already collected above, and we have 
\[
 (\eta' +  p^*\xi')-\eta =  p^*\xi' -  p^*\xi =  p^*\delta \phi.
\]
This shows that the action is transitive.

To see that the action is free, let $\xi \in \Alt^2_{\mathrm{cl}}(\mathfrak{f}, \mathfrak{a})$ and suppose that $(s, \eta)$ is isomorphic to $(s, \eta')$, where $\eta' = \eta +  p^*\xi$.
Then there exists a linear map $\phi : \mathfrak{f} \to \mathfrak{h}$ which by the first condition actually takes values in $\mathfrak{a}$ and satisfies $\delta(  p^*\phi) = \eta' - \eta =   p^*\xi$.
But this shows that $\xi$ is exact, hence is zero in $H^2(\mathfrak{f}, \mathfrak{a})$.

Finally, for each adjustment $(s, \eta)$, the group of automorphisms is given by
\[
\Aut(s, \eta) = H^1(\mathfrak{f}, \mathfrak{a}) = \Lin(\mathfrak{f}/[\mathfrak{f}, \mathfrak{f}], \mathfrak{a}).
\]
Indeed, if $\phi$ is an automorphism of $(s, \eta)$, then the first condition implies that $t_* \phi = 0$, so that $\phi$ takes values in $\mathfrak{a}$, while the second condition implies that $\phi$ is closed. 
\end{proof}

% \begin{theorem}
% Let $\mathfrak{G}=(\mathfrak{h} \stackrel{t}{\to} \mathfrak{g} \stackrel{\alpha}{\to} \Der(\mathfrak{h}))$ be a central crossed module with Kassel-Loday class $\KL(\mathfrak{G})$.
% The set $\pi_0(\ADJ(\mathfrak{G}))$ of isomorphism classes of  adjustments is a principal $H^2(\mathfrak{f}, \mathfrak{a})$-bundle over the set 
% $\Sym^2(\mathfrak{f}, \mathfrak{a})^{\ad}_{\KL(\mathfrak{G})}$, consisting of the ad-invariant symmetric bilinear forms $B$ that map to $\KL(\mathfrak{G})$ under the Chern-Weil homomorphism.
% \[
% \begin{tikzcd}
% H^2(\mathfrak{f}, \mathfrak{a}) \ar[r] & \pi_0(\ADJ(\mathfrak{G})) \ar[d]
% \\
% & \Sym^2(\mathfrak{f}, \mathfrak{a})^{\ad}_{\KL(\mathfrak{G})}
% \ar[r, mapsto] & \KL(\mathfrak{G}).
% \end{tikzcd}
% \]
% \end{theorem}

\subsection{Functors from butterflies}

\label{functors-from-butterflies}

Next, we consider two central crossed modules  $\mathfrak{G}_i=(\mathfrak{h}_i \stackrel{t_i}{\to} \mathfrak{g}_i \stackrel{\alpha_i}{\to} \Der(\mathfrak{h}_i))$ of Lie algebras, and construct a functor
\begin{equation*}
\ADJ^{q}(\mathfrak{k}):\ADJ(\mathfrak{G}_1) \to \ADJ(\mathfrak{G}_2)
\end{equation*}
associated to an invertible butterfly $\mathfrak{k}:\mathfrak{G}_1 \to \mathfrak{G}_2$ and a section $q$ of $\mathfrak{k}$.
We write again $\phi_q: \mathfrak{g}_1 \to \mathfrak{g}_2$ and $f_q: \mathfrak{h}_1 \to \mathfrak{h}_2$ for the linear maps induced by $q$,   $\mathfrak{f}_i := \pi_0(\mathfrak{G}_0)$ for the homotopy groups, and  $\Phi:  \mathfrak{f}_1 \to  \mathfrak{f}_2$ for the Lie algebra homomorphisms induced by $\phi_q$. 
Let $(s_1, \eta_1)$ be an object in $\ADJ(\mathfrak{G}_1)$. 
We define a section $s_2$ in $\mathfrak{G}_2$ by   
\[
s_2:=\phi_q s_1\Phi^{-1}.
\]
Then, $q$ is neat with respect to $s_1$ and $s_2$.
Our functor is defined on the level of objects by
\begin{equation*}
\ADJ^{q}(\mathfrak{k})(s_1, \eta_1) := (s_2, \Adj^{q}(\mathfrak{k})(\eta_1))\text{, }
\end{equation*}   
where $\Adj^{q}(\mathfrak{k}):\Adj^{s_1}(\mathfrak{G}_1) \to \Adj^{s_2}(\mathfrak{G}_2)$ is the map defined in \cref{covariance-under-butterflies}.
On the level of morphisms, suppose $\phi_1 : \mathfrak{f}_1 \to \mathfrak{h}_1$ is a linear map such that
\begin{equation*}
s_1-s_1' = t_1\phi_1 \qquad \text{and} \qquad \eta_1 '- \eta_1 = \delta ( p_1^*\phi_1)\text{, }
\end{equation*}
i.e., $\phi_1$ is a morphism from $(s_1, \eta_1)$ to $(s_1', \eta_1')$. 

\begin{lemma}
\label{lemma-5-5}
$\phi_2 := f_q  \phi_1  \Phi^{-1}$ is a morphism in $\ADJ(\mathfrak{G}_2)$ between $\ADJ^{q}(\mathfrak{k})(s_1, \eta_1)$ and $\ADJ^{q}(\mathfrak{k})(s_1', \eta_1')$. 
\end{lemma}

Given the lemma, we define $\ADJ^{q}(\mathfrak{k})(\phi_1):=\phi_2$, completing the definition of the functor $\ADJ^{q}(\mathfrak{k})$. 
We observe that the assignment $\phi_1 \mapsto \phi_2$ is linear, which implies its functoriality.

\begin{proof}[Proof of \cref{lemma-5-5}]
Set $(s_2, \eta_2):= \ADJ^{q}(\mathfrak{k})(s_1, \eta_1)$ and $(s_2', \eta_2') := \ADJ^{q}(\mathfrak{k})(s_1', \eta_1')$.
The first condition is
\begin{equation*}
s_2-s'_2 =\phi_q (s_1- s_1')\Phi^{-1}=\phi_qt_1\phi_1 \Phi^{-1}=\phi_qr_1i_1\phi_1 \Phi^{-1}=t_2 f_q \phi_1 \Phi^{-1}=t_2\phi_2\text{.}
\end{equation*} 
To check the second condition, we choose half splittings $u_1$ and $u_1'$ extending $s_1$ and $s_1'$, respectively, and consider $\beta, \beta'$ such that $ p_1^{*}\beta := \eta_1 - \omega_{u_1}$ and $ p_1^{*}\beta' := \eta_1 '- \omega_{u'_1}$. By assumption, we have
\begin{equation*}
\delta ( p_1^*\phi_1)=\eta_1 '- \eta_1 = p_1^{*}(\beta'-\beta)-\omega_{u_1}  +\omega_{u_1'}\text{.}
\end{equation*}
By \cref{LemmaChangeSplitting}, there exists $\tilde\omega_1\in \Alt^2(\mathfrak{f}_1,\mathfrak{h}_1)$ such that $p_1^{*}\tilde\omega_1=\omega_{u_1'}-\omega_{u_1}$; thus, we obtain
\begin{equation*}
\delta \phi_1 = \beta'-\beta+\tilde\omega_1\text{.}
\end{equation*} 
Moreover, we choose half splittings $u_2$ and $u_2'$ extending $s_2$ and $s_2'$, respectively, giving the formulas $\eta_2 = \omega_{u_2}+ p_2^{*}(\Phi^{-1})^{*}(F_{*}{\beta}-{R_q})$ and $\eta_2' = \omega_{u'_2}+ p_2^{*}(\Phi^{-1})^{*}(F_{*}{\beta'}-{R_q'})$.
Just as in \cref{butterfly-map-depends-only-on-idempotent}, we have
\begin{equation*}
R_q'=R_q+\Phi^{*}\tilde\omega_2-F_{*}(\tilde\omega_1)\text{,}
\end{equation*}
where $\tilde\omega_2\in \Alt^2(\mathfrak{f}_2,\mathfrak{h}_2)$ such that $p_2^{*}\tilde\omega_2=\omega_{u'_2}-\omega_{u_2}$. 
Thus, we obtain
\begin{align*}
\eta_2'-\eta_2 
&=\omega_{u'_2}-\omega_{u_2}+ p_2^{*}(\Phi^{-1})^{*}(F_{*}{(\beta'-\beta})-R_q'+R_q)
\\&=p_2^{*}\tilde \omega_{2}+ p_2^{*}(\Phi^{-1})^{*}(F_{*}(\delta\phi_1-\tilde\omega_1)-\Phi^{*}\tilde\omega_2+F_{*}(\tilde\omega_1))
\\&=  p_2^{*}(\Phi^{-1})^{*}F_{*}(\delta\phi_1)\text{.} 
\\&=  p_2^{*}\delta\phi_2\text{.} 
\end{align*}
This shows that $\phi_2$ is a morphism from $(s_2, \eta_2)$ to $(s_2', \eta_2')$ in $\ADJ(\mathfrak{G}_2)$. 
\end{proof}

In order to investigate the dependence of the functor $\ADJ^q(\mathfrak{k})$ on the choice of the section $q$, we consider another section $q'$ and let
 $\gamma:\mathfrak{g}_1 \to \mathfrak{h}_2$ be the unique linear map satisfying $q'-q=i_2 \gamma$. For an object $(s, \eta)$ in $\ADJ(\mathfrak{G}_1)$, we define
\begin{equation*}
\phi_{(s, \eta)} := \gamma  s \Phi^{-1}: \mathfrak{f}_2 \to \mathfrak{h}_2\text{.}
\end{equation*}

\begin{lemma}
\label{natural-iso-for-two-sections}
If $q$ and $q'$ are  sections of an invertible butterfly $\mathfrak{k}: \mathfrak{G}_1 \to \mathfrak{G}_2$, 
the assignment $(s, \eta) \mapsto \phi_{(s, \eta)}$ defines a natural isomorphism
\begin{equation*}
\ADJ^{q}(\mathfrak{k})\cong \ADJ^{q'}(\mathfrak{k})\text{.}
\end{equation*} 
\end{lemma}

\begin{proof}
We consider the adjustments $\eta_2 := \Adj^{q}(\mathfrak{k})(\eta)$ and $\eta_2' := \Adj^{q'}(\mathfrak{B}k)(\eta)$, as well as the sections $s_2$ and $s_2'$ determined by $q$ and $q'$, respectively.    \cref{butterfly-map-dependence-on-splitting} provides the identity
\begin{equation*}
\eta_2'-\eta_2 =  p_2^{*}\delta\phi_{(s, \eta)}\text{.}
\end{equation*}
Moreover, 
\begin{align*}
s_2'-s_2
&=r_2( q' -q)s_1\Phi^{-1}
=r_2i_2 ds_1\Phi^{-1}
= t_2 \phi_{(s, \eta)}\text{;}
\end{align*}
this shows that $\phi_{(s, \eta)}$ is a morphism  $(s_2, \eta_2) \to (s_2', \eta_2')$. 
It remains to prove that the assignment $(s, \eta) \mapsto \phi_{(s, \eta)}$  is natural, which is to show that, for each morphism $\phi:(\rho, \eta) \to (\rho', \eta')$  in $\ADJ(\mathfrak{G}_1)$, the diagram 
\[
\begin{tikzcd}[column sep=4em]
\ADJ^{q}(\mathfrak{k})(\rho_1, \eta_1)
  \arrow[r, "\phi_{(s,\eta)}"]
  \arrow[d, "\ADJ^{q}(\mathfrak{k})(\phi)"']
& \ADJ^{q'}(\mathfrak{k})(\rho_1, \eta_1)
  \arrow[d, "\ADJ^{q'}(\mathfrak{k})(\phi)"] \\
\ADJ^{q}(\mathfrak{k})(\rho_1', \eta_1')
  \arrow[r, "\phi_{(s', \eta')}"']
& \ADJ^{q'}(\mathfrak{k})(\rho_1', \eta_1')
\end{tikzcd}
\]
is commutative. 
As preparation, we calculate
\begin{align*}
i_2(f_{q}-f_{q'})\phi p_1 
&=i_2(j'-j)i_1\phi p_1
\\&= i_2 \gamma r_1i_1\phi p_1 
\\&=i_2 \gamma t_1\phi p_1 
\\&= i_2\gamma(\rho'-\rho) 
\\&= i_2\gamma(s-s') p_1\text{.}
\end{align*}
Since $i_2$ is injective and $ p_1$ is surjective, it implies
\begin{equation*}
(f_{q}-f_{q'})\phi = \gamma(s-s')\text{.}
\end{equation*}
Now we prove the commutativity of the diagram above:
\begin{align*}
\ADJ^{q}(\mathfrak{k})(\phi)+\phi_{(s', \eta')}
&=f_{q}  \phi \Phi^{-1} +\gamma  s' \Phi^{-1}
= \gamma  s \Phi^{-1}+ f_{q'}  \phi  \Phi^{-1}
= \phi_{(s, \eta)}+\ADJ^{q'}(\mathfrak{k})(\phi)\text{.}
\end{align*}
This completes the proof. 
\end{proof}

Now we consider two invertible butterflies  $\mathfrak{k}$, $\mathfrak{k}': \mathfrak{G}_1\to \mathfrak{G}_2$   with sections $q$ and $q'$, respectively, and a morphism  $k: \mathfrak{k} \to \mathfrak{k}'$. If $k$ is compatible with the sections in the sense that $q'=k \circ q$, then it follows from \cref{compatible-sections} that $\ADJ^{q}(\mathfrak{k})=\ADJ^{q'}(\mathfrak{k}')$. If $k$ is not compatible, then $q'':= k\circ q$ is a second section of $\mathfrak{k}'$ that is compatible with $k$, and \cref{natural-iso-for-two-sections} provides a natural isomorphism
\begin{equation}
\label{natural-iso-assigned-to-2-morphism}
\ADJ^{q}(\mathfrak{k}) = \ADJ^{q''}(\mathfrak{k}') \to \ADJ^{q'}(\mathfrak{k}')\text{.}
\end{equation}

Wrapping up, we let $\cm^{\mathrm{sec}}$ be the bicategory whose objects are central crossed modules of Lie algebras, whose 1-morphisms are invertible butterflies equipped with  sections, and whose 2-morphisms are all morphisms between butterflies (not necessarily compatible with the sections). The composition
of 1-morphisms is the composition of butterflies,  equipped with the section defined in \cref{composition-of-butterflies-and-neat-sections}.  
\begin{theorem}
\label{2-functor}
The groupoids $\ADJ(\mathfrak{G})$, the functors $\ADJ^q(\mathfrak{k})$, and the natural transformations \eqref{natural-iso-assigned-to-2-morphism} establish a strict 2-functor
\begin{equation*}
\ADJ^{\mathrm{sec}}: \cm^{\mathrm{sec}} \to \grpd.
\end{equation*}
\end{theorem}   

\begin{proof}
It remains to check strict functoriality. Let $\mathfrak{G}_i=(\mathfrak{h}_i \stackrel{t_{i}}{\to} \mathfrak{g}_i \stackrel{\alpha_{i}}{\to} \Der(\mathfrak{h}_i))$ be crossed modules of Lie algebras, $i=1, 2, 3$, let $\mathfrak{k}:\mathfrak{G}_1 \to \mathfrak{G}_2$ and $\mathfrak{k}':\mathfrak{G}_2 \to \mathfrak{G}_3$ be invertible butterflies equipped with  sections $q$ and $q'$, and let the composition $\mathfrak{k}'\circ \mathfrak{k}$ be equipped with the section $\tilde q$ of \cref{composition-of-butterflies-and-neat-sections}, i.e.
\begin{equation*}
\tilde q(x_1):= [q(x_1), (q'  r_2 q)(x_1)]\text{.}
\end{equation*} 
We have to check
\begin{equation*}
\ADJ^{\tilde q}(\mathfrak{k}' \circ \mathfrak{k}) = \ADJ^{q'}(\mathfrak{k}') \circ \ADJ^{q}(\mathfrak{k})\text{.}
\end{equation*}
We start with an object $(s_1, \eta_1)\in \ADJ(\mathfrak{G}_1)$. We define the section   $s_2:=r_2 q s_1\Phi_1^{-1}$ in $\mathfrak{G}_2$ and have $\ADJ^{q}(\mathfrak{k})(s_1, \eta_1) = (s_2, \Adj^{q}(\mathfrak{k})(\eta_1))$. Next we define $s_3 := r_3'q's_2\Phi_2^{-1}$, getting \begin{equation*}
(\ADJ^{q'}(\mathfrak{k}')\circ \ADJ^{q}(\mathfrak{k}))(s_1, \eta_1)=\ADJ^{q'}(\mathfrak{k}')(s_2, \Adj^{q}(\mathfrak{k})(\eta_1))=(s_3, \Adj^{\tilde q}(\mathfrak{k}' \circ \mathfrak{k}) (\eta_1))\text{, }
\end{equation*}
where the last step uses the statement of \cref{composition-of-butterflies-and-neat-sections}.
 Now we  notice that
\begin{equation*}
s_3 =r_3'q'r_2qs_1\Phi_1^{-1}\Phi_2^{-1}=\tilde r_3\tilde qs_1(\Phi_2\circ\Phi_1)^{-1}\text{, }
\end{equation*}
so that it is precisely the section produced by the functor
$\Adj^{\tilde q}(\mathfrak{k}' \circ \mathfrak{k})$. 
This shows the claimed equality on the level of objects. On the level of morphisms, the composite $\Adj^{q'}(\mathfrak{k}') \circ \Adj^{q}(\mathfrak{k})$ sends a  morphism $\phi$ in $\ADJ(\mathfrak{G}_1)$  first to $f_q\phi\Phi_1^{-1}$ and then further to $f_{q'}f_q  \phi  \Phi_1^{-1}\Phi_2^{-1}$, which, via the identity $f_{\tilde q} = f_{q'}  f_q$, coincides with the value of $\phi$ under the functor $\ADJ^{\tilde q}(\mathfrak{k}' \circ \mathfrak{k})$.  
\end{proof}

The forgetful 2-functor
\begin{equation*}
\cm^{\mathrm{sec}} \to \cm
\end{equation*}
to the bicategory of (central) crossed modules of Lie algebras and invertible  butterflies
is an equivalence: it is the identity on the level of objects, and locally fully faithful and essentially surjective. 
This implies the following result for the homotopy 1-categories (obtained by identifying 2-isomorphic 1-morphisms)

\begin{corollary}
The 2-functor of \cref{2-functor} induces a well-defined canonical functor
\begin{equation*}
\mathrm{h}_1\cm \to \mathrm{h}_1\grpd\text{.}
\end{equation*}
\end{corollary}

Since weak equivalences of crossed modules are the isomorphisms in $\mathrm{h}_1\cm$, and equivalences of categories are the isomorphisms in $\mathrm{h}_1\grpd$, we obtain the following.

\begin{corollary}
\label{equivalent-crossed-modules-have-equivalent-groupoids-of-adjustments}
Weakly equivalent crossed modules of Lie algebras have equivalent groupoids of infinitesimal adjustments.
\end{corollary}
   
Moreover, in combination with \cref{integration-of-adjustments} we obtain the following. 
        
\begin{corollary}
Suppose $\Gamma_1$ and $\Gamma_2$ are weakly equivalent crossed modules of Lie groups with their Lie groups $G_1$ and $G_2$ connected and simply connected, and their Lie groups $H_1$ and $H_2$ connected. 
Then, there exists  an equivalence $\ADJ(\Gamma_1) \cong \ADJ(\Gamma_2)$. 
\end{corollary}

\subsection{The bicategory of adjusted crossed modules}        
        
\label{adjusted-crossed-modules}        
        
We consider the 2-functor $\ADJ^{\mathrm{sec}}: \cm^{\mathrm{sec}} \to \grpd$ from \cref{2-functor}, and perform the bicategorical Grothendieck construction; see, e.g. \cite{Bakovic}. That is, we consider the bicategory 
\begin{equation*}
\cm^{\mathrm{adj}} := \int_{\cm^{\mathrm{sec}}} \ADJ^{\mathrm{sec}}
\end{equation*}
 with:
\begin{description}

\item[Objects:] 
An object is a triple $(\mathfrak{G}, s, \eta)$ consisting of a crossed module $\mathfrak{G}$, a section $s$, and an infinitesimal adjustment $\eta$ on $\mathfrak{G}$ that is adapted to $s$.

\item[1-morphisms:] 
A 1-morphism $(\mathfrak{G}_1, s_1, \eta_1) \to (\mathfrak{G}_2, s_2, \eta_2)$ is a triple $(\mathfrak{k}, q, \phi)$ consisting of an invertible butterfly $\mathfrak{k}: \mathfrak{G}_1 \to \mathfrak{G}_2$, a  section $q$ in $\mathfrak{k}$, and a morphism $\phi: \ADJ^{q}(\mathfrak{k})(s_1, \eta_1) \to (s_2, \eta_2)$ in $\ADJ(\mathfrak{G}_2)$.
The section $q$ is automatically neat with respect to $s_1$ and $s_2$.
\begin{comment}
Let $(s_1', \eta_1') := \ADJ^{q}(\mathfrak{k})(s_1, \eta_1)$. Then, $q$ is neat with respect to  $s_1$ and $s_1'$ (see the beginning of \cref{functors-from-butterflies}).
Since $\phi$ is a morphism in  $\ADJ(\mathfrak{G}_2)$, we have $s_2-s_1'=t_{*}\phi$. This implies $\rho_{s_2}=\rho_{s_1'}$, which, in turn, by \cref{Neatness}, means that $q$ is also neat with respect to  $s_1$ and $s_2$.
 
\end{comment}

\item[Composition of 1-morphisms:]
        The composition of 1-morphisms is given by
        \[
        (\mathfrak{k}', q', \phi') \circ (\mathfrak{k}, q, \phi) = (\mathfrak{k}' \circ \mathfrak{k}, \tilde{q}, \tilde{\phi})
        \]
        where $\tilde{q}$ is the  section from \cref{composition-of-butterflies-and-neat-sections} and where $\tilde{\phi}$ is given by
\[
\tilde{\phi} = \phi' \circ \ADJ^{q'}(\mathfrak{k}')(\phi).
\]
\begin{comment}
This makes sense because
\begin{align*}
\ADJ^{q'}\big(\mathfrak{k}')(\ADJ^q(\mathfrak{k})(s_1, \eta_1)\big) 
&= \ADJ^{q'}(s_2, \Adj^q(\eta_1))
\\
&= (s_3, \Adj^{q'}(\Adj^q(\eta_1)))
\\
&= (s_3, \Adj^{\tilde{q}}(\eta_1)))
\\
&= \ADJ^{\tilde{q}}(s_1, \eta_1)
\end{align*}
by \cref{composition-of-butterflies-and-neat-sections}.
\end{comment}

\item[2-morphisms: ]A 2-morphism $(\mathfrak{k}, q, \phi)\Rightarrow (\mathfrak{k}', q', \phi')$ is a 2-morphism $k: \mathfrak{k} \Rightarrow \mathfrak{k}'$ such that the diagram
\begin{equation}
\label{2-MorphismCMadj}
\begin{tikzcd}
\ADJ^{q}(\mathfrak{k})(s_1,\eta_1) \arrow[rr] \arrow[dr, "\phi"'] 
&& \ADJ^{q'}(\mathfrak{k}')(s_1,\eta_1) \arrow[dl, "\phi'"] \\
& (s_2,\eta_2)
\end{tikzcd}
\end{equation}
is commutative, where the top arrow is the component of the natural transformation \eqref{natural-iso-assigned-to-2-morphism} at the object $(s_1, \eta_1)$.

\end{description}

\begin{remark}
\label{strict-1-morphisms}
Suppose   $\mathfrak{G}_i=(\mathfrak{h}_i \stackrel{t_i}{\to} \mathfrak{g}_i \stackrel{\alpha_i}{\to} \Der(\mathfrak{h}_i))$ are central crossed modules of Lie algebras, $i=1,2$, both equipped with an adapted adjustment $\eta_i$. Suppose $(f, \phi)$ is a strict intertwiner (see \cref{strict-intertwiners}) between $\mathfrak{G}_1$ and $\mathfrak{G}_2$ such  that
\begin{equation*}
f(\eta_1(X_1, X_2))=\eta_2(\phi(X_1), \phi(X_2))
\end{equation*}
holds for all $X_1,X_2\in \mathfrak{g}_1$, and such that the induced butterfly $\mathfrak{k}$ is invertible. Let $q$ be the canonical section of $\mathfrak{k}$, let $s_1$ be an arbitrary section in $\mathfrak{G}_1$, and let $s_2 := \phi s_1\Phi^{-1}$. By \cref{adjustment-respect} (also see \cref{strict-intertwiners-adjustments}), we have
$\Adj^q(\mathfrak{k})(\eta_1)=\eta_2$;
thus, 
\begin{equation*}
(\mathfrak{k}, q, \id): (\mathfrak{G}_1, s_1, \eta_1) \to (\mathfrak{G}_2, s_2, \eta_2)
\end{equation*}
is a 1-isomorphism in $\cm^{\mathrm{adj}}$. 
In this sense, strict intertwiners that strictly preserve adjustments,  give rise to 1-isomorphisms in $\cm^{\mathrm{adj}}$.  
\end{remark}

\begin{remark}
A  bicategory of adjusted crossed modules (of Lie groups), similar to our bicategory $\cm^{\mathrm{adj}}$, is  considered in \cite{Rist2022}, using spans of 1-isomorphisms induced by strict intertwiners, as in \cref{strict-1-morphisms}.
\end{remark}

As in \cref{classification-result}, we may restrict our considerations to crossed modules with fixed homotopy Lie algebras $\mathfrak{a}$ and $\mathfrak{f}$, and to invertible  butterflies inducing the identities on those. 
Then, \cref{homotopy-groups-of-ADJ,affine-linear-map} imply that there is a well-defined map
\begin{equation*}
\KL^{\mathrm{adj}}: \pi_0 \cm^{\mathrm{adj}}(\mathfrak{f}, \mathfrak{a}) \to \Sym^2(\mathfrak{f}, \mathfrak{a})^{\ad}, \qquad  [(\mathfrak{G}, s, \eta)] \mapsto \KL^{\mathrm{adj}}(\mathfrak{G},\eta)\text{,}
\end{equation*}
In the following we show that this map is a bijection, and thereby establish it as an adjusted analogue of the classification of crossed modules in \cref{classification-of-crossed-modules}.

Our main tool is the following construction. Let $(\mathfrak{G}, s, \eta)$ be an adjusted crossed module, and  $\xi \in \Alt^2_{\mathrm{cl}}(\mathfrak{f}, \mathfrak{a})$.  
We notice that $(\phi, f, \lambda)$, defined by 
\[
\lambda := \iota_*p^*\xi \in \Alt_{\mathrm{cl}}^2(\mathfrak{g}, \mathfrak{h}), \qquad \phi = \id_{\mathfrak{g}}, \qquad \text{and} \qquad f := \id_{\mathfrak{h}},
\]
constitutes cocycle data for a butterfly in the sense of \cref{butterflies}: the cocycle conditions are satisfied because $\lambda$ is closed and in the image of $\iota$, and because $\mathfrak{G}$ is central. Thus, the reconstruction \eqref{reconstruction-from-cocycle-data} produces a butterfly $\mathfrak{k}_{\xi}:\mathfrak{G} \to \mathfrak{G}$. It has a canonical section $q_{\xi}$ (see the proof of \cref{constructing-butterflies}),   and it induces the identities on $\mathfrak{f}$ and $\mathfrak{a}$. 
The following lemma summarizes  the role of these butterflies $\mathfrak{k}_{\xi}$.

\begin{lemma}
\label{LemmaMorphismsInCrMod}
Let $(\mathfrak{G}, s, \eta)$ be an adjusted crossed module. 
\begin{enumerate}[(a)]

\item 
For  any $\xi \in \Alt^2_{\mathrm{cl}}(\mathfrak{f}, \mathfrak{a})$, the butterfly $\mathfrak{k}_{\xi}$ extends to a 1-isomorphism
\begin{equation*}
(\mathfrak{k}_{\xi}, q_{\xi}, \id) :(\mathfrak{G}, s, \eta) \to (\mathfrak{G}, s, \eta + p^*\xi)
\end{equation*} 
in $\cm^{\mathrm{adj}}(\mathfrak{f}, \mathfrak{a})$.

\item

If $\eta$ and $\eta'$ are two infinitesimal adjustments on $\mathfrak{G}$ adapted to the same section $s$ and $\KL^{\mathrm{adj}}(\mathfrak{G}, \eta) = \KL^{\mathrm{adj}}(\mathfrak{G}, \eta')$, then $(\mathfrak{G}, s, \eta)$ and $(\mathfrak{G}, s, \eta')$ are isomorphic in $\cm^{\mathrm{adj}}(\mathfrak{f}, \mathfrak{a})$. 
\end{enumerate}
\end{lemma}

\begin{proof}
(a) Let $u$ be a splitting of $\mathfrak{G}$ extending $s$.
We compute the cochain $R_{q_{\xi}}$ of the butterfly $\mathfrak{k}_{\xi}$ from  \eqref{cochain-of-butterfly} with respect to $u$ in domain and target.
By \cref{constructing-butterflies}, 
we have $\phi_{q_{\xi}}=\phi=\id$ and $f_{q_{\xi}}=f=\id$, so that \eqref{cochain-of-butterfly} gives $R_q'=-\lambda=-\iota_*p^*\xi$, yielding $R_q=-\xi$. 
Now we are in position to compute
$\Adj^{q}(\mathfrak{k}_{\xi})(s, \eta)$. Since $\phi_q=\phi=\id$, the new section is again $s$.
 We consider $\beta: \mathfrak{f}\times \mathfrak{f}\to \mathfrak{h}$ with $p^{*}\beta := \eta - \omega_{u}$. 
 Then,
\begin{equation*}
\Adj^{q}(\mathfrak{k})(\eta)=\omega_{u}+p^{*}(\beta-{R_q})=\eta-p^{*}R_q=\eta+p^{*}\xi\text{.}
\end{equation*}
This shows that $\ADJ^q(\mathfrak{k})(s, \eta)=(s, \eta+p^{*}\xi)$, and proves (a).

(b)
If $\eta, \eta'$ are two $s$-adapted adjustments on $\mathfrak{G}$ with $\KL^{\mathrm{adj}}(\mathfrak{G}, \eta) = \KL^{\mathrm{adj}}(\mathfrak{G}, \eta')$, then the symmetric parts of $\eta$ and $\eta'$ agree, hence 
\[
\eta' - \eta = p^* \xi
\]
for some $\xi \in \Alt^2(\mathfrak{f}, \mathfrak{a})$, which is closed by \eqref{FunnyFourTermSequence}.
Hence, (a) provides the claimed isomorphism $(\mathfrak{G}, s, \eta) \cong (\mathfrak{G}, s, \eta')$.
\end{proof}

The main purpose of the above lemma is to prove the following result.

\begin{theorem}
\label{adjusted-KL}
The map  $\KL^{\mathrm{adj}}$ is a bijection, and fits into the commutative diagram
\begin{equation*}
\xymatrix@C=3em{\pi_0\cm^{\mathrm{adj}}(\mathfrak{f}, \mathfrak{a}) \ar[d] \ar[r]^-{\KL^{\mathrm{adj}}} & \Sym^2(\mathfrak{f}, \mathfrak{a})^{\ad} \ar[d]^{\mathrm{cw}} \\ \pi_0\cm(\mathfrak{f}, \mathfrak{a}) \ar[r]_-{\KL} & H^3(\mathfrak{f}, \mathfrak{a}) }
\end{equation*}
\end{theorem}

\begin{proof}
The diagram is commutative due to \cref{existence-of-infinitesimal-adjustments}. The bottom horizontal map is a bijection due to \cref{classification-of-crossed-modules}. 
The top horizontal map is surjective again due to  \cref{existence-of-infinitesimal-adjustments}. 
For injectivity of $\KL^{\mathrm{adj}}$, assume $(\mathfrak{G}_1, s_1, \eta_1)$ and $(\mathfrak{G}_2, s_2, \eta_2)$ are objects of $\cm^{\mathrm{adj}}(\mathfrak{f}, \mathfrak{a})$ with $B_{\eta_1}=B_{\eta_2}$. By commutativity, this means that $\KL(\mathfrak{G}_1)=\KL(\mathfrak{G}_2)$; hence, there exists an invertible butterfly $\mathfrak{k}':\mathfrak{G}_1 \to \mathfrak{G}_2$. 
We choose a neat section $q'$ (with respect to $s_1$ and $s_2$) in $\mathfrak{k'}$ and consider $\eta_2' := \Adj^{q'}(\mathfrak{k}')(\eta_1)$. 
Then, we have $(s_2, \eta_2')=\ADJ^{q'}(s_1, \eta_1)$, and  $(\mathfrak{k}', q', \id)$ is an isomorphism in $\cm^{\mathrm{adj}}(\mathfrak{f}, \mathfrak{a})$ from  $(\mathfrak{G}_1, s_1, \eta_1)$ to  $(\mathfrak{G}_2, s_2, \eta_2')$. 
As $\eta_1$ and $\eta_2$ have the same symmetric part by assumption, and $\eta_2'$ has the same symmetric part as $\eta_1$ by construction, objects $(\mathfrak{G}_1, s_1, \eta_1)$ and  $(\mathfrak{G}_2, s_2, \eta_2')$ are isomorphic by \cref{LemmaMorphismsInCrMod}.
\end{proof}

\begin{remark}
\Cref{LemmaMorphismsInCrMod} has a two further consequences, which we describe here. 
Suppose $\mathfrak{G}$ is a central crossed module of Lie algebras, with $\mathfrak{a}:=\pi_1\mathfrak{G}$ and $\mathfrak{f}:=\pi_0\mathfrak{G}$.
\begin{enumerate}[(a)]

\item
If $H^2(\mathfrak{f}, \mathfrak{a})$ is non-trivial and  $\mathfrak{G}$ admits adjustments, it happens that two infinitesimal adjustments $\eta, \eta'$ on  $\mathfrak{G}$, adapted to a section $s$ are not isomorphic in the category $\ADJ^{s}(\mathfrak{G})$ whereas the adjusted crossed modules $(\mathfrak{G}, s, \eta)$ and $(\mathfrak{G}, s, \eta')$ are isomorphic in the bicategory $\cm^{\mathrm{adj}}$. 
This is because is a non-trivial self-butterfly of $\mathfrak{G}$ mapping $\eta$ to $\eta'$, see \cref{LemmaMorphismsInCrMod} (b).

\item
Let $\eta_1$, $\eta_2$ be two infinitesimal adjustments on $\mathfrak{G}$ with $\KL^{\mathrm{adj}}(\mathfrak{G}, \eta_1) = \KL^{\mathrm{adj}}(\mathfrak{G}, \eta_2)$, adapted to two possibly different sections.
Then there exists an invertible butterfly $\mathfrak{k} : \mathfrak{G} \to \mathfrak{G}$ inducing the identity on $\pi_0(\mathfrak{G})$ and $\pi_1(\mathfrak{G})$, together with a neat section $q$ such that $\ADJ^q(\mathfrak{k})$ sends $\eta_1$ to $\eta_2$.
\end{enumerate}
\end{remark}

\section{Constructions of crossed modules with prescribed adjusted Kassel-Loday class}

In this section, we provide  constructions of crossed modules and adjustments, starting from a bilinear form $B \in \Sym^2(\mathfrak{f}, \mathfrak{a})^{\ad}$. In \cref{SectionConstructionCrossedModule1}, $\mathfrak{f}$ and $\mathfrak{a}$ may be any finite-dimensional Lie algebras, with $\mathfrak{a}$ abelian, and our construction produces an adjusted crossed module of Lie algebras with adjusted Kassel-Loday class $B$.
In \cref{SectionConstructionCrossedModule2}, $\mathfrak{f}$ and $\mathfrak{a}$ are the Lie algebras of certain Lie groups $F$ and $A$, and our construction produces a crossed module of Lie groups with homotopy groups $F$ and $A$, together with an adjustment, reducing under differentiation to the structures of \cref{SectionConstructionCrossedModule1}.

\subsection{Construction of a crossed module of Lie algebras}
\label{SectionConstructionCrossedModule1}

Let $\mathfrak{f}$ be a Lie algebra and let $\mathfrak{a}$ an abelian Lie algebra, both finite-dimensional.
In this section, we will construct from the input datum of an $\ad$-invariant symmetric bilinear form $B \in \Sym^2(\mathfrak{f}, \mathfrak{a})^{\ad}$ an adjusted crossed module $(\mathfrak{G}, s, \eta)$ of Lie algebras with $\pi_0(\mathfrak{G}) = \mathfrak{f}$, $\pi_1(\mathfrak{G}) = \mathfrak{a}$ and adjusted Kassel-Loday class 
\[
\KL^{\mathrm{adj}}(\mathfrak{G}, \eta) = B.
\]
In other words, we construct an explicit inverse of the map $\KL^{\mathrm{adj}}$ from  \eqref{adjusted-KL}.

Denote by $P_0 \mathfrak{f}$ the space of smooth paths $f: [0, 1] \to \mathfrak{f}$ such that $f(0) =0$.
It is a Fr\'echet Lie algebra with the pointwise defined Lie bracket.
We define an $\mathfrak{a}$-valued bilinear form $\tilde{\eta}$ on $P_0\mathfrak{f}$ by the formula
\begin{equation}
\label{FormulaEtaTilde}
\tilde{\eta}_B(f, g) = - 2\int_0^1 B\big(f'(t), g(t)\big) dt.
\end{equation}
By $\ad$-invariance of $B$, we have $\tilde\eta_B \in T(P_0\mathfrak{f}, \mathfrak{a})$, with the symmetric part of $\tilde\eta_B$ given by
\begin{align}
\label{SymmetricPartEtaTilde}
\tilde\eta_B^{\sym}(f, g) = \frac{1}{2} \cdot \big(\tilde\eta_B(f, g) + \tilde\eta_B(g, f)\big)
&= - \int_0^1 \frac{d}{dt} B\big(f(t), g(t)\big) dt =  - B\bigl(f(1), g(1)\big).
\end{align}
%
\begin{comment}
To see that $\tilde\eta_B \in T(P_0\mathfrak{f}, \mathfrak{a})$, we calculate
\begin{align*}
\eta([f, g], h) + \eta(g, [f, h]) - \eta(f, [g, h])
&= -\int_0^1 \Bigl( B\big([f'(t), g(t)], h(t)\big) + B\big([f(t), g'(t)], h(t)\big) 
\\
&\qquad + B\bigl(g'(t), [f(t), h(t)]\big) - B\bigl(f'(t), [g(t), h(t)]\big) \Big)dt.
\end{align*}
The right hand side vanishes using anti-symmetry of the Lie bracket and $\ad$-invariance of $B$. 

\end{comment}

Denote by $L_0 \mathfrak{f} \subseteq P_0\mathfrak{f}$ the subspace of those paths $f$ that also satisfy $f(1) = 0$.
By the above calculation, $\tilde\eta_B$ is anti-symmetric on $L_0\mathfrak{f}$ and by the exact sequence \eqref{FunnyFourTermSequence}, is closed with respect to the Chevalley-Eilenberg differential.
Hence, we obtain a central Lie algebra extension
\[
\widetilde{L_0\mathfrak{f}} := L_0\mathfrak{f} \ltimes \mathfrak{a},
\]
with Lie bracket
\begin{equation}
\label{DefinitionLieBracket}
\big[(f, a), (g, b)\big] = \big( [f, g], \tilde\eta_B(f, g)\big),
\end{equation}
together with an action of $P_0\mathfrak{f}$ on $\widetilde{L_0\mathfrak{f}}$ by a similar formula,
\begin{equation}
\label{ActionByEtaTilde}
\alpha_*\bigl(f, (g, b)\big) = \big([f, g], \tilde\eta_B(f, g)\big).
\end{equation}
\begin{comment}
This is indeed an action:
\begin{align*}
\alpha_*\bigl(f, \alpha_*\bigl(g, (h, a)\big)\big) - \alpha_*\bigl(g, \alpha_*\bigl(f, (h, a)\big)\big) 
&=
\alpha_*\bigl(f, ([g, h], \tilde\eta_B(g, h)\big) - \alpha_*\bigl(g, ([f, h], \tilde\eta_B(f, h)\big)
\\
&= \big([f, [g, h]], \tilde\eta_B(f, [g, h])\big) - \big([g, [f, h]], \tilde\eta_B(g, [f, h])\big)
\\
&= \big([f, [g, h]] - [g, [f, h]], \tilde\eta_B(f, [g, h]) - \tilde\eta_B(g, [f, h])\big)
\\
&= \big(-[h, [f, g]], \tilde\eta_B([f, g], h)\big)
\\
&= \alpha_*\big([f, g], (h, a)\big)
\end{align*}
\end{comment}
Defining $t_*: \smash{\widetilde{L_0\mathfrak{f}}} = L_0\mathfrak{f} \ltimes \mathfrak{a} \to P_0\mathfrak{f}$ as the projection onto the first coordinate followed by the inclusion, we have the Peiffer identities
\begin{align*}
t_*\alpha_*\bigl(f, (g, b)\bigr) &= \bigl[f, t_*(g, b)\bigr]
\\
\alpha_*\bigl(t_*(f, a), (g, b)\bigr) &= \big[(f, a), (g, b)\big].
\end{align*}
Thus,
\begin{equation}
\label{PathLieAlgebraCrossedModule}
\mathfrak{G}_B := \Bigl(\widetilde{L_0\mathfrak{f}} \stackrel{t_*}{\longrightarrow}  P_0\mathfrak{f} \stackrel{\alpha_*}{\longrightarrow} \Der(\widetilde{L_0\mathfrak{f}})\Bigr)  
\end{equation}
is a crossed module of Lie algebras. 
It is central, and we have $\pi_0(\mathfrak{G}_B) = \mathfrak{f}$ with $p_* = \mathrm{ev}_1 : P_0\mathfrak{f} \to \mathfrak{f}$ and $\pi_1(\mathfrak{G}_B) = \mathfrak{a}$.

To construct an adapted adjustment on $\mathfrak{G}_B$, we note that a section is a linear map $s: \mathfrak{f} \to P_0\mathfrak{f}$ such that $s(x)(1) = x$.
For example, $s_0(x)(t) := tx$ is a canonical choice. 
A further class of possible choices is $s_{\psi}(x)(t) := \psi(t) x$, where $\psi: [0,1] \to [0,1]$ is any smooth map with $\psi(0)=0$ and $\psi(1)=1$; this reduces to the canonical section for $\psi=\id$.

\begin{theorem}
\label{adjustment-from-B}
Let $\mathfrak{f}$ and $\mathfrak{a}$ be finite-dimensional Lie algebras, with $\mathfrak{a}$ abelian.
Let $B \in \mathrm{Sym}^2(\mathfrak{f}, \mathfrak{a})^{\mathrm{ad}}$, and let $s$ be a section of the associated crossed module  $\mathfrak{G}_B$. 
Then, the formula
\[
\eta_{B,s}(f, g) := \Bigl( [f, g] - s\big( [f(1), g(1)]\big), \tilde\eta_B(f, g)\Bigr)
\]
defines an infinitesimal adjustment on $\mathfrak{G}_B$ adapted to the section $s$ 
such that
\[
\KL^{\mathrm{adj}}(\mathfrak{G}_B, \eta_{B,s}) = B\text{.}
\]
\end{theorem}

\begin{proof}
The identity \eqref{AdaptedEquationKappastar} is obvious, as $\rho_s(f) = f - s(f(1))$.
The identities \eqref{EtaIdentities-2} and \eqref{EtaIdentities-1} follow immediately from the formula \eqref{ActionByEtaTilde} for the action.
In the $L_0\mathfrak{g}$ component, \eqref{EtaIdentities-0} follows from the Jacobi identity.
In the $\mathfrak{a}$ component, \eqref{EtaIdentities-0} follows from the fact that $\tilde\eta_B$ is contained in $T(P_0\mathfrak{g}, \mathfrak{a})$.

The symmetric part of $\eta$ equals the symmetric part of $\tilde\eta_B$, which by \eqref{SymmetricPartEtaTilde} is precisely the pullback of $- B$ under $\mathrm{ev}_1 : P_0\mathfrak{f} \to \mathfrak{f}$.
Hence $\KL^{\mathrm{adj}}(\mathfrak{G}, \eta) = B$. 
\end{proof}

\begin{remark}
\label{RemarkKLclass}
A consequence of \cref{adjustment-from-B}, following from \cref{existence-of-infinitesimal-adjustments}, is that 
\begin{equation*}
\mathrm{KL}(\mathfrak{G}_B)=[\cw(B)]\text{.}
\end{equation*}
This can also be checked manually:
Every section $s$ of $\mathfrak{G}_B$ has a  canonical extension to a splitting $u_s: P_0\mathfrak{f} \to \smash{\widetilde{L_0\mathfrak{f}}}$,  namely
\[
u_s(f) = \big(f - s(f(1)), 0\big),
\]
and one may calculate that  the 2-form $\omega_{u_s}$ from \eqref{omegaU} is 
\[
\omega_{u_{s}}(f, g) = \Bigl([f, g] - s\big([f(1), g(1)]\big), \tilde\eta_B(f, g) - \tilde\eta_B\big(s(f(1)), s(g(1))\bigr) \Bigr).
\]
\begin{comment}
\begin{align*}
\omega_u(f, g)
&= \alpha_*\bigl(f, u(g)\bigr) - \alpha_*\bigl(g, u(f)\bigr) - \bigl[u(f), u(g)\bigr] + u\bigl([\rho^\perp(f), \rho^\perp(g)]\bigr)
\\
&= \big( [f, g - s(g(1))], \tilde\eta_B(f, g - s(g(1)))\big) - \bigl([g, f - s(f(1))], \tilde\eta_B(g, f - s(f(1)))\bigr) 
\\
&\qquad - \bigl[(f - s(f(1)), 0), (g - s(g(1)), 0)\bigr] + u \big([s(f(1)), s(g(1))]\big)
\\
&= \Bigl( [f, g] - [f, s(g(1))] - [g, f] + [g, s(f(1))] , \tilde\eta_B(f, g) - \tilde\eta_B(g, f) - \tilde\eta_B(f, s(g(1))) + \tilde\eta_B(g, s(f(1)))\Bigr)
\\
&\qquad -\Bigl( [f, g] - [s(f(1)), g] - [f, s(g(1))] + [s(f(1)), s(g(1))], \tilde\eta_B(f - s(f(1)), g- s(g(1)))\Bigr)
\\
&\qquad + \Bigl([s(f(1)), s(g(1))] - s([f(1), g(1)]), 0\Bigr)
\\
&= \Bigl([f, g] - s([f(1), g(1)]), - \tilde\eta_B(g, f) + \tilde\eta_B\bigl(g, s(f(1)\bigr) + \tilde\eta_B\bigl(s(f(1)), g\bigr) - \tilde\eta_B\bigl(s(g(1)), s(g(1))\bigr) \Bigr)
\\
&= \Bigl([f, g] - s([f(1), g(1)]), \tilde\eta_B(f, g) - 2\tilde\eta_B^{\mathrm{s}}(f, g) + 2\tilde\eta_B^{\mathrm{s}}(s(f(1)), g) - \tilde\eta_B\bigl(s(g(1)), s(g(1)) \Bigr). 
\end{align*}
The symmetric terms vanish because the symmetric part of $\tilde\eta_B$ only depends on the end points of the paths.
\end{comment}
%
For the canonical section $s_0$, we get 
\[
\tilde\eta_B\bigl(s_0(f(1)), s_0(g(1))\bigr) = -2\int_0^1 B\bigl(f(1), tg(1)\big)dt = -B\bigl(f(1), g(1)\big) = \tilde\eta_B^{\sym}(f, g),
\]
hence
\[
\omega_{u_{s_0}}(f, g) = \Bigl([f, g] - s\big([f(1), g(1)]\big), \tilde\eta_B^{\anti}(f, g) \Bigr).
\]
Further, the corresponding representative of the Kassel-Loday class of $\mathfrak{G}_B$ is
\begin{equation*}
C_{u_{s_0}}=\cw(B)\text{.}
\end{equation*}
For an arbitrary section $s$, we define
\begin{align*}
\tilde{\theta}_s(f, g)&:=\tilde\eta_B\bigl(s_0(f(1)), s_0(g(1))\bigr) - \tilde\eta_B\bigl(s(f(1)), s(g(1))\bigr)\text{,}
\end{align*}
 which descends to an element $\theta_s\in \Alt^2(\mathfrak{f},\mathfrak{a})$. 
\begin{comment}
$\tilde\theta_s$ vanishes on $L_0\mathfrak{f}$ and is anti-symmetric, as by \eqref{SymmetricPartEtaTilde} the symmetric part of $\tilde\eta_B$ only depends on the values of the arguments at the end points.
\end{comment}
Then, we get
 \begin{equation*}
C_{u_{s}}= C_{u_{s_0}}+\delta\theta_s\text{.}
\end{equation*}
For instance, one can check that $\theta_{\psi}=0$ for any smooth map $\psi:[0,1] \to [0,1]$ with $\psi(0)=0$ and $\psi(1)=1$, so that $C_{u_{s_{\psi}}}=\cw(B)$.  
\begin{comment}
        If $\psi$ is any smooth function with $\psi(0) = 0$ and $\psi(1) = 1$, then we have
        \begin{align*}
        \tilde\eta_B\bigl(s(f), s(g)\bigr)      
        &= - 2\int_0^1 B\bigl( \psi'(t)\cdot f, \psi(t) \cdot g\bigr) dt = -2 B(f, g) \cdot \int_0^1 \frac{1}{2} \bigl(\psi(t)\bigr)'dt = - B(f, g).
        \end{align*}
   This is independent of $\psi$, hence the difference of these terms for different $\psi$ vanish.
\end{comment}
\end{remark}

\begin{remark}
\label{uniqueness-of-adjustments-on-GB}
Depending on $\mathfrak{f}$ and $\mathfrak{a}$, the adjustments $\eta_{B,s}$ of \cref{adjustment-from-B} are in general not the only adjustments on $\mathfrak{G}_B$. 
By \cref{homotopy-groups-of-ADJ}, the preimage of $B$ under the map  $\pi_0\ADJ(\mathfrak{G}_B) \to \mathrm{Sym}^2(\mathfrak{f}, \mathfrak{a})^{\mathrm{ad}}$ is an affine space over $H^2(\mathfrak{f}, \mathfrak{a})$, and the automorphism group of the object $(s, \eta_{B,s})$ in the groupoid $\ADJ(\mathfrak{G}_B)$ is $H^1(\mathfrak{f}, \mathfrak{a})$. However, if $s$ and $s'$ are arbitrary sections, then $(s, \eta_{B,s})$ and $(s', \eta_{B,s'})$ are canonically isomorphic in $\ADJ(\mathfrak{G}_B)$. 
\begin{comment}
Indeed, define $\phi': \mathfrak{f} \to L_0\mathfrak{f}$ by $\phi'(x) := s(x)-s'(x)$. Then, 
define $\phi: \mathfrak{f} \to \widetilde{L_0\mathfrak{f}}=L_0\mathfrak{f} \ltimes \mathfrak{a}$ by $\phi = (\phi', 0)$. This satisfies $s-s'=t_{*}\phi$. Moreover, 
\begin{equation*}
(\eta_{B,s'}-\eta_{B,s})(f,g) = (-s'([f(1), g(1)] ) + s([f(1), g(1)]) , 0)= (\delta p^{*}\phi'(f, g), 0)=\delta p^{*}\phi(f, g)\text{.}
\end{equation*}
This shows that $\phi$ is an isomorphism. 
\end{comment}   
\end{remark}

\begin{remark}
\label{flat-end-points}
In relation with Lie 2-groups, it is relevant to consider a variation on the above construction where one considers paths that are \emph{flat} at the end points, i.e., have all derivatives vanish at the end points of the interval $[0, 1]$; see \cite{LudewigWaldorf2Group}.
This yields a new crossed module
\begin{equation}
\label{PathLieAlgebraCrossedModuleflat}
\mathfrak{G}_B^{\mathrm{fl}} = \Bigl(\widetilde{L_0^{\mathrm{fl}}\mathfrak{f}} \stackrel{t_*}{\longrightarrow}  P_0^{\mathrm{fl}}\mathfrak{f} \stackrel{\alpha_*}{\longrightarrow} \Der(\widetilde{L_0^{\mathrm{fl}}\mathfrak{f}})\Bigr) \end{equation}
The diagram
\begin{equation*}
\xymatrix{\widetilde{L_0^{\mathrm{fl}}\mathfrak{f}} \ar[d] \ar[r]^{t_*} &  P_0^{\mathrm{fl}}\mathfrak{f} \ar[d]\\ \widetilde{L_0\mathfrak{f}}\ar[r]_{t_{*}} & P_0\mathfrak{f}\text{,}}
\end{equation*}
whose vertical arrows are the inclusion maps, is a strict intertwiner from $\mathfrak{G}^{\mathrm{fl}}_B$ to $\mathfrak{G}_B$, which induces the identity on $\pi_0$ and $\pi_1$. Hence, $\mathfrak{G}_B^{\mathrm{fl}}$ and $\mathfrak{G}_B$ are weakly equivalent crossed modules of Lie algebras, and have -- in particular -- the same Kassel-Loday class.
The disadvantage of $\mathfrak{G}_B^{\mathrm{fl}}$ is that the canonical section $s_0$ of $\mathfrak{G}_B$ does not map into $P_0^{\mathrm{fl}}\mathfrak{f}$, so that $\mathfrak{G}_B^{\mathrm{fl}}$ does not possess a canonical section. 
However, if $\psi : [0, 1] \to [0, 1]$ is a smooth function with $\psi(0) =0$ and $\psi(1) = 1$ that is flat at the end points, then $s_{\psi}$ is a section of $\mathfrak{G}_B^{\mathrm{fl}}$, and the discussion in \cref{RemarkKLclass} shows that $\smash{C_{u_{s_{\psi}}}=\mathrm{cw}(B)}$ is  the corresponding representative of the Kassel-Loday class of $\mathfrak{G}_B^{\mathrm{fl}}$.
\\
Concerning the adjustments, let $s$ be any section of $\mathfrak{G}_B^{\mathrm{fl}}$, which is then also a section of $\mathfrak{G}_B$, and set $\eta^{\mathrm{fl}}_{B,s}:=\eta_{B,s}|_{P_0^{\mathrm{fl}}\mathfrak{f} \times P_0^{\mathrm{fl}}\mathfrak{f}}$.
Then $(\mathfrak{G}_B^{\mathrm{fl}}, s, \eta_{B,s}^{\mathrm{fl}})$ is an object in $\cm^{\mathrm{adj}}(\mathfrak{f}, \mathfrak{a})$ and by \cref{adjustemnts-and-strict-intertwiners}, is isomorphic to $(\mathfrak{G}_B, s, \eta_{B,s})$.
In particular, $\smash{\KL^{\mathrm{adj}}(\mathfrak{G}_B^{\mathrm{fl}}, \eta_{B,s}^{\mathrm{fl}})=B}$. Finally, since the groupoids of adjustments of $\mathfrak{G}_B$ and $\mathfrak{G}_B^{\mathrm{fl}}$ are equivalent by \cref{equivalent-crossed-modules-have-equivalent-groupoids-of-adjustments}, all statements of \cref{uniqueness-of-adjustments-on-GB} about $\ADJ(\mathfrak{G}_B)$ also hold for $\ADJ(\mathfrak{G}_B^{\mathrm{fl}})$.
\end{remark}

\begin{remark}
It is not a coincidence that the Lie algebras $\mathfrak{g}$ and $\mathfrak{h}$ in the  construction of the crossed module $\mathfrak{G}_B$ are infinite-dimensional.
By \cite[Theorem~5]{Hochschild1954}, a 3-class $C \in H^3(\mathfrak{f}, \mathfrak{a})$ is the Kassel-Loday class of a \emph{finite-dimensional} crossed module if and only if its restriction to any semisimple subalgebra $\mathfrak{s} \subseteq\mathfrak{f}$ vanishes. 
Further, the additional requirement that $C$ is in the image of the Chern-Weil homomorphism in order for the crossed module to admit an adjustment rules out many non-trivial classes satisfying these conditions, for example when $\mathfrak{f}$ is abelian.
We do not know if there exists a finite-dimensional crossed module that admits an adjustment but has non-trivial Kassel-Loday class.
\end{remark}

\subsection{Construction of a crossed module of Lie groups}
\label{SectionConstructionCrossedModule2}

Let $F$ and $A$ be connected Lie groups, $A$ abelian, with Lie algebras $\mathfrak{f}$ and $\mathfrak{a}$, and let $B \in \Sym^2(\mathfrak{f}, \mathfrak{a})^{\ad}$ be given. 
We are now looking for an integral version of \cref{adjustment-from-B}: we ask if  there is a crossed module $\Gamma$ of Lie groups, with $\pi_0(\Gamma) = F$ and $\pi_1(\Gamma) = A$, equipped with a section $s$ of its corresponding crossed module $\mathfrak{G}$ of Lie algebras, and further equipped with an  adjustment $\kappa$ adapted to $s$, such that $\KL^{\mathrm{adj}}(\mathfrak{G}, \kappa_{*})=B$.

One cannot expect that the answer is always positive: a necessary condition is that  $[\cw(B)]\in H^3(\mathfrak{f}, \mathfrak{a})$ can be realized as the Kassel-Loday class of (the crossed module of Lie algebras corresponding to) a crossed module of Lie groups. 
In terms of the classification of crossed modules by Lie group cohomology (see \cref{smooth-group-cohomology-classification}), the obstruction is that $[\mathrm{cw}(B)]$ must be the image of the map $H^3(F, A) \to H^3(\mathfrak{f}, \mathfrak{ a})$.

We may identify  $A \cong \mathfrak{a}/\Lambda$, where $\Lambda \subseteq \mathfrak{a}$ is a lattice.
We consider, for  $C \in \Alt^3_{\mathrm{cl}}(\mathfrak{f}, \mathfrak{a})$,  the group of periods 
\[
\mathrm{Per}(C) := \Bigl\{ \int_Z \bar{C} ~\Big|~ Z~\text{is a smooth singular 3-cycle in $F$} \Big\} \subset \mathfrak{a}\text{,}
\]
where $\bar{C}$ is the invariant $3$-form on $F$ obtained from $C$ by left translation.
In this section, we prove the following theorem, showing that it is sufficient to require that the group of periods of $\cw(B)$ is \quot{integral}.

\begin{theorem}
\label{ThmExistenceLieGroups}
Let $F$ be a connected, simply connected, and finite-dimensional Lie group with Lie algebra $\mathfrak{f}$,  and let $A$ be a connected, abelian, and finite-dimensional Lie group with Lie algebra $\mathfrak{a}$. 
Let $\Lambda\subset \mathfrak{a}$ be a lattice such that $A\cong \mathfrak{a}/\Lambda$. 
Let $B \in \Sym^2(\mathfrak{f}, \mathfrak{a})^{\ad}$ such that 
\begin{equation}
\label{PerCondition}
        \mathrm{Per}(\cw(B)) \subseteq \Lambda.
\end{equation}
Then, there exists a crossed module $\Gamma_B$ of Lie groups, with $\pi_0(\Gamma_B) = F$ and $\pi_1(\Gamma_B) = A$, with the following property: for every section $s$ of its corresponding crossed module $\mathfrak{G}$ of Lie algebras, there exists an  adjustment $\kappa_{B,s}$ on $\Gamma_B$ adapted to $s$, such that $\KL^{\mathrm{adj}}(\mathfrak{G}, \kappa_{*})=B$.
\end{theorem}

To construct the crossed module $\Gamma$, we follow the method of \cite{LudewigWaldorf2Group}, which we recall now.
For a subinterval $I \subseteq \R$, we denote by $L_I F$ the set of smooth maps $\gamma : \R \to F$ supported in $I$, meaning that $\gamma(t) = e$ whenever $t \notin I$. 
This is an infinite-dimensional Lie group with Lie algebra $L_I\mathfrak{f}$, defined analogously.
We consider a adapted version of the 2-cocycle $\tilde\eta_B$ of \cref{FormulaEtaTilde}, given by
\[
\tilde\eta_{B,I}(f, g) := - 2 \int_\R B\bigl(f'(t), g(t)\bigr) dt\text{.}
\]
We remark that $\tilde\eta_{B,[0,1]}=\iota^{*}\tilde\eta_B$, with $\iota: L_{[0,1]}\mathfrak{f} \to P_0\mathfrak{f}$ the inclusion. 

\begin{lemma}
\label{LemmaUniqueCentralExtension}
Let $I \subseteq \R$ be a subinterval.
Under the assumptions of \cref{ThmExistenceLieGroups}, there exists a  central extension 
\begin{equation}
\label{LieGroupCentralAExtension}
0 \longrightarrow A \longmapsto \widetilde{L_I F} \longrightarrow L_I F \longrightarrow 0
\end{equation}
such that $\tilde\eta_{B,I}$ is a classifying cocycle of the corresponding Lie algebra extension.
Moreover, this central extension is unique up to  isomorphism.
\end{lemma}

\begin{proof}
For every (possibly infinite-dimensional) connected Lie group $L$ with Lie algebra $\mathfrak{l}$, the group $\mathrm{Ext}(L, A)$ of central extensions of $L$ by $A$ fits into an exact sequence
\begin{equation*}
\Hom(\pi_1(L), A) \longrightarrow \mathrm{Ext}(L, A) \longrightarrow H^2(\mathfrak{l}, \mathfrak{a}) \longrightarrow \Hom(\pi_2(L), A) \times \Hom(\pi_1(L), \mathrm{Lin}(\mathfrak{l}, \mathfrak{a})),
\end{equation*}
see \cite{Ne02}.
Here, the second map sends a central extension to the class of its corresponding Lie algebra extension, and the first component of the third map, $H^2(\mathfrak{l}, \mathfrak{a}) \to \Hom(\pi_2(L), A)$, sends a cocycle $\omega$ to the homomorphism $\pi_2(L) \to A$ obtained by integrating the invariant 2-form $\overline{\omega}$ associated to $\omega$ over smooth representatives and then applying the quotient map $\mathfrak{a} \to \mathfrak{a}/\Lambda \cong A$.

We apply this result to $L = L_I F$. 
By standard arguments, we have isomorphisms $\pi_k(L_IF) \cong \pi_{k+1}(F)$ for all $k\geq 0$.
 Thus, $L$ is connected since $F$ is simply connected. 
As $F$ is finite-dimensional, we have $ \pi_1(L_IF) \cong\pi_2(F)= 0$; hence, the first map in the exact sequence is zero, making the second map injective and thus showing the claimed uniqueness. 
Moreover, the range of the third map is just $\Hom(\pi_2(L), A)$.

We claim that the condition \eqref{PerCondition} implies that the third map in the sequence sends $[\tilde\eta_B]$ to the zero element of $\Hom(\pi_2(L_IF), A)\cong \mathrm{Hom}(\pi_3(F), A) $; 
equivalently, that we have
\[
\int_{S^2} \gamma^* \overline{\tilde{\eta}_B} \in \Lambda
\]
for each smooth map $\gamma : S^2 \to L_I F$.
By the above discussion, this claim implies the existence of the desired central extension.
To verify the claim, we use that the left invariant form $\overline{\tilde{\eta}_B}$ defined by $\tilde{\eta}_B$ defines the same class in $H^2(L_I F, \R)$ as $-\tau(\overline{\cw(B)})$, where $\tau$ is the transgression map
\begin{equation*}
\xymatrix{
H^3(F, \R) \ar[r]^-{\ev^*} &
 H^3(L_IF \times S^1, \R) 
 \ar[r]^-{\int_{S^1}}
 &
H^2(L_IF, \R).
}
\end{equation*}
This is stated in \cite[Prop.~4.4.4]{PressleySegal} but with incorrect constants; see \cite[Lemma~A.3]{LudewigClifford}.
%
\begin{comment}
We calculate
  \begin{align*}
  \tau(\overline{\cw(B)})_\gamma(X, Y)
   &= \int_{S^1} \overline{\cw(B)}(\gamma^\prime(t), X(t), Y(t)) dt\\
   &= \int_{S^1} \cw(B)\bigl( dL_{\gamma(t)}^{-1}({\gamma}'(t)), dL_{\gamma(t)}^{-1}(X(t)), dL_{\gamma(t)}^{-1}(Y(t))\bigr) dt\\
   &= \int_{S^1}B\bigl(dL_{\gamma(t)}^{-1}({\gamma}'(t)), [dL_{\gamma(t)}^{-1} (X(t)), dL_{\gamma(t)}^{-1} (Y(t))]) dt
  \end{align*}
Let $\beta \in \Omega^1(L_IF)$ be given by
  \begin{equation*}
    \beta_\gamma(X) = \int_{S^1} B(     dL_{\gamma(t)}^{-1}({\gamma}'(t)), d L_{\gamma(t)}^{-1} (X(t)) dt.
  \end{equation*}
For the exterior derivative of $\beta$, we find (for suitable extensions of $X$ and $Y$ to vector fields on $L_I F$)
  \begin{align*}
  d\beta_\gamma(X, Y) &= \partial_X \{ \beta(Y)\}_\gamma - \partial_Y \{ \beta(X)\}_\gamma - \beta_\gamma([X, Y]) \\
  &= \int_{S^1} \Big( B(        dL_{\gamma(t)}^{-1}(X'(t)), d L_{\gamma(t)}^{-1} (Y(t)) - B(    dL_{\gamma(t)}^{-1}(Y'(t)), d L_{\gamma(t)}^{-1} (X(t))
  \\
  &\qquad\qquad - B(dL_{\gamma(t)}^{-1} (\gamma^\prime(t)), dL_{\gamma(t)}^{-1} [X, Y](t))\Big) dt.
  \end{align*}
Integrating by parts, we obtain $d\beta = - \overline{\tilde{\eta}_B} - \tau(\overline{\cw(B)})$.
\end{comment}
%
If now $\gamma : S^2 \to L_IF$ is smooth, we get by definition of the transgression map that
\begin{align}
\int_{S^2} \gamma^* \overline{\tilde{\eta}_B} 
&= -\int_{S^2} \gamma^*\tau(\overline{\cw(B)})\nonumber
\\\nonumber
&= -\int_{S^2} \int_{S^1}(\gamma \times \id)^* \ev^*\overline{\cw(B)}
\\
&= -\int_{S^2 \times S^1} (\gamma^\vee)^*\overline{\cw(B)},
\label{claim-integrality}
\end{align}
where $\gamma^\vee : S^2 \times S^1 \to F$ is the currying of $\gamma : S^2 \to L_I F$, given by 
\[
\gamma^\vee(x, t) = \gamma(x)(t) = \ev(\gamma(x), t) = (\ev \circ (\gamma \times \id))(x, t). 
\]
By condition \eqref{PerCondition}, the last expression of \cref{claim-integrality} gives an element of $\Lambda$. This proves the claim. 
\end{proof}

From now on, $\widetilde{L_IF}$ always denotes the unique central extension specified by \cref{LemmaUniqueCentralExtension}.
We denote by $P_I F$ the set of smooth paths $\gamma : \R \to F$ that are locally constant outside $I$ and satisfy $\gamma(t) = e$ for $t$ to the left of $I$.
The crossed module $\Gamma$ consists of the group homomorphism
\begin{equation}
\label{Lie4TermSeq}
 \widetilde{L_{[0, 1]} F} \stackrel{t}{\longrightarrow} P_{[0, 1]} F 
\end{equation}
obtained as the foot point projection $\smash{\widetilde{L_{[0, 1]}F}} \to L_{[0, 1]}F$ followed by the inclusion $L_{[0, 1]}F \hookrightarrow P_{[0, 1]}F$.

We use the following observation to define a crossed module action: for $I \subseteq J$ an inclusion of subintervals, inclusion yields a Lie group homomorphism $\iota : L_I F \to L_J F$, and pullback along this map yields a central extension $\iota^*\smash{\widetilde{L_JF}}$ of $L_I F$.
However, as $\iota^{*}\tilde\eta_{B,J}=\tilde\eta_{B,I}$, hence the uniqueness of \cref{LemmaUniqueCentralExtension} implies that
\[
\iota^*\widetilde{L_JF} \cong \widetilde{L_IF}\text{.}
\]
There is a Lie group homomorphism
\begin{equation*}
P_{[0, 1]} F \longrightarrow L_{[0, 2]}F,
\qquad
\gamma \longmapsto \gamma \cup \gamma, \quad \text{with}\quad (\gamma\cup\gamma)(t) := \begin{cases}
        \gamma(t) & t \in [0, 1]
        \\
        \gamma(2 - t) & t \in [1, 2]
        \\
        e & t \notin [0, 2],
 \end{cases} 
\end{equation*}
which yields an action of $P_{[0, 1]} F$ on $\widetilde{L_{[0, 2]}F}$ by conjugation,
\begin{equation}
\label{DefinitionOfCrossedModuleAction}
\alpha(\gamma, \Phi) := \widetilde{\gamma \cup \gamma} \cdot \Phi \cdot \widetilde{\gamma \cup \gamma}^{-1},    
\end{equation}
where $\smash{\widetilde{\gamma \cup \gamma}}$ is any lift of $\gamma \cup \gamma$ to the central extension, and \cref{DefinitionOfCrossedModuleAction} does not depend on that choice since the extension is central.
Since conjugation with $\gamma \cup \gamma$ preserves the subgroup $L_{[0, 1]} F \subseteq L_{[0, 2]} F$, the action restricts to an action on the central extension $\iota^*\smash{\widetilde{L_{[0, 2]}F}} \cong \smash{\widetilde{L_{[0, 1]}F}}$.
It is obvious that the map $t$ from \eqref{Lie4TermSeq} intertwines this action with the conjugation action of $P_{[0, 1]} F$ on $L_{[0, 1]}F$.
However, as observed in \cite[Lemma 3.2.2]{LudewigWaldorf2Group}, for the action to be a crossed module action, we need the central extension $\smash{\widetilde{L_{[0, 1]}F}}$ to be \emph{disjoint commutative}.
Luckily, we have the following lemma.

\begin{lemma}
The central extension $\widetilde{L_{[0, 1]} F}$ is disjoint commutative.
\end{lemma}

\begin{proof}
We need to show that if two loops $\gamma_1, \gamma_2 \in L_{[0, 1]}F$ have disjoint supports, then any two lifts $\tilde{\gamma}_1, \tilde{\gamma}_2$  to the central extension commute in $\smash{\widetilde{L_{[0, 1]}F}}$.
Let $I, J \subseteq [0, 1]$ be two disjoint subintervals.
We define a map 
\[
b : L_IF \times L_J F \longrightarrow A, \qquad b(\gamma_1, \gamma_2) = \tilde{\gamma}_1 \tilde{\gamma}_2 \tilde{\gamma}_1^{-1} \tilde{\gamma}_2^{-1},
\]
where $\tilde{\gamma}_1$, $\tilde{\gamma}_2$ are arbitrary lifts of $\gamma_1$, respectively $\gamma_2$ to the central extension. 
It is easy to see that this is independent of the choice of lift and in fact a bihomomorphism (see \cite[Proof of Lemma 2.4.2]{LudewigWaldorf2Group}).
In the following we  show that this bihomomorphism is trivial; this implies disjoint commutativity. 

Since $L_IF$ and $L_J F$ are connected, any  bihomomorphism is determined uniquely by the corresponding Lie algebra map $b_* : L_I \mathfrak{f} \times L_J \mathfrak{f} \to \mathfrak{a}$, which in this case is given by
\[
 b_*(X_1, X_2)  = [\tilde{X}_1, \tilde{X}_2].
\]
Here $\tilde{X}_1, \tilde{X}_2$ are lifts of $X_1$, respectively $X_2$ to the central extension.
%
\begin{comment}
Such a bihomomorphism is equivalent to the Lie group homomorphism $L_I F \to \Hom(L_J F, A)$, $\gamma \mapsto b(\gamma, -)$ where the latter has the Lie group structure given by pointwise multiplication.
The corresponding tangent map is
\[
L_I \mathfrak{f} \longrightarrow \Hom(L_J \mathfrak{f}, \mathfrak{a})
\]
which is the same thing as a map $L_I \mathfrak{f}  \times L_J \mathfrak{f} \to \mathfrak{a}$.
Observe here that $\Hom(L_J \mathfrak{f}, \mathfrak{a})$ is indeed the Lie algebra of $\Hom(L_JF, A)$: If $\varphi : L_J \mathfrak{f}\to \mathfrak{a}$ is a linear map, then $\exp(\varphi)(e^X) := \exp(\varphi(X))$ is a Lie group homomorphism, as $\exp(\varphi)(e^Xe^Y) = \exp(\varphi)(e^Z) = \exp(\varphi(Z)) = \exp(\varphi(X) + \varphi(Y)) = \exp(\varphi(X)\exp(\varphi(Y))$, as $Z = X + Y + [X, Y] + \dots$.
\end{comment}
%
Since $\tilde\eta_{B,[0,1]}$ is the classifying cocycle of $\smash{\widetilde{L_{[0, 1]}F}}$, the corresponding Lie algebra central extension $\smash{\widetilde{L_{[0, 1]} \mathfrak{f}}}$ may be identified with the semidirect product $L_{[0, 1]} \mathfrak{f} \ltimes \mathfrak{a}$ in such a way that the Lie bracket is  given by the formula \eqref{DefinitionLieBracket}.
Under this identification, the Lie algebra map $b_*$ is precisely the restriction of $\tilde\eta_{B,[0,1]}$ to $L_I \mathfrak{f} \times L_J \mathfrak{f}$.
However, since $I \cap J = \emptyset$, this restriction is zero.
This implies that the bihomomorphism $b$ is trivial.
\end{proof}

By the above discussion, using the action \eqref{DefinitionOfCrossedModuleAction}, we obtain a crossed module of Lie groups
\[
\Gamma_B = \Bigl( \widetilde{L_{[0, 1]} F} \stackrel{t}{\longrightarrow} P_{[0, 1]} F \stackrel{\alpha}{\longrightarrow} \Aut(\widetilde{L_{[0, 1]} F})\Bigr)
\]
We observe that restriction to the interval $[0, 1]$ identifies $L_{[0, 1]} \mathfrak{f}$ and $P_{[0, 1]} \mathfrak{f}$ with the Lie algebras $L_0^{\mathrm{fl}} \mathfrak{f}$ and $P_0^{\mathrm{fl}}\mathfrak{f}$ from \cref{flat-end-points}.
Moreover, if $\iota : L_{[0, 1]} \mathfrak{f} \to L_{[0, 2]} \mathfrak{f}$ is the inclusion map, then then we have $\iota^*\tilde{\eta}_{B, [0, 2]} = \tilde{\eta}_{B, [0, 1]}$, which in turn is the restriction of the form $\tilde{\eta}_B$ from \eqref{FormulaEtaTilde} on $P_0\mathfrak{f}$ to $L_{[0, 1]} \mathfrak{f}$ under the identification $L_{[0, 1]} \mathfrak{f} \cong L_0^{\mathrm{fl}}\mathfrak{f}$.
We conclude that, under this identification, the Lie algebra crossed module obtained from $\Gamma_B$ by differentiation is isomorphic to the Lie algebra crossed module $\mathfrak{G}_B^{\mathrm{fl}}$ given in \eqref{PathLieAlgebraCrossedModuleflat}.

\begin{comment}
The Liealgebra cocycle classifying $\iota^*\smash{\widetilde{L_{[0, 2]}F}}$ is the pullback of the classifying cocycle of $\smash{\widetilde{L_{[0, 2]}F}}$ under the induced Lie algebra map $\iota_* : L_{[0, 1]} \mathfrak{f} \to L_{[0, 2]} \mathfrak{f}$.
Generally if $I \subseteq J$ is an inclusion of intervals and $\iota_{I, J} : L_I F \to L_J F$ is the inclusion map, then the defining cocycles $\tilde{\eta}_{B, I}$  satisfy the relation
\[
((\iota_{I, J})_*)^* \tilde{\eta}_{B, J} = \tilde{\eta}_{B, I}.
\]
In particular, for $I = [0, 1]$ and $J = [0, 2]$, we obtain that the classifying cocycle of $\iota^*\smash{\widetilde{L_{[0, 2]}F}}$ is the cocycle \eqref{FormulaEtaTilde} classifying the Lie algebra extension involved in the construction of $\mathfrak{G}_B^{\mathrm{fl}}$.
Hence the Lie algebra extension $\iota^*\smash{\widetilde{L_{[0, 2]}F}}$ and $L_0^{\mathrm{fl}} \mathfrak{f} \ltimes \mathfrak{a}$ are isomorphic. 
The action of $L_{[0, 1]}F$ on $\iota^*\smash{\widetilde{L_{[0, 2]}F}}$ given in \eqref{DefinitionOfCrossedModuleAction} induces the Lie algebra action given by
\[
\alpha_*(X, \tilde{Y}) = [ \tilde{X}, \tilde{Y}],
\]
where $\tilde{X}$ is an arbitrary lift of $X$ to the central extension. 
The action of $L_0^{\mathrm{fl}}\mathfrak{f}$ on $L_0^{\mathrm{fl}} \mathfrak{f} \ltimes \mathfrak{a}$ has the same property. 
This shows that the two actions are intertwined by the isomorphism of the central extensions. 
Hence the isomorphism induces a strict isomorphism of the crossed modules.
\end{comment}

In \cref{SectionConstructionCrossedModule1} and \cref{flat-end-points}, we constructed -- to every section $s$ of $\mathfrak{G}_B^{\mathrm{fl}}$, an  infinitesimal adjustment $\eta^{\mathrm{fl}}_{B,s}$ with adjusted Kassel-Loday class $B$ on $\mathfrak{G}_B^{\mathrm{fl}}$.
Since $P_{[0, 1]}F$ is connected and simply connected (actually contractible) and $\smash{\widetilde{L_{[0, 1]}F}}$ is connected, \cref{theorem-1} proves that they all integrate uniquely  to  adjustments $\kappa_{B,s}$ on ${\Gamma_B}$.
This finishes the proof of \cref{ThmExistenceLieGroups}.

\begin{remark}
\label{groupoid-of-adjustments-of-GammaB}
By \cref{theorem-3}, differentiation of adjustments is an equivalence $\ADJ(\Gamma_B) \cong \ADJ(\mathfrak{G}_B^{\mathrm{fl}})$. Thus, the statements of \cref{uniqueness-of-adjustments-on-GB} continue to hold for $\ADJ(\Gamma_B)$: the preimage of $B$ under the map  
\begin{equation*}
\pi_0\ADJ(\Gamma_B) \to \mathrm{Sym}^2(\mathfrak{f}, \mathfrak{a})^{\mathrm{ad}}: [(s, \kappa)] \mapsto \KL^{\mathrm{adj}}(\kappa_{*})
\end{equation*}
is an affine space over $H^2(\mathfrak{f}, \mathfrak{a})$, and the automorphism group of the object $(s, \kappa_{B,s})$ in the groupoid $\ADJ(\Gamma_B)$ is $H^1(\mathfrak{f}, \mathfrak{a})$. Moreover, if $s$ and $s'$ are arbitrary sections of $\mathfrak{G}_B^{\mathrm{fl}}$, then $(s, \kappa_{B,s})$ and $(s', \kappa_{B,s'})$ are canonically isomorphic in $\ADJ(\Gamma_B)$.  
\end{remark}

\section{Examples}

\subsection{Product crossed modules}

\label{Products}

Let $A$ be a finite-dimensional abelian Lie group and let $F$ be an arbitrary finite-dimensional Lie group. 
Let $\mathfrak{a}$ and $\mathfrak{f}$ be their Lie algebras.
Consider the crossed module
\[
BA \times F_{\mathrm{dis}} = (A \stackrel{t}{\to} F \stackrel{\alpha}{\to} \Aut(A)),
\]
where both the map $t$ and the action $\alpha$ are trivial.  Thus, $\pi_0(BA \times F_{\mathrm{dis}})=F$ and $\pi_1(BA \times F_{\mathrm{dis}})=A$. 
This is called a \emph{product crossed module} because it corresponds to the trivial central categorical group extension of $F_{dis}$ by $BA$, representing the trivial element in $H^3(F, A)$ under the classification of \cref{smooth-group-cohomology-classification}.
 Let 
\begin{equation*}
B\mathfrak{a} \times \mathfrak{f}_{\mathrm{dis}}=(\mathfrak{a} \to \mathfrak{f} \to \Der(\mathfrak{a}))
\end{equation*}
be the associated crossed module of Lie algebras.
Since any splitting is zero, we have
\begin{equation*}
\KL(B\mathfrak{a} \times \mathfrak{f}_{\mathrm{dis}}) = 0\text{.}
\end{equation*}

It follows from \cref{theorem-2} that infinitesimal adjustments on $\mathfrak{G}$ exist and form an affine space over $T(\mathfrak{f}, \mathfrak{a})$.
In fact, one checks directly that each element of $T(\mathfrak{f}, \mathfrak{a})$ is an adjustment, and hence
\[
\Adj(\mathfrak{G}) = T(\mathfrak{f}, \mathfrak{a})
\]
Since there is only the zero splitting, any adjustment is adapted.

Each adjustment $\eta \in T(\mathfrak{f}, \mathfrak{a})$ determines an object in the groupoid $\ADJ(\mathfrak{G})$ of adjustments. 
By \cref{theorem-3}, two such objects are isomorphic in this groupoid if and only if they differ by an element of $\delta \Alt^1(\mathfrak{f}, \mathfrak{a}) \subseteq T(\mathfrak{f}, \mathfrak{a})$, hence
\[
\pi_0 \ADJ(\mathfrak{G}) = T(\mathfrak{f}, \mathfrak{a})/\delta \Alt^1(\mathfrak{f}, \mathfrak{a}).
\]
Moreover, the automorphism group in $\ADJ(\mathfrak{G})$ of each adjustment is precisely $H^1(\mathfrak{f}, \mathfrak{a})$.

If $F$ is connected, then for each $\eta \in T(\mathfrak{f}, \mathfrak{a})$, there exists at most one adjustment $\kappa$ on $BA \times F_{\mathrm{dis}}$ with $\kappa_* =\eta$ and if $F$ is additionally simply connected and $A$ is connected, such a $\kappa$ always exists.

\begin{example}
For each abelian Lie group $A$, there exists a unique adjustment on $BA$.       
\end{example}

\begin{example}
The zero element of $T(\mathfrak{f}, \mathfrak{a})$ is an infinitesimal adjustment which always integrates to an adjustment.
More generally, if $u : \mathfrak{f} \to \mathfrak{a}$ is an arbitrary linear map, then
\[
\kappa(g, X) = u(\Ad_g(X) - X)
\]
is an adjustment on $BA \times F_{\mathrm{dis}}$ with corresponding infinitesimal adjustment
\[
\eta(X, Y) = u([X, Y]) = - \delta u(X, Y),
\]
which is isomorphic to zero in $\ADJ(\mathfrak{G})$.
\end{example}

\subsection{The string 2-group}

\label{string-2-group}

The string group $\String(n)$, $n \geq 5$, arises from the construction in \cref{SectionConstructionCrossedModule2} by setting $F = \Spin(n)$, $A = \U(1)$ and 
\[
B(x, y) = -\frac{1}{8\pi^2} \tr (xy)\text{.}
\]
We identify $A \cong \R/\Z$, and have 
\begin{equation*}
\cw(B)(x,y,z) = -\frac{1}{8\pi^2}\mathrm{tr}(x, [y, z]) \in \Alt^3(\mathfrak{spin}(n), \R)\text{;}
\end{equation*}
the normalization is chosen such that the corresponding invariant 3-form $\overline{\cw(B)}$ on $\Spin(n)$ is the image of a generator of $H^3(\Spin(n), \Z) \cong \Z$ under the map $H^3(\Spin(n), \Z) \to H^3(\Spin(n), \R)$. 
This guarantees that
$\cw(B)$ satisfies the periodicity condition of \cref{ThmExistenceLieGroups}.
The crossed module $\String(n):=\Gamma_B$ provided by \cref{ThmExistenceLieGroups} is the string 2-group in the version considered in  \cite{LudewigWaldorf2Group}, and $\mathfrak{string}(n) := \mathfrak{G}_B^{\mathrm{fl}}$ is the corresponding Lie 2-algebra.

\begin{proposition}
The string 2-group $\String(n)$ admits, for each section $s$ of $\mathfrak{string}(n)$, an  adjustment  adapted to $s$. The groupoid of adjustments is, for arbitrary $s$, equivalent to  the trivial groupoid on the single object $(s, \kappa_{B,s})$:
\begin{equation*}
\ADJ(\String(n)) \cong \{(s, \kappa_{B,s})\}_{dis}\text{.}
\end{equation*}  
\end{proposition}

\begin{proof}
\Cref{ThmExistenceLieGroups} constructs the required adjustment $\kappa_{B,s}$ showing the first claim.
Since $\Spin(n)$ is simple, the Chern-Weil homomorphism is actually an isomorphism (combine \cite[Lemme~11.1]{Koszul} with \cite[Thm.~21.1]{ChevalleyEilenberg}), so that there exist no other preimages of $\KL(\mathfrak{string}(n))$ other than $B$.  
Thus, by  \cref{homotopy-groups-of-ADJ}, $\pi_0\ADJ(\String(n))$ is an affine space over $H^2(\mathfrak{spin}(n) , \R)$, which, again by simplicity of $\mathrm{Spin}(n)$, vanishes. Since also $H^1(\mathfrak{spin}(n), \R) = 0$, we have $\mathrm{Aut}(s, \kappa_{B,s})=0$. 
\end{proof}

\subsection{Categorical tori}
\label{ex:TD}
\label{ex:adj:TD}

We consider $n\in \N$ and a bilinear form $J: \R^{n} \times \R^{n} \to \R$, which is \emph{integral} in the sense that it restricts to a bilinear form $\Z^{n} \times \Z^{n} \to \Z$.
From this data, we set up the crossed module
\begin{equation*}
\EuScript{T}_J := (\mathbb{T} \times \Z^{n} \stackrel t\to \R^{n}\stackrel\alpha\to\Aut(\mathbb{T} \times \Z^{n}))\text{,}
\end{equation*}
where $\mathbb{T}:= \R/\Z$, and 
\begin{align*}
t(s, m) &:= m
\\
\alpha_a(s, m) &:= (s+[J(a, m)], m)\text{.}
\end{align*}
It is a central crossed module with $\pi_0\EuScript{T}_J =\T^{n}$ and $\pi_1\EuScript{T}_J =\mathbb{T}$; it may be viewed as a central extension
\begin{equation*}
1 \to B\T \to \EuScript{T}_J \to \T^{n}\to 1
\end{equation*} 
and has been introduced and studied by Ganter \cite{Ganter2014}, who proved that it is classified, in the sense of \cref{smooth-group-cohomology-classification}, by the class in $H^3(\T^{n}, \U(1))$ obtained as the image of  the \emph{integral} symmetric bilinear form $I:= J + J^{tr}$ under the (usual) Chern-Weil homomorphism
\begin{equation*}
\Sym(\R^{n},\R)^{\mathrm{ad}} \to  H^4(B\T^{n},\Z) \cong H^3(\T^{n}, \U(1))\text{.}
\end{equation*}

 The corresponding crossed module of Lie algebras is
\begin{equation*}
\mathfrak{t}_J=(\R \stackrel {0}\to \R^{n} \stackrel{0}\to \Der(\R))\text{,}
\end{equation*}
and the four term-sequence is
\begin{equation*}
0 \to \R \stackrel\id\to \R \stackrel 0 \to \R^{n} \stackrel\id\to \R^{n} \to 0\text{.} 
\end{equation*}
Non-trivial are the map   $(\alpha_{a})_{*}=\id_{\R}$ (the differential of the action with a fixed group element $a\in \R^{n}$) and the map  $(\tilde \alpha_{(s, m)})_{*}(a) \mapsto ([J(a, m)], 0) = (\iota_* J)(a, m)$ of \cref{alpha-tilde}.
Here $\iota : \R \to \T \times \Z^n$ is the map $\iota(a) :=([a], 0)$.

There is exactly one section, $s=\id_{\R^{n}}$,  and one (half) splitting, $u=0$. Thus, all categorical tori have trivial Kassel-Loday class, 
\begin{equation*}
\mathrm{KL}(\mathfrak{t}_J)=0\text{.}
\end{equation*}
Thus, by \cref{theorem-2}, $\mathfrak{t}_J$ admits infinitesimal adjustments. Moreover, both $\Adj(\mathfrak{t}_J)$ and $\Adj^s(\mathfrak{t}_J)$ are affine spaces over the vector space $\mathrm{Bil}(\R^{n}, \R)$ of bilinear forms on $\R^{n}$; hence, we have
\begin{equation*}
\Adj^{s}(\mathfrak{t}_J) =\Adj(\mathfrak{t}_J)\text{.}
\end{equation*}
Indeed, it can be seen explicitly that the adaptedness condition is trivially satisfied for all (infinitesimal) adjustments. In fact, $\eta=0$ is an infinitesimal adjustment, and hence we even have
\begin{equation*}
\Adj(\mathfrak{t}_J) = \mathrm{Bil}(\R^{n}, \R)\text{.}
\end{equation*}
In particular, the bilinear form $J$ is an infinitesimal adjustment.

The adjusted Kassel-Loday class of $\eta \in \Adj(\mathfrak{t}_J)$ is its symmetrization, 
\begin{equation*}
\KL^{\mathrm{adj}}(\mathfrak{t}_J, \eta) =-\eta^{\sym}:= -\frac{1}{2}(\eta + \eta^{tr}) \in \mathrm{Sym}(\R^{n}, \R)^{\ad}=\mathrm{Sym}(\R^{n}, \R)\text{.}
\end{equation*}
We remark that the Chern-Weil homomorphism is the zero map. The fibre over a fixed $B\in \mathrm{Sym}(\R^{n}, \R)$ is an affine space over $\Alt^2(\R^{n}, \R)$.

The triviality of the adaptedness condition also shows $\Adj^s(\EuScript{T}_J)=\Adj(\EuScript{T}_J)$, so that we only have to discuss the non-adapted situation. 
Since $\R^{n}$ is connected, we have an injective map
\begin{equation*}
\Adj(\EuScript{T}_J)\hookrightarrow \Adj(\mathfrak{t}_J)
\end{equation*} 
by \cref{theorem-1}.
Since $\R^{n}$ is also simply-connected, the only obstruction to integrating an infinitesimal adjustment lies in the non-connectedness of $\mathbb{T} \times \Z^{n}$, and thus in the fulfillment of the condition \eqref{adjustment-condition-2}. We note that the adjustment obtained by integration of an infinitesimal adjustment $\eta \in \mathrm{Bil}(\R^{n}, \R)$ is given by the same formula again. \cref{adjustment-condition-2} reads
\begin{equation*}
\kappa\big(t(s, m), a\big) = (\tilde{\alpha}_{-(s, m)})_* (a)= (\iota_* J)(a, -m)=-(\iota_*J^{tr})(m, a)\text{,}
\end{equation*}    
which fixes $\kappa = -\iota_* J^{tr}$ as the only possible integrated adjustment. 
Thus,
\begin{equation*}
\Adj(\EuScript{T}_J) = \{-\iota_* J^{tr}\}\text{.}
\end{equation*}
The adjusted Kassel-Loday class of the integrating adjustment is
\begin{equation*}
\KL^{\mathrm{adj}}( \mathfrak{t}_J, -\iota_*J^{tr}) := \frac{1}{2}(J + J^{tr}) \in \mathrm{Sym}(\R^{n}, \R)^{\ad}\text{.}
\end{equation*}

Finally, let us look at the groupoid approach to adjustments. 
By \cref{theorem-3}, the differentiation functor $\ADJ(\EuScript{T}_J) \to \ADJ(\mathfrak{t}_J)$ is full and faithful, and we have 
\begin{equation*}
\pi_0\ADJ(\mathfrak{t}_J) = \mathrm{Bil}(\R^{n}, \R)
\quad\text{ and }\quad
\pi_1\ADJ(\mathfrak{t}_J) =(\R^{n})^{\vee}\text{.}  
\end{equation*}
By the above discussion $\ADJ(\EuScript{T}_J)$ is a groupoid with a single object, $(\id,-\iota_*J^{tr})$, and hence 
\begin{equation*}
\ADJ(\EuScript{T}_J) \cong B(\R^{n})^{\vee}\text{.}
\end{equation*}

\subsection{Automorphism 2-groups of algebras}
\label{ex:AUTA}

We consider a unital, associative, finite-dimensional algebra $A$ over $k=\R, \C$, and its \emph{automorphism 2-group}, represented by the crossed module
\begin{equation*}
\AUT(A) = (A^{\times} \stackrel t\to  \Aut(A) \stackrel \id\to \Aut(A))\text{, }
\end{equation*}
where $A^{\times}$ is the group of units of $A$, $t(u)$ is the inner automorphism corresponding to a unit $u$, and $\Aut(A)$ is the group of automorphisms of $A$.
We have $\pi_0\AUT(A)=\Out(A)$, the outer automorphism group, and $\pi_1\AUT(A)=Z(A)^{\times}$, the group of central units.
As all groups involved are finite-dimensional Lie groups, $\AUT(A)$ is always smoothly separable in this case.
To be central, $\Out(A)$ must to act trivially on $Z(A)^{\times}$; this is the case, for instance, when $A$ is a central algebra (so that $Z(A)^{\times}=k^{\times}$).

The induced crossed module of Lie algebras is
\begin{equation*}
\mathfrak{aut}(A) = (A \stackrel {t_{*}}\to  \Der(A) \stackrel \id\to \Der(A))\text{, }
\end{equation*}
where $t_{*}$ is the assignment of inner derivations, i.e., $t_{*}(a)(b):=ab-ba$. We remark that the differential of the map $\tilde\alpha_u$ of \eqref{alpha-tilde}, for $u\in A^{\times}$, is 
\[
(\tilde\alpha_u)_{*}(\delta) = u^{-1}\delta(u).
\]  
We have  $\pi_0(\mathfrak{aut}(A))=\mathfrak{der}(A)/A$ and $\pi_1(\mathfrak{aut}(A))=Z(A)$, and the four-term exact sequence is
\begin{equation*}
0 \to Z(A) \to A \to \mathfrak{der}(A)\to \mathfrak{der}(A)/A \to 0\text{.}
\end{equation*}

There is not much we can say in generality here, and so we proceed with assuming that $A$ is central and simple. 
Then, for $k=\C$ we have $A=\C^{n \times n}$ and for $k=\R$ we have either $A=\R^{n \times n}$ or $A=\mathbb{H}^{n \times n}$. By the Skolem-Noether theorem, we have  $\Out(A)=1$ in all cases, so that   $\pi_0(\mathfrak{aut}(A))=0$, and the four-term-sequence is
\begin{equation*}
0 \to k \to A \to \mathfrak{der}(A)\to 0 \to 0\text{.}
\end{equation*}      
There is a unique section, $s=0$, and a splitting $u$ is the same as a linear map $j: A \to k$ such that $j(a)\cdot 1=a$ for all $a\in A$. All relevant Lie algebra cohomology groups
are zero, as well as the Kassel-Loday class and the Chern-Weil homomorphism. This shows, by \cref{theorem-2}, that
\begin{equation*}
\Adj(\mathfrak{aut}(A))=\Adj^{s}(\mathfrak{aut}(A))=\{\eta\}\text{,}
\end{equation*}
i.e., there is a unique infinitesimal adjustment $\eta$, which is adapted to $s$. Since  $t_{*}:A \to \Der(A)$ is surjective, the unique infinitesimal adjustment $\eta$ is determined by \cref{EtaIdentities-2}, which says
\begin{equation*}
\eta(t_{*}(a),t_{*}(b))=\alpha_{*}(t_{*}(a),b)=t_{*}(a)(b)=ab-ba\text{,}
\end{equation*} 
for all $a, b\in A$.
One can check that this formula indeed defines an infinitesimal adjustment.
\begin{comment}
Antisymmetry shows \cref{EtaIdentities-1}, and together with the Jacobi identity of the commutator also \cref{EtaIdentities-0}. Adaptedness is 
\begin{align*}
t_* \eta(t_{*}(a), t_{*}(b))(c)&=(ab-ba)c-c(ab-ba) 
\\&= a(bc-cb)-(bc-cb)a-b(ac-ca)+(ac-ca)b
\\&=t_{*}(a)(bc-cb)-t_{*}(b)(ac-ca)
\\&= [t_{*}(a), t_{*}(b)](c)\text{.}
\end{align*}
It remains to check well-definedness: if $a'=a+x1$ and $b'=b+y1$ then 
\begin{equation*}
a'b'-b'a'=(a+x1)(b+y1)-(b+y1)(a+x1)=ab-ba\text{.}
\end{equation*}
\end{comment}
One can also check that $\eta$ integrates to an adjustment $\kappa\in \Adj^{s}(\AUT(A))$, determined by
\begin{equation*}
\kappa(t(a), t_{*}(b))= a t_*(b)(a^{-1}) = aba^{-1}-b\text{.}
\end{equation*}
\begin{comment}
Note that for $t(u) \in \mathrm{Aut}(A)$, then $(\alpha_{t(u)})_{*} (a)=t(u)(a)=uau^{-1}$; this shows \cref{adjustment-condition-3}. Moreover,  $(\tilde\alpha_u)_{*}(t_{*}(a))= u^{-1}(au-ua)=u^{-1}au-a$, and this shows \cref{adjustment-condition-2}. We check \cref{adjustment-condition-1}:
\begin{align*}
\kappa(t(u_1)t(u_2), t_{*}(b)) &= \kappa(t(u_1u_2), t_{*}(b))
\\&=u_1u_2bu_2^{-1}u_1^{-1}-b
\\&=u_1(u_2 b u_2^{-1})u_1^{-1}-u_2bu_2^{-1}+u_2bu_2^{-1}-b
\\&= \kappa\big(t(u_1), \Ad_{t(u_2)}(t_{*}(b))\big) + \kappa(t(u_2), t_{*}(b)) 
\end{align*}
Adaptedness is:
\begin{align*}
t_* \kappa(t(u), t_{*}(b)) &=t_{*}(ubu^{-1}-b)
\\&= t_{*}(ubu^{-1})-t_{*}(b)
\\&= \Ad_{t(u)}(t_{*}(b)) - t_{*}(b)
\end{align*}
Well-definedness is easy to see.
\end{comment}
More can be said separately in each case:
\begin{itemize}
\item
For $k=\C$,  we have $\Aut(A)=\mathrm{PGL}_n(\C)$, which is connected, so that the map $\Adj(\AUT(A)) \to \Adj(\mathfrak{aut}(A))$ is injective, by \cref{theorem-1}. This shows that 
\begin{equation*}
\Adj(\AUT(A))=\Adj^{s}(\AUT(A))=\{\kappa\}\text{.}
\end{equation*}

\item
For $k=\R$ and $A=\R^{n\times n}$, we have $\Aut(A)=\mathrm{PGL}_n(\R)$ which is not connected, and so there could be more adjustments on $\AUT(A)$ than $\kappa$.

\item
In the case $k=\R$ and $A=\mathbb{H}^{n\times n}$ is similar: Here $\Aut(A) = \mathrm{PGL}_n(\mathbb{H}) = \mathrm{GL}_n(\mathbb{H})/\R^\times$, which is connected.
This shows that also in this case, there is a unique adjustment on $\AUT(A)$. 
\end{itemize}

\appendix

\section{Butterflies}

\label{butterflies}

Crossed modules of Lie groups and Lie algebras form bicategories, whose 1-morphisms are called \quot{butterflies}; see \cite{Noohi2005, Noohi2007, Aldrovandi2009}.

\begin{definition}
Let $\Gamma_1=(H_1 \stackrel{t_1}{\to} G_1 \stackrel{\alpha_1}{\to} \Aut(H_1))$ and  $\Gamma_2=(H_2 \stackrel{t_2}{\to} G_2 \stackrel{\alpha_2}{\to} \Aut(H_2))$  be    crossed modules of Lie groups. A butterfly consists of a Lie group $K$ together with Lie group homomorphisms  that make up a commutative diagram 
\begin{equation}
\label{eq:butterfly}
\begin{aligned}
\xymatrix{H_1 \ar[dd]_{t_1} \ar[dr]^{i_1} && H_2 \ar[dd]^{t_2}\ar[dl]_{i_2} \\ & K \ar[dl]^{r_1}\ar[dr]_{r_2} \\ G_1 && G_2\text{, }}
\end{aligned}
\end{equation}
such that both diagonal sequences are complexes, the NE-SW  sequence is a short exact sequence of Lie groups, and the equations
\begin{equation} \label{RelationsButterflyLieGroups}
i_1(\alpha_1(r_1(x), h_1)) = xi_1(h_1)x^{-1}
\quad\text{ and }\quad
i_2(\alpha_2(r_2(x), h_2)) = xi_2(h_2)x^{-1}
\end{equation} 
hold for all $h_1\in H_1$, $h_2\in H_2$ and $x\in K$. 
\end{definition}

A morphism between two butterflies
\begin{equation*}
\begin{aligned}
\xymatrix{H_1 \ar[dd]_{t_1} \ar[dr]^{i_1} && H_2 \ar[dd]^{t_2}\ar[dl]_{i_2} \\ & K \ar[dl]^{r_1}\ar[dr]_{r_2} \\ G_1 && G_2\text{, }}
\quad\text{ and }\quad
\xymatrix{H_1 \ar[dd]_{t_1} \ar[dr]^{i_1'} && H_2 \ar[dd]^{t_2}\ar[dl]_{i_2'} \\ & K' \ar[dl]^{r_1'}\ar[dr]_{r_2'} \\ G_1 && G_2}
\end{aligned}
\end{equation*}
is a group homomorphism $k: K \to K'$ that commutes with all other maps in the obvious way.  
Since $k$ is, in particular, a morphism between Lie group extensions, it is automatically invertible.
Butterflies between two crossed modules form a groupoid $\BUT(\Gamma_1, \Gamma_2)$.

The identity butterfly of a crossed module  $\Gamma=(H \stackrel{t}{\to} G \stackrel{\alpha}{\to} \Aut(H))$ is given by $K:= H \rtimes_\alpha G$, with
\begin{align*}
i_1(h) &:= (h^{-1}, t(h) ) & i_2(h) &:= ( h, 1)
\\
r_1(h, g) &:=g & r_2(h, g) &:= t(h)g\text{.} 
\end{align*}
%
\begin{comment} 
The semi-direct product group structure on $K=H \rtimes_\alpha G$ has to be taken as
\begin{equation*}
(h, g)\cdot (h', g') := (h \alpha(g, h'), gg')\text{.}
\end{equation*}
Since $i_1$ does not look like a group homomorphism, let's check it:
\begin{align*}
i_1(hh') &= ((h')^{-1} h^{-1}, t(h)t(h')) 
\\
&= \bigl(h^{-1} h (h')^{-1} h^{-1}, t(h)t(h') \bigr)
\\
&= \bigl(h^{-1} \alpha(t(h), (h')^{-1}), t(h)t(h') \bigr)       
\\
&= (h^{-1}, t(h)) \cdot ((h')^{-1}, t(h')).
\end{align*}
Also 
\begin{align*}
r_2((h, g)\cdot (h', g') )
&= r_2(h\alpha(g, h'), gg')
\\
&= t(h\alpha(g, h'))gg'
\\
&= t(h) t(\alpha(g, h'))gg'
\\
&= t(h) g t(h')g^{-1} gg'
\\
&= t(h) g t(h') g'
\\
&= r_2(h, g) \cdot r_2(h', g').
\end{align*}
\end{comment}
The composition of butterflies 
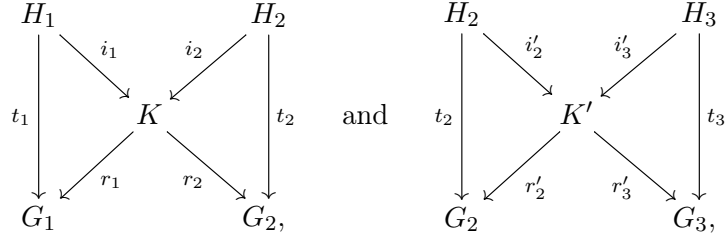
\begin{equation*}
\begin{aligned}
\xymatrix{H_1 \ar[dd]_{t_1} \ar[dr]^{i_1} && H_2 \ar[dd]^{t_2}\ar[dl]_{i_2} \\ & K \ar[dl]^{r_1}\ar[dr]_{r_2} \\ G_1 && G_2\text{, }}
\end{aligned}
\quad\text{ and }\quad
\begin{aligned}
\xymatrix{H_2 \ar[dd]_{t_2} \ar[dr]^{i_2'} && H_3 \ar[dd]^{t_3}\ar[dl]_{i'_3} \\ & K' \ar[dl]^{r_2'}\ar[dr]_{r'_3} \\ G_2 && G_3\text{, }}
\end{aligned}
\end{equation*}
is given by the Lie group
\begin{equation*}
\tilde K:=(K \times_{G_2} K')/\tilde i(H_2)\text{, }
\end{equation*}
where $\tilde i:H_2 \to \tilde K$ is given by $\tilde i(h_2):=(i_2(h_2), i_2'(h_2))$, which is a normal subgroup embedding, and the following maps:
\begin{align*}
&\tilde i_1: H_1 \to \tilde K & \tilde i_1(h_1) &:= [i_1(h_1), 1]
\\
&\tilde i_3: H_3 \to \tilde K & \tilde i_3(h_3) &:= [1, i_3'(h_3)]
\\
&\tilde r_1: \tilde K \to  G_1 & \tilde r_1([k, k']) &:= r_1(k)
\\
&\tilde r_3: \tilde K \to G_3 & \tilde r_3([k, k']) &:= r_3'(k')\text{.}
\end{align*}
Crossed modules of Lie groups form a bicategory $\CM$, with $\mathrm{Hom}_{\CM}(\Gamma_1, \Gamma_2) = \BUT(\Gamma_1, \Gamma_2)$ \cite{Aldrovandi2009}. Within this bicategory, 
a butterfly is invertible  if and only if its NW-SE sequence is also exact, in which case an inverse butterfly is obtained by vertical reflection of the butterfly \cite{Aldrovandi2009}. If $K: \Gamma_1 \to \Gamma_2$ is an invertible butterfly, then the 2-isomorphisms $K^{-1}\circ K \Rightarrow \id_{\Gamma_1}$ and $K \circ K^{-1}\Rightarrow \id_{\Gamma_2}$ are given by
\begin{equation*}
[k, k'] \mapsto (r_1(k), i_1^{-1}(k^{-1}k'^{}))
\quad\text{ and }\quad
[k, k'] \mapsto (r_2(k), i_2^{-1}(k^{-1}k'^{}))\text{, }
\end{equation*}
respectively. 
\begin{comment}
The middle groups of $B^{-1}\circ B$ and $B \circ B^{-1}$ are, 
\begin{equation*}
(K \times_{G_2} K)/\Delta(H_2)
\quad\text{ and }\quad
(K \times_{G_1} K)/\Delta(H_1)
\end{equation*}
respectively. 
\end{comment}

\begin{remark}
On the level of the homotopy groups $A_i := \mathrm{ker}(t_i) \subset H_i$ and $F_i :=G_i/t_i(H_i)$, a butterfly $K: \Gamma_1 \to \Gamma_2$ induces group homomorphisms $\pi_0 K: F_1 \to F_2$ and $\pi_1 K: A_1 \to A_2$ such that
\begin{equation}
\label{maps-on-homotopy-groupsLieGroups}
\pi_0K([r_1(k)])= [r_2(k)]
\quad\text{ and }\quad
i_2(\pi_1K(a))=i_1(a)^{-1} 
\end{equation}
hold for all $k\in K$ and $a\in A_1$.
\begin{comment}
The inversion in the second equation looks strange, but it has its justification. 
\end{comment}
\end{remark}

Applying the Lie functor to a butterfly $K$ between crossed modules of Lie groups yields a butterfly $\mathfrak{k}$ of crossed modules of Lie algebras
\begin{equation}
\label{LieAlgebraButterfly}
\begin{aligned}
\xymatrix{\mathfrak{h}_1 \ar[dd]_{t_1} \ar[dr]^{i_1} && \mathfrak{h}_2 \ar[dd]^{t_2}\ar[dl]_{i_2} \\ & \mathfrak{k} \ar[dl]^{r_1}\ar[dr]_{r_2} \\ \mathfrak{g}_1 && \mathfrak{g}_2\text{. }}
\end{aligned}
\end{equation}
Here, both diagonal sequences are complexes, the NE-SW  sequence is a short exact sequence of Lie algebras, and the equations
\begin{equation} 
\label{RelationsButterfly}
i_1(\alpha_1(r_1(X), y_1)) =     [X, i_1(y_1)]
\quad\text{ and }\quad
i_2(\alpha_2(r_2(X), y_2)) = [X, i_2(y_2)]
\end{equation} 
hold for all $y_1\in \mathfrak{h}_1$, $y_2\in \mathfrak{h}_2$ and $X\in \mathfrak{k}$.

\begin{comment}
The identity butterfly of a crossed module  $\mathfrak{G}=(\mathfrak{h} \stackrel{t}{\to} \mathfrak{g} \stackrel{\alpha}{\to} \Der(\mathfrak{h}))$ is given by $\mathfrak{k}:=\mathfrak{g} \ltimes_{\alpha} \mathfrak{h}$, with 
\begin{align*}
i_1(Y) &:= (t(Y), -Y) & i_2(Y) &:= (0, Y)
\\
r_1(X, Y) &:=X & r_2(X, Y) &:= t(Y)+X\text{.} 
\end{align*}
The semi-direct product group structure on $\mathfrak{k}=\mathfrak{g} \ltimes_{\alpha} \mathfrak{h}$ has to be taken as
\begin{equation*}
[(X, Y), (X', Y')] = ([X, X'], [Y, Y']+\alpha(X, Y')-\alpha(X', Y))\text{.}
\end{equation*}
This construction is a special case of \cref{constructing-butterflies} applied to $\phi=f=\id$ and $\lambda=0$.
\end{comment}

Butterflies between crossed modules of Lie algebras form another groupoid $\but(\mathfrak{G}_1, \mathfrak{G}_2)$, and there is a bicategory $\cm$ with $\mathrm{Hom}_{\cm}(\mathfrak{G}_1, \mathfrak{G}_2) = \but(\mathfrak{G}_1, \mathfrak{G}_2)$. 

On the level of the homotopy Lie algebras $\mathfrak{a}_i := \mathrm{ker}(t_i) \subset \mathfrak{h}_i$ and $\mathfrak{f}_i :=\mathfrak{g}_i/t_i(\mathfrak{h}_i)$, a butterfly $\mathfrak{k}\in \but(\mathfrak{G}_1, \mathfrak{G}_2)$ induces Lie algebra homomorphisms $\pi_0\mathfrak{k}: \mathfrak{f}_1 \to \mathfrak{f}_2$ and $\pi_1\mathfrak{k}: \mathfrak{a}_1 \to \mathfrak{a}_2$ such that
\begin{equation}
\label{maps-on-homotopy-groups}
\pi_0\mathfrak{k}([r_1(X)])= [r_2(X)]
\quad\text{ and }\quad
i_2(\pi_1\mathfrak{k}(y))=-i_1(y) 
\end{equation}
hold for all $X\in \mathfrak{k}$ and $y\in \mathfrak{a}_1$.
We provide the following result.
\begin{lemma}
\label{invertibility-of-butterflies}
A butterfly $\mathfrak{k}$ between crossed modules of Lie algebras is invertible if and only if $\pi_0 \mathfrak{k}$ and $\pi_1\mathfrak{k}$ are isomorphisms.
\end{lemma}

\begin{proof}
The only if-part follows from the functoriality of the constructions. 
Conversely, suppose   $\pi_0\mathfrak{k}$ and $\pi_1\mathfrak{k}$ are isomorphisms.
We need to show that the NW-SE sequence in \eqref{LieAlgebraButterfly} is exact.
To this end, let $X\in \mathfrak{k}$ with $r_2(X)=0$, in particular  $\pi_0\mathfrak{k}([r_1(X)])=[r_2(X)]=0$.
As $\pi_0\mathfrak{k}$ is invertible, we get $[r_1(X)]=0$, so there exists $y\in \mathfrak{h}_1$ such that $t_1(y)=r_1(X)$. 
Hence, $r_1(i_1(y)-X)=0$, so by exactness of the NE-SW sequence, there exists $y'\in \mathfrak{h}_2$ such that $i_2(Y')=i_1(y)-X$. 
We have $t_2(y')=r_2(i_2(y'))=r_2(i_1(y))-r_2(X)=0$ and hence $y'\in \mathfrak{a}_2$. 
Now consider $\tilde y := y+(\pi_1\mathfrak{k})^{-1}(y')$. 
From \eqref{maps-on-homotopy-groups}, we get satisfies $i_1(\tilde y)=i_1yY)-i_2(y')=X$. 
This shows that $\mathrm{ker}(r_2)\subset \mathrm{im}(i_1)$; hence, the NW-SE sequence is exact. 
\end{proof}

In the remainder of this appendix we provide a method to classify butterflies between crossed modules of Lie algebras by cocycle data.
We consider two crossed modules  $\smash{\mathfrak{G}_i=(\mathfrak{h}_i \stackrel{t_i}{\to} \mathfrak{g}_i \stackrel{\alpha_i}{\to} \Der(\mathfrak{h}_i))}$ of Lie algebras, $i=1,2$, and a butterfly $\mathfrak{k}:\mathfrak{G}_1\to \mathfrak{G}_2$ as in \cref{LieAlgebraButterfly}.  The main tool to extract cocycle data is a section
 $q: \mathfrak{g}_1 \to \mathfrak{k}$ of the short exact NE-SW sequence of $\mathfrak{k}$, i.e. a linear map such that $r_1q=\id_{\mathfrak{g}_2}$. 
 We recall that $q$ contains the same information as a linear map $j: \mathfrak{k} \to \mathfrak{h}_1$ with $ji_2=\id_{\mathfrak{k}}$; the relation between $j$ and $q$ is 
\begin{equation}
\label{Relationqj}
i_2  j+q  r_1=\id_{\mathfrak{k}}.
\end{equation}
We consider   $\lambda_q\in \Alt^2(\mathfrak{g}_1, \mathfrak{h}_2)$ defined by 
\begin{equation}
\label{definition-of-lambda}
\lambda_q(X, Y) := j([q(X), q(Y)])\text{.}
\end{equation}
Applying $i_2$ to this definition und using \eqref{Relationqj} yields
\begin{equation}
\label{i2ofLambda}
[q(X), q(Y)] -q([X, Y])=i_2(\lambda_q(X, Y))\text{, }
\end{equation}
and so provides an expression that captures the failure of $q$ to be a Lie algebra homomorphism.
%
\begin{comment}
Evaluating the definition of $\lambda_q$ on elements in the image of $r_1$ yields
\begin{equation}
%\label{expression-for-j}
 [j(X), j(Y)] +j([X, Y])= \lambda_q(r_1(X), r_1(Y))+\alpha_2(r_2(X), j(Y))-\alpha_2(r_2(Y), j(X))\text{.}
\end{equation}
It is also useful to consider $\psi_q: \mathfrak{g}_1 \to \Der(\mathfrak{h}_2)$ defined by $\psi_q(X)(Y) := [q(X), i_2(Y)]$.
Together with the map $\phi_q$ defined below, we have
\begin{align*}
\psi_q(X)(Y)=[q(X), i_2(Y)]=qr_1([q(X), i_2(Y)])+i_2j([q(X), i_2(Y)])=\alpha_2(\phi(X), j(i_2(Y)))=\alpha_2(\phi(X), Y)
\end{align*}
The pair $(\lambda_q, \psi_q)$ is the so-called \emph{Lie algebra factor system} that classifies the non-abelian Lie algebra extension $\mathfrak{h}_2 \to \mathfrak{k}\to \mathfrak{g}_1$ \cite{Hochschild1954}.
\end{comment}
%
We consider the related linear maps
\begin{equation}
\label{phiqfq}
\phi_q := r_2  q: \mathfrak{g}_1 \to \mathfrak{g}_2
\quad\text{ and }\quad
f_q :=- j  i_1: \mathfrak{h}_1 \to \mathfrak{h}_2
\end{equation}
satisfying $\pi_0\mathfrak{k}([X])=[\phi_q(X)]$ and $f_q|_{\mathfrak{a}_1}=\pi_1\mathfrak{k}$, where $ \pi_0\mathfrak{k}$ and $\pi_1\mathfrak{k}$ are the homomorphisms induced by the butterfly $\mathfrak{k}$ on the level of homotopy groups. 
\[
\hspace{5.5cm}
\begin{tikzcd}
\mathfrak{h}_1 \ar[dd, "t_1"'] \ar[dr, "i_1"', near start] 
\ar[rr, bend right=45, shift left=3.5, dashed, "f_q", shorten >=5pt, shorten <=5pt]&& \mathfrak{h}_2 \ar[dd, "t_2"] \ar[dl] \\ & 
\mathfrak{k} 
\ar[ur, dashed, bend right =25, "j"']
\ar[dl, "r_1", near start] \ar[dr, "r_2", near end]  
\\ 
\mathfrak{g}_1 
\ar[ur, dashed, bend left=25, "q"]
\ar[rr, shorten >=6pt, shorten <=6pt, bend left=45, dashed, shift right=4, "\phi_q"']
&& \mathfrak{g}_2\text{, }
\end{tikzcd}
\qquad 
\begin{aligned}
\text{(triangles involving dashed}\\
\text{arrows do not commute)~}
\end{aligned}   
\]
\begin{comment}
Indeed, the calculation
\begin{equation*}
\phi_q(t_1(Y))=r_2(q(t_1(Y)))=r_2(q(r_1(i_1(Y))))=r_2(i_1(Y)-i_2(j(i_1(Y)))=-t_2(j(i_1(Y)))
\end{equation*}
shows that $\phi_q$ induces a well-defined map $\mathfrak{f}_1 \to \mathfrak{f}_2$, and 
\begin{equation*}
\phi_q(r_1(X))=r_2(q(r_1(X)))=r_2(X-i_2(j(X)))=r_2(X)-t_2(j(X))
\end{equation*}
shows that this map satisfies \cref{maps-on-homotopy-groups}. Likewise, if $Y\in \mathfrak{a}_1$, then
\begin{equation*}
t_2(f_q(Y))=-r_2(i_2(j(i_1(Y))))=-r_2(i_1(Y)-q(r_1(i_1(Y))))=r_2(q(t_1(Y)))=0
\end{equation*}
shows that $f_q$ restricts to a map $\mathfrak{a}_1 \to \mathfrak{a}_2$, and
\begin{equation*}
i_2(f_q(Y))=-i_2(j(i_1(Y)))=-i_1(Y)+q(r_1(i_1(Y)))=-i_1(Y)
\end{equation*}
shows that this map satisfies \cref{maps-on-homotopy-groups}.
\end{comment}
Moreover, the diagram
\begin{equation}
\label{PhiFIntertinest}
\begin{tikzcd}
  \mathfrak{h}_1 
    \arrow[r, "f_q"] 
    \arrow[d, "t_1"'] 
  & \mathfrak{h}_2 
    \arrow[d, "t_2"] \\
  \mathfrak{g}_1 
    \arrow[r, "\phi_q"'] 
  & \mathfrak{g}_2
\end{tikzcd}
\end{equation}
commutes, and we have a rule for interchanging the actions of the two crossed modules:
\begin{align}
\label{exchange-of-actions}
\alpha_2(\phi_q(X), f_q(y))=  f_q(\alpha_1(X, y))+\lambda_q(X, t_1(y))\text{.}
\end{align}
In other words, $\lambda_q$ also encodes the failure of $(\phi_q, f_q)$ to intertwine the crossed module actions.
%
\begin{comment}
Indeed, 
\begin{equation*}
t_2(f_q(y))=-t_2(j(i_1(y)))=-r_2(i_2(j(i_1(y))))=r_2(q(t_1(y)))=\phi_q(t_1(y))\text{.}
\end{equation*}
The calculation for the second identity is
\begin{align*}
\alpha_2(\phi_q(X), f_q(y)) &=-ji_2 \alpha_2(r_2(q(X)), ji_1(y))
\\&=-j([q(X), i_2 ji_1(y)])
\\&=-j([q(X), i_1(y)-qt_1(y)])
\\&= -j([q(X), i_1(y)])+\lambda_q(X, t_1(y))
\\&= -ji_1(\alpha_1(r_1(q(X)), y))+\lambda_q(X, t_1(y))
\\&= -ji_1(\alpha_1(X, y))+\lambda_q(X, t_1(y))
\\&= f_q(\alpha_1(X, y))+\lambda_q(X, t_1(y))\text{.}
\end{align*}
We remark that putting $X=t(y')$ gives, together with \cref{PhiFIntertinest},
\begin{equation*}
[f_q(y'), f_q(y)]=  f_q([y', y])+\lambda_q(t(y'), t_1(y))\text{.}
\end{equation*}
\end{comment}
We remark that $\phi_q$ and $f_q$ are not a Lie algebra homomorphisms; for instance, applying $r_2$ to \eqref{i2ofLambda} we get 
\begin{equation}
\label{tofLambda}
\big[\phi_q(X), \phi_q(Y)\big] - \phi_q([X, Y]) = t_2\big(\lambda_q(X, Y)\big)\text{.}
\end{equation} 
Finally, we compute
\begin{align*}
\alpha_2\big(\phi_q(X), \lambda_q(Y, Z)\big) &= ji_2\big(\alpha_2(r_2(q(X)), \lambda_q(Y, Z))\big)
\\&=j\big([q(X), i_2(\lambda_q(Y, Z))]\big) & & \text{from \eqref{RelationsButterfly}}
\\&=j\big(\big[q(X), [q(Y), q(Z)]-q([Y, Z])\big]\big) & & \text{from \eqref{i2ofLambda}}
\\&= j\big(\big[q(X), [q(Y), q(Z)]\big]\big)-\lambda_q(X, [Y, Z]).
\end{align*}
Cyclically permuting the entries in the previous identity, we get
\begin{align}
&\hspace{-2em}\alpha_2\big(\phi_q(X), \lambda_q(Y, Z)\big)+\alpha_2\big(\phi_q(Y), \lambda_q(Z, X)\big)+\alpha_2\big(\phi_q(Z), \lambda_q(X, Y)\big)\nonumber
\\&= j\big(\big[q(Z), [q(X), q(Y)]\big]\big)+j\big(\big[q(Y), [q(Z), q(X)]\big]\big) + j\big(\big[q(X), [q(Y), q(Z)]\big]\big)\nonumber
\\&\quad-\lambda_q(Z, [X, Y])-\lambda_q(Y, [Z, X])-\lambda_q(X, [Y, Z])\nonumber
\\
&= - (\delta \lambda_q)(X, Y, Z)
\label{cyclic-identity}
\end{align}
where the second line vanishes because of the Jacobi identity. 

Wrapping up, we regard  triples $(\phi, f, \lambda)$ with linear maps $\phi: \mathfrak{g}_1 \to \mathfrak{g}_2$ and $f: \mathfrak{h}_1 \to \mathfrak{h}_2$, and $\lambda\in \Alt^2(\mathfrak{g}_1,\mathfrak{h}_2)$ satisfying  \cref{tofLambda,exchange-of-actions,PhiFIntertinest,cyclic-identity} as  \emph{cocycle data} for butterflies between $\mathfrak{G}_1$ and $\mathfrak{G}_2$.
If we change $q$ to $q':=q+i_2\gamma$ for  a linear map $\gamma: \mathfrak{g}_1 \to \mathfrak{h}_2$, then 
\begin{comment}
The corresponding retract $j$ is $j':= j-\gamma r_1$. 
\end{comment}
 the accordant change of the cocycle data is given by 
 \begin{equation}
 \label{equivalence-of-cocycle-data}
\begin{aligned}
\phi_{q'}(X) &=\phi_q(X)+t_2\gamma(X)
\\
f_{q'}(y) &=f_q(y)+\gamma t_1(y)
\\
\lambda_{q'}(X, Y)&=\lambda_q(X, Y)+\alpha_2(\phi_q(X), \gamma(Y))-\alpha_2(\phi_q(Y), \gamma(X))
\\
&\qquad +[\gamma(X), \gamma(Y)]
-\gamma ([X, Y]).
\end{aligned}
\end{equation}
%
\begin{comment}
Indeed,
\begin{align*}
\lambda_{q'}(X, Y) &:= j'([q'(X), q'(Y)])
\\&=(j-\gamma r_1)([(q+i_2\gamma)(X),(q+i_2\gamma)(Y)])
\\&=(j-\gamma r_1)([q(X),q(Y)]+[q(X),i_2\gamma(Y)]+[i_2\gamma(X),q(Y)]+[i_2\gamma(X),i_2\gamma(Y)])
\\&=j([q(X),q(Y)]+i_2(\alpha_2(\phi_q(X),\gamma(Y)))-i_2(\alpha_2(\phi_q(Y),\gamma(X)))+i_2([\gamma(X),\gamma(Y)]))
-\gamma ([X,Y])
\\&=\lambda_q(X,Y)+\alpha_2(\phi_q(X),\gamma(Y))-\alpha_2(\phi_q(Y),\gamma(X))+[\gamma(X),\gamma(Y)]
-\gamma ([X,Y])\text{.}
\end{align*}
\end{comment}
Thus, we consider cocycle data $(\phi, f, \lambda)$ and $(\phi', f', \lambda')$ \emph{equivalent} if there exists a linear map $\gamma: \mathfrak{g}_1 \to \mathfrak{h}_2$ such that the three relations \eqref{equivalence-of-cocycle-data} are satisfied. 
If $k: \mathfrak{k} \to \mathfrak{k}'$ is a morphism between butterflies, and $q$ is a section in $\mathfrak{k}$, then $kq$ is a section in $\mathfrak{k}'$, both producing the same cocycle data. Thus, we obtain a well-defined map
\begin{equation}
\label{extraction-of-local-data}
\pi_0 \but(\mathfrak{G}_1, \mathfrak{G}_2) \to \left \lbrace \text{equivalence classes of cocycle data} \right \rbrace\text{.}
\end{equation} 
We prove below that this map is a bijection, and start with constructing an inverse map, a \quot{reconstruction} of butterflies from cocycle data. Let $(\phi, f, \lambda)$ be cocycle data. 
Then, the formula
\begin{equation*}
[(x, X), (y, Y)] := ([x, y]+\alpha_2(\phi(X), y)-\alpha_2(\phi(Y), x)+\lambda(X, Y), [X, Y])
\end{equation*}
defines a Lie algebra structure on $\mathfrak{k} := \mathfrak{h}_2 \oplus \mathfrak{g}_1$.
Indeed, \cref{tofLambda,cyclic-identity} ensure that the pair $(\psi, \lambda)$, where $\psi: \mathfrak{g}_1 \to \Der(\mathfrak{h}_2)$ is defined by $\psi(X)(y) := \alpha_2(\phi(X), y)$, is a  \emph{Lie algebra factor system}, from which it is known that it defines a non-abelian Lie algebra extension $\mathfrak{h}_2 \to \mathfrak{k} \to \mathfrak{g}_1$ in the specified way \cite{Hochschild1954}. 
The maps
\begin{align*}
i_1&: \mathfrak{h}_1 \to \mathfrak{k}, &&& i_1(y) &:= (-f(y), t_1(y))
\\
i_2&: \mathfrak{h}_2 \to \mathfrak{k}, &&& i_2(y) &:= (y, 0)
\\
r_1&: \mathfrak{k} \to \mathfrak{g}_1, &&& r_1(X, y) &:= X 
\\
r_2&: \mathfrak{k} \to \mathfrak{g}_2, &&& r_2(X, y) &:= \phi(X)+t_2(y) 
\end{align*}
complete $\mathfrak{k}$ to a butterfly $\mathfrak{k}: \mathfrak{G}_1 \to \mathfrak{G}_2$.
Indeed, the NW-SE sequence is a complex because of \cref{PhiFIntertinest}, and the wings of the butterfly commute obviously.
The relations \eqref{RelationsButterfly} can be proved easily using \eqref{PhiFIntertinest} and \eqref{exchange-of-actions}.
\begin{comment}
Indeed, 
\begin{align*}
i_1(\alpha_1(r_1(X, y_2), y_1)) 
&=i_1(\alpha_1(X, y_1))
\\&= (t_1(\alpha_1(X, y_1)), -f(\alpha_1(X, y_1)))
\\&= ([X, t_1(y_1)], -f(\alpha_1(X, y_1)))
\\&= ([X, t_1(y_1)], -\alpha_2(\phi(X), f(y_1))+\lambda(X, t_1(y_1)))
\\&= ([X, t_1(y_1)], -[y_2, f(y_1)]-\alpha_2(\phi(X), f(y_1))-\alpha_2(t_2(f(y_1)), y_2)+\lambda(X, t_1(y_1)))
\\&= ([X, t_1(y_1)], -[y_2, f(y_1)]-\alpha_2(\phi(X), f(y_1))-\alpha_2(\phi(t_1(y_1)), y_2)+\lambda(X, t_1(y_1)))
\\&= [(X, y_2), (t_1(y_1), -f(y_1))]
\\&= [(X, y_2), i_1(y_1)]
\end{align*}
and
\begin{align*}
i_2(\alpha_2(r_2(X, y), y_2)) 
&= i_2(\alpha_2(\phi(X)+t_2(y), y_2))
\\&= (0, \alpha_2(\phi(X), y_2)+\alpha_2(t_2(y), y_2)) 
\\&= (0, [y, y_2]+\alpha_2(\phi(X), y_2))
\\&= [(X, y), (0, y_2)]
\\&= [(X, y), i_2(y_2)]
\end{align*}
\end{comment}

\begin{remark}
\label{strict-intertwiners}
Cocycle data of the form $(\phi, f, 0)$ is also known as a ,,strict intertwiner`` from $\mathfrak{G}_1$ to $\mathfrak{G}_2$. We note that  \cref{tofLambda} implies that $\phi$ is a  Lie algebra homomorphism, and \cref{exchange-of-actions,PhiFIntertinest} imply that $f$ is a Lie algebra homomorphism. Remaining are only the cocycle conditions \cref{PhiFIntertinest}, saying $t_2f=\phi t_1$, and \cref{exchange-of-actions}, which simplyfies to $\alpha_2(\phi_q(X), f_q(y))=  f_q(\alpha_1(X, y))$. Above construction of a butterfly from cocycle data shows, in this case, how strict intertwiners give rise to butterflies.    
\end{remark}

If we start with equivalence cocycle data $(\phi, f, \lambda)$ and $(\phi', f', \lambda')$, and the equivalence is expressed by a linear map $\gamma: \mathfrak{g}_1 \to \mathfrak{h}_2$, then it is straightforward to check that 
\begin{equation*}
\mathfrak{h}_2 \oplus \mathfrak{g}_1  \to  \mathfrak{h}_2 \oplus \mathfrak{g}_1; \quad (y, X) \mapsto (y-\gamma(X), X)
\end{equation*} 
establishes an isomorphism of Lie algebras and moreover extends to an isomorphism between the reconstructed butterflies. This shows that we have a well-defined map 
\begin{equation}
\label{reconstruction-from-cocycle-data}
\left \lbrace \text{equivalence classes of cocycle data} \right \rbrace \to \pi_0 \but(\mathfrak{G}_1, \mathfrak{G}_2)\text{.}
\end{equation}

\begin{lemma}
\label{constructing-butterflies}
The extraction of cocycle data via \cref{extraction-of-local-data}, and the reconstruction from cocycle data via \cref{reconstruction-from-cocycle-data} are inverse to each other, and establish a bijection
\begin{equation*}
\pi_0 \but(\mathfrak{G}_1, \mathfrak{G}_2) \cong \left \lbrace \text{equivalence classes of cocycle data } \right \rbrace\text{.}
\end{equation*}
\end{lemma}

\begin{proof}
The butterfly reconstructed from cocycle data $(\phi, f, \lambda)$ has a canonical section,  $q_0: \mathfrak{g}_1 \to \mathfrak{k}$, $q_0(X):=(0, X)$, and it is easy to see that the cocycle data obtained from this section is precisely the given one, $(\phi, f, \lambda)=(\phi_{q_0}, f_{q_0}, \lambda_{q_0})$. 

Conversely, suppose $\mathfrak{k}:\mathfrak{G}_1 \to \mathfrak{G}_2$ is a butterfly, $q$ is a section,  $(\phi_q, f_q, \lambda_q)$ is the corresponding cocycle data, and $\mathfrak{k}'$ is the butterfly reconstructed from $(\phi_q, f_q, \lambda_q)$ in the above way, then the map $(y, X) \mapsto q(X)+i_2(Y)$ establishes an isomorphism $\mathfrak{k}' \to \mathfrak{k}$ of butterflies. 
\end{proof}

\renewbibmacro*{in:}{%
  \ifboolexpr{
    test {\iffieldundef{journaltitle}} %
    and test {\iffieldundef{booktitle}} %
    and test {\iffieldundef{maintitle}} %
    and test {\iffieldundef{eventtitle}} %
  }
    {}% kein Container -> kein "In:"
    {\printtext{\bibstring{in}\intitlepunct}}% sonst wie üblich
}
\DeclareFieldFormat{eprint}{%
  \iffieldundef{url}
    {\texttt{\thefield{eprint}}}% kein URL -> nur Text
    {\href{\thefield{url}}{\texttt{[\thefield{eprint}]}}}% mit URL -> klickbarer eprint
}

\defbibenvironment{bibliography}
  {\list
     {\printfield[labelnumberwidth]{labelnumber}}%
     {\setlength{\labelwidth}{2em}%
      \setlength{\leftmargin}{3em}%
      \setlength{\labelsep}{1em}% Abstand zwischen Label und Eintrag
      \setlength{\itemindent}{0pt}%
      \setlength{\itemsep}{\bibitemsep}%
      \setlength{\parsep}{\bibparsep}}%
   \renewcommand*{\makelabel}[1]{##1}} % kein \hfill
  {\endlist}
  {\item}

\tocsection{\refname}
\raggedright
\printbibliography[heading=none]

\end{document}